\documentclass{amsart} 
\usepackage[parfill]{parskip}
\usepackage{graphicx}  
\usepackage{verbatim}
\usepackage{color} 
\usepackage{hyperref} 
\hypersetup{backref,pdfpagemode=FullScreen,colorlinks=true}
\usepackage{appendix, url}
\usepackage{amsmath, amssymb, mathrsfs, amsfonts,  amsthm} 
\usepackage{tikz-cd}

\DeclareMathOperator{\D}{\Delta}

\DeclareMathOperator{\sign}{sign}

\DeclareMathOperator{\degree}{degree}

\DeclareMathOperator{\interior}{interior}

\DeclareMathOperator{\lie}{Lie}

\DeclareMathOperator{\image}{image}

\newtheorem {theorem} {Theorem} [section] 
\newtheorem {meta-theorem} {Meta-Theorem} [section]

\newtheorem{lemma}[theorem] {Lemma} 

\newtheorem {question}  [theorem] {Question}
\newtheorem {definition} [theorem] {Definition} 
\newtheorem {proposition}  [theorem]{Proposition} 

\newtheorem {corollary}[theorem]  {Corollary} 
\newtheorem {example} [theorem]  {Example}
  
\newtheorem {notation}[theorem] {Notation}

\newtheorem {remark} [theorem] {Remark}
\numberwithin{equation}{section}
\DeclareMathOperator{\colim}{colim}
\DeclareMathOperator{\obj}{obj}
\begin{document}  
\today
\author {Yasha Savelyev} 
\address{University of Colima, faculty of science, 
Bernal Díaz del Castillo 340,
Col. Villas San Sebastian,
28045 Colima Colima, Mexico}
\email{yasha.savelyev@gmail.com}
\title[Smooth simplicial sets and universal
Chern-Weil]{Smooth simplicial sets and universal Chern-Weil
for infinite dimensional groups}
\begin{abstract}
We give the construction of the universal, natural up to
homotopy Chern-Weil differential graded algebra homomorphism:
$$cw: \mathcal{I} (G)  \to \Omega  ^{\bullet } (BG, \mathbb{R})$$
for infinite dimensional Milnor regular Lie groups $G$,
where $\Omega ^{\bullet}(BG, \mathbb{R})$ is
a certain de Rham algebra of $BG$ (Milnor $BG$ up to a natural weak
homotopy equivalence) and where $\mathcal{I}
(G)$ is the algebra of continuous, $Ad _{G}$ invariant,
symmetric multilinear functionals on
the Lie algebra. In particular, this applies to the group of
compactly generated Hamiltonian symplectomorphisms,
using which we verify a conjecture of Reznikov.
For the construction of $cw$ we introduce a basic geometric-categorical
notion of a smooth simplicial set. Loosely, this is to Chen
spaces as simplicial sets are to spaces. We then give a new
construction of the classifying space of $G$ as 
a smooth Kan complex, with the geometric realization weakly
equivalent to the Milnor $BG$.  
\end{abstract}
 \maketitle  
\tableofcontents
\section{Introduction} 
\label{sec:Introduction} First, we introduce the notion of a smooth 
simplicial set, which is most directly an analogue
in the world of simplicial sets of Chen spaces
~\cite{cite_ChenIteratedPathIntegrals}, and less directly
of diffeological spaces of Souriau
~\cite{cite_SouriauDiffeology}.  The Chen/diffeological
spaces are perhaps the most basic notions of a ``smooth space''. 
The language of smooth simplicial sets turns out to be a powerful
tool to resolve the long standing problem of the
construction of the universal Chern-Weil $dg$ algebra homomorphism for infinite dimensional Lie groups.   

\subsection{Generalized Lie groups} \label{sec_Generalized Lie groups}
By an infinite dimensional Lie group $G$ we will mean a Lie group whose
underlying infinite dimensional smooth manifold is 
modeled on a locally convex topological vector space.
For Chern-Weil theory we will need that $G$ is Milnor regular, see ~\cite[Definition
II.5.2]{cite_Neeb}. This condition may not ultimately be
necessary, see Remark \ref{rem_curvatureform}. 
As explained in
~\cite{cite_Neeb},  a basic example of a regular group is the
group $\operatorname {Diff} _{c} (M)$ of compactly supported
diffeomorphisms of a smooth finite dimensional
manifold. In this case regularity boils down to
the fact that a smooth family of smooth (compactly
supported) vector fields $\{X
_{t}\}$ integrates to a smooth family of self-diffeomorphisms (a.k.a. a smooth isotopy), and the
dependence of this on $\{X _{t}\}$ is itself smooth.
\footnote{As I understand, there are no known examples of
locally convex Lie groups that are not regular, see however
~\cite{cite_MagnotNotRegular} for a diffeological example.}
The group $\operatorname {Diff} _{c} (M)$ in addition has
the homotopy type of a CW complex, see
~\cite{cite_overflowCWstructureDiff}. This also applies to
other $LF$ (limit Frechet) generalized Lie groups like the
group of compactly generated Hamiltonian symplectomorphisms,
which forms an important example for us.

%

Since this definition also encompasses standard finite dimensional Lie groups, it
is convenient to give this a new working name: \textbf{\emph{generalized Lie
group}}. 

\subsection{Universal $dg$ Chern-Weil homomorphism} \label{sec_Universal $dg$ Chern-Weil}
One problem of topology is the construction
of a ``smooth structure'' on the Milnor classifying space
$BG$ of a generalized (real or complex) Lie group $G$.  
There are specific
requirements  for what such a notion of a smooth
structure should entail. At the very least we hope
to be able to carry out Chern-Weil theory universally
on $BG$. We now describe this.

In what follows, we keep to real generalized Lie groups $G$, but there is no essential difficulty to extending our theory to the complex case.
Denote by $\mathcal{I} (G)$ the
algebra over $\mathbb{R} ^{} $ of $Ad _{G}$ invariant,
continuous, symmetric multilinear functionals on the
Lie algebra $\mathfrak g$, see
Section \ref{sec_The algebraI}.

Denote by $\Omega  ^{\bullet} (BG,
\mathbb{R})$ the ``$dg$ algebra of differential forms on
$BG$'' (for the moment left unspecified) whose cohomology
is  $H ^{\bullet}(BG, \mathbb{R} ^{} )$. Then we want a ``purely'' differential geometric
construction of the Chern-Weil $dg$ algebra homomorphism
(natural up to homotopy of $dg$ maps, Definition \ref{def_dghomotopy}):
\begin{equation*}
   cw: \mathcal{I} (G)  \to \Omega  ^{\bullet} (BG,
   \mathbb{R}),
\end{equation*}
where the differential on the left is trivial (and the
grading is in even degrees).  
The goal is to set up all structures in such a way that the
differential geometry in this construction becomes trivial,
modulo the inherent differential geometry involved in
Chern-Weil forms.

For finite 
dimensional Lie groups,  the cohomological universal
Chern-Weil homomorphism:
\begin{equation} \label{eq_chernweilbasic}
hcw: \mathcal{I} (G)  \to H  ^{\bullet } (BG, \mathbb{R})
\end{equation}
has been studied for instance by Bott 
~\cite{cite_BottChernWeil}. 
It has also been directly constructed in
Dupont~\cite{cite_DupontCharClasses} using simplicial techniques,
also important for us here. However, to produce a (natural) $dg$
enhancement of this, some geometric theory is
ostensibly required. 
For example, working on the  $dg$ level is important  in
Freed-Hopkins ~\cite{cite_FreedHopkins}. However they work
on smooth classifying stacks (and with classical Lie groups $G$), rather then on
the Milnor spaces $BG$.


\subsubsection{Smooth structures on $BG$} \label{sec_Smooth structures on $BG$}
One candidate for a smooth structure on $BG$ is
some kind of 
diffeology. For example Magnot and
Watts~\cite{cite_MagnoWattsDiffeology}  construct a
natural diffeology on the Milnor classifying space
$BG$. Another approach to this is contained in 
Christensen-Wu~\cite{cite_ChristensenWusmoothclassifying}, where the authors also state their plan to develop 
some kind of universal Chern-Weil theory in the future.

A further specific possible requirement for the above
discussed ``smooth
structures'',  is that the
smooth singular simplicial set $BG _{\bullet } $ should have a geometric realization
weakly homotopy equivalent to $BG$. See for
instance ~\cite{cite_OhTanakaSmoothType} for one approach to this particular problem in the context of diffeologies. In the category of smooth simplicial sets we have a stronger
analogue of this, in the form 
of Proposition \ref{prop:smoothsimplices}. 
The latter and more specifically Theorem \ref{thm:classifyKan} is used in ~\cite{cite_SavelyevGlobalFukayaCategoryI} to prove universality of the global Fukaya category.

The structure of a smooth simplicial set
is initially more flexible than a space with
diffeology, but we may add further conditions, like the Kan
condition, which will be important for us. Given
a generalized Lie group $G$, for each choice of a particular
kind of Grothendieck universe $\mathcal{U} $ we construct a smooth simplicial set $BG
^{\mathcal{U}}$ with a specific classifying
property. We note that 
this is \emph{not}  the Milnor construction. But it will be
shown that $|BG ^{\mathcal{U} }|$ always has the weak
homotopy type of Milnor $BG$.  

The  simplicial set $BG 
^{\mathcal{U}}$ is moreover a Kan
complex, and so is a basic example of what we call a smooth Kan
complex. Our constructions, will work naturally on $BG
^{\mathcal{U} }$ rather than its geometric realization.
And all the desires of ``smoothness'' mentioned above then is
some sense hold true for $BG
^{\mathcal{U}}$ via its smooth Kan complex
structure. 
\subsubsection{Cohomological Chern-Weil homomorphism} \label{sec_Basic statement}
We first note that the cohomological Chern-Weil
homomorphism \ref{eq_chernweilbasic} has a direct extension
for generalized Lie groups. The base space $Y$ of ordinary
smooth $G$-bundles \footnote{A $G$-bundle is henceforth
a short name for: a principal $G$ bundle.} is throughout assumed to be a finite
dimensional smooth manifold, without corners unless
specified so. The following is proved in Section
\ref{sec_Universal cohomological Chern-Weil homomorphism}.
\begin{theorem} \label{thm:A}
Let $G$ be a generalized Lie group then there is an algebra homomorphism:
\begin{equation*}
   hcw: \mathcal{I} (G)  \to H  ^{\bullet } (BG, \mathbb{R}).
\end{equation*}
This has the following property. Suppose that $P \to Y$ is a smooth
$G$-bundle over a smooth manifold $Y$, and
   $$hcw ^{P}: \mathcal{I} (G)  \to H  
   ^{\bullet } (Y,
   \mathbb{R}) $$  is the associated Chern-Weil map. Then
	 $$hcw ^{P}
	 = f _{P} ^{*} \circ hcw $$ where $$f _{P}: Y \to
	 BG$$
   is the classifying map of $P$ and $$f _{P} ^{*}: H   ^{\bullet } (BG,
	 \mathbb{R} ^{} ) \to H  ^{\bullet} (Y, \mathbb{R}
	 ^{} )$$ the induced algebra map. 
\end{theorem}
\subsection{$dg$ Enhancement}
\label{sec_Enhancement to a $dg$ homomorphism}
In the infinite dimensional case, the $dg$ algebra
$\mathcal{I} (G)$ may not be freely
generated \footnote{This was
nicely pointed out to me by a referee, using Reznikov's
tensors (appearing just ahead) as an example.}. Therefore we cannot readily enhance
the map $hcw$ to a differential graded map purely formally.
Furthermore, even when a formal enhancement is possible 
it may not be adequate for more in depth geometric applications like the theory of
Cheeger-Simons characters,
~\cite{cite_CheegersSimonsCharacters}, which needs
suitably natural differential forms. 

We can at first avoid Grothendieck universes in the
statement, and in fact get a purely algebraic-topological result.
We will however need 
the notion of geometric homotopy of $dg$ maps \footnote {A
$dg$ map will be shorthand for
$dg$ algebra homomorphism.} as given
in Definition \ref{def_dghomotopy}, the latter is shown in
the Appendix \ref{appendix_Ainfhomotopy} to induce an $A _{\infty}$ homotopy of $dg$
maps. However, unless specified otherwise homotopy will mean
geometric homotopy throughout.
Let $A (BG)$ denote the Whitney-Sullivan commutative $dga$ of 
$BG$ over $\mathbb{R} ^{} $, as reviewed in Section
\ref{sec_Whitney-Sullivan de Rham algebra of a topological
space}. The following is also proved there:
\begin{theorem} \label{thm_APLMain}
Let $G$ be a generalized Lie group having the homotopy of
a CW complex (e.g. $G = \operatorname {Diff} _{c} (M)$). Then there is a lifting: 
\begin{equation*}
\begin{tikzcd}
& A (BG) \ar [d] \\
\mathcal{I} (G) \ar[r, "hcw"] \ar[ur, "cw"] & H ^{\bullet
} (BG, \mathbb{R} ^{} ),
\end{tikzcd}
\end{equation*}
where the arrow on the right is the cohomology projection,
and $cw$ is a $dg$ map that is natural up to
homotopy.  
\end{theorem}
\begin{remark} \label{rem_} As one ingredient the proof uses 
the Grothendieck's axiom of universes, as there are no universes in the
statement one may wonder if there is a pure ZFC proof of the
above.
\end{remark}

There is an ostensibly stronger formulation of the above.  More specifically, once
a ``smooth model''  for $BG$ as a smooth Kan complex $BG
^{\mathcal{U} }$ is
fixed, there is a canonical ``de Rham algebra'' $\Omega
^{\bullet } (BG ^{\mathcal{U}}, \mathbb{R} ^{} )$, as described in Section \ref{sec:Differential:forms}. 
We also need some additional ingredients. If $P \to Y$ is
a $\mathcal{U} $-small $G$-bundle, then there is
a certain classifying simplicial map $f _{P ^{\Delta^{} }}:
Y _{\bullet }  \to BG ^{\mathcal{U}}$, see Theorem
\ref{thm:classifyKan}, where $Y _{\bullet}$ denotes the
smooth singular set of $Y$. 
Also there is an obvious natural $dg$ map:
$$\Theta: \Omega ^{\bullet} (Y, \mathbb{R} ^{} ) \to \Omega
^{\bullet} (Y _{\bullet }, \mathbb{R} ^{} ).$$ The following is proved in
Section \ref{sec_Universal dg Chern-Weil homomorphism
proof}.
\begin{theorem} \label{thm_ChernWeildgIntro}
Let $G$ be a generalized Lie group and $\mathcal{U}$
a $G$-admissible Grothendieck universe. For each choice of a $G$-connection $D$ on the
universal $G$-bundle $EG ^{\mathcal{U} } \to BG
^{\mathcal{U}}$  there is
a natural $dg$ map: 
\begin{equation*}
   cw ^{D}: \mathcal{I} (G)  \to \Omega  ^{\bullet
	 } (BG ^{\mathcal{U}}, \mathbb{R}),
\end{equation*}
whose homotopy class is independent of the choice of $D$.
%
The map has the following property. Suppose that $P \to Y$ is a $\mathcal{U} $-small smooth
$G$-bundle, over a smooth manifold $Y$, and
   $$cw ^{P}: \mathcal{I} (G)  \to \Omega  
   ^{\bullet } (Y,
   \mathbb{R}) $$  is the associated standard Chern-Weil
	 $dg$ map
	 (natural up to homotopy). Then $$\Theta  \circ cw ^{P}
	 \simeq f _{P ^{\Delta^{} }} ^{*} \circ cw, \quad
	 (\text{homotopy relation}).  $$  
\end{theorem}
\subsection{An example: the group of Hamiltonian symplectomorphisms} \label{sec_An example}
Here is one concrete example, with more details in Sections
\ref{sec:Chern-WeiltheoryonBHam} and \ref{sec:couplingClass}.   Let $\mathcal{H} = 
\operatorname {Ham} (M, \omega)  $ denote the generalized
Lie group of compactly generated Hamiltonian
symplectomorphisms of some symplectic $2n$-manifold $(M,
\omega )$. Here $\phi$ is compactly generated means that
there is a smooth compactly supported $H: M \times
[0,1] \to \mathbb{R} ^{} $ s.t. $\phi $ is the time one map
of the Hamiltonian flow $\{X _{t}\} $, $\omega (X _{t}, \cdot) = dH _{t} $.

Let $\mathfrak 
h $ denote the Lie algebra of $\mathcal{H} $. When $M$ is compact $\mathfrak 
h$  is naturally isomorphic to the space of mean 0 (with
respect to $dvol _{\omega} = \omega ^{n}$) smooth functions on 
$M$; otherwise it is the space of all smooth 
compactly supported functions.  In 
~\cite{cite_ReznikovCharacteristicclassesinsymplectictopology} 
Reznikov defined $Ad _{\mathcal{H} }$ invariant, continuous,
symmetric multilinear functionals 
$\{r _{k} \} _{k \geq 1}$  on the Lie algebra 
$\mathfrak h$. These are defined by:
\begin{equation*}
(H _{1}, \ldots, H _{k}) \mapsto \int _{M} H _{1} \cdot
\ldots \cdot H _{k} \, \omega ^{n}.
\end{equation*}
Denote by $\mathcal{R}ez$ the sub-algebra
of $\mathcal{I} (\mathcal{H})$  generated by
$\{r _{k}\}$. 

By Chern-Weil theory we get cohomology classes $c
^{r _{k}} (P) \in H ^{2k} (X, \mathbb{R} ^{} ) $ for any
smooth $\mathcal{H} $-bundle $P$  over a smooth manifold $X$.
The following results are proved in Section
\ref{sec:Chern-WeiltheoryonBHam}.
Using Theorem \ref{thm_ChernWeildgIntro} we get: 
\begin{corollary}
\label{thm:RezExtends} 
There is a $dg$ map (natural up to
homotopy)
\begin{equation*}
   rez: \mathcal{R}ez  \to \Omega  ^{\bullet } (B
	 \mathcal{H} ^{\mathcal{U}},
   \mathbb{R}),
\end{equation*}
satisfying the restriction property as in Theorem \ref{thm_ChernWeildgIntro}. And in particular, 
there are universal Reznikov cohomology 
   classes $c ^{r _{k}}  \in H ^{2k} (B 
   \mathcal{H}, \mathbb{R} ^{} ) $, satisfying 
   the following.   Let $Z \to Y$ be a smooth principal
$\mathcal{H}$-bundle.  Let $c ^{r _{k}} (Z) \in H ^{2k} (Y) $
denote the Reznikov class. Then $$f _{Z}
^{*} c ^{r _{k}}  = c^{r _{k}}  (Z),
$$  where $f _{Z}: Y \to B \mathcal{H} $  is the
classifying map of the underlying topological
$\mathcal{H}$-bundle.
\end{corollary}
The second part of the corollary is an explicit form of a statement  asserted 
by Reznikov~\cite [page
12]{cite_ReznikovCharacteristicclassesinsymplectictopology}, on the extension of his classes to the universal 
level on $B \mathcal{H}$ \footnote{His assertion is left
without proof.}. The above corollary is actually
stronger, since we don't require compactness of $M$.

Likewise, we obtain a differential 
geometric proof that the Guillemin-Sternberg-Lerman 
coupling class $\mathfrak c (P) \in H ^{2} (P)  $  
~\cite{cite_GuilleminLermanEtAlSymplecticfibrationsandmultiplicitydiagrams}, 
~\cite{cite_McDuffSalamonIntroductiontosymplectictopology}  
  of a Hamiltonian fibration (Definition 
\ref{def:hamfibration})  has a universal representative. Specifically, let ${M} ^{\mathcal{H}} $ denote the $M$-fibration associated to the 
universal principal $\mathcal{H} $-fibration $\mathcal{E} 
\to B \mathcal{H}  $. (In other words the universal Hamiltonian 
$M$-bundle.) 
\begin{theorem}
\label{thm:couplingClass}  There is a cohomology 
class $\mathfrak c \in H ^{2} ({M} ^{\mathcal{H}} ) $ 
so that if $P \to X$ is a smooth Hamiltonian 
$M$-fibration and $\widetilde{f}: P \to {M} ^{\mathcal{H}}  $ 
the corresponding map then $\widetilde{f}  _{P} ^{*} \mathfrak c = \mathfrak c (P) $. 
\end{theorem}
For $M$ closed this is proved by
Kedra-McDuff~\cite[Proposition
3.1]{cite_KedraMcDuffHomotopypropertiesofHamiltoniangroupactions}
using homotopy theory techniques.

Here is a basic application. Let $Symp (\mathbb{CP} ^{k}) $ denote the group of
symplectomorphisms of $\mathbb{CP} ^{k} $, that is
diffeomorphisms $\phi: \mathbb{CP} ^{k}  \to
\mathbb{CP} ^{k} $ s.t. $\phi ^{*} \omega _{0} =
\omega _{0}$ for $\omega _{0}$ the Fubini-Study symplectic
2-form on $\mathbb{CP} ^{k} $.  
Using the above corollary, we may obtain an elementary proof of the following theorem
of Kedra-McDuff:
\begin{theorem} [Kedra-McDuff] \label{thm:B} 
Let  $$i: BPU (n)  \to BSymp (\mathbb{CP}
^{n-1}) $$ be the natural map. 
Then $$i _{*}: H _{\bullet}(BPU (n), \mathbb{R} ^{})   \to
H _{\bullet} (BSymp (\mathbb{CP} ^{n-1}), \mathbb{R} ^{} ) $$ is an injection 
for all $n \geq 2$. 
\end{theorem}
More history and background surrounding these
theorems is in Sections
\ref{sec:universalChernWeil} and \ref{sec:Chern-WeiltheoryonBHam}.

\subsection{Other examples} \label{sec_Other examples of generalized groups}
One other basic set of examples of generalized Lie groups,
with a wealth of invariant polynomials on the Lie 
algebra, are the loop groups. That is the groups $LG, \Omega
G$, where $G$ is any (finite dimensional) Lie 
group, $LG$ is the free loop space, and $\Omega G$ is the
based loop space at $id$. See for instance
~\cite{cite_RosenbergLoopSpacesChern} for related
computations.
Loop groups are prominent in conformal field theory, 
see for instance \cite{cite_PressleySegalLoopgroups} for
the foundation of the subject. The relevant Chern-Weil
theory then has physical connotations. Other examples of infinite
dimensional Chern-Weil theory include:
~\cite{cite_MagnotChernForms},
~\cite{cite_MickelssonChernWeil},
~\cite{cite_RosenbergChernWeil},
~\cite{cite_RosenbergChern2}.


There are various precedents in
giving a differential geometric definition
of the (infinite dimensional group) Chern-Weil homomorphism in some cases, for example Magnot~\cite{cite_MagnotNonLinearGrassmanians}. 


\subsection {Acknowledgements}  I am grateful
to Daniel Freed, Yael Karshon, Dennis Sullivan, 
Egor Shelukhin, and Jean-Pierre Magnot, for comments 
and questions. I am also grateful to Dusa McDuff 
for explaining the proof of Corollary 1.6 in 
~\cite{cite_KedraMcDuffHomotopypropertiesofHamiltoniangroupactions}.

\section {Preliminaries and notation}
\label{sec:preliminaries}
 We denote by $\D$ the simplex category:
\begin{itemize}
	\item The set of objects of $\Delta^{} $ is $\mathbb{N}$.
	\item $\hom _{\D}
 (n, m) $ is the set of non-decreasing maps $[n] \to [m]$, where $[n]= \{0,1, \ldots, n\}$, with its natural order.
\end{itemize} 
  A simplicial set $X$ is a functor                                 
\begin{equation*}
   X: \D ^{op}  \to Set.
\end{equation*} 
The set $X (n)$ is called the set of $n$-simplices of $X$.
Given a collection of sets $\{X
(n) \} _{{n \in \mathbb{N}}} $, by a \textbf{\emph{simplicial structure}}
we will mean the extension of this data to a
functor: $X: \D^{op} \to Set $. 

$\D ^{d} _{simp}  $ will denote a particular
simplicial set: the standard representable
$d$-simplex, with $$\D ^{d} _{simp} (n) = hom _{\D} (n, d).                                         $$ 
A morphism or a map of simplicial sets, or a \textbf{\emph{simplicial map}}  $f: X \to Y$ is a
natural transformation $f$ of the corresponding
functors.  The category of simplicial sets will be
denoted by $s-Set$. 

By a $d$-simplex $\Sigma$ of a simplicial set $X$,
we may mean, interchangeably, either the element
in $X (d)$ or the map of simplicial sets:
$$\Sigma: \Delta ^{d} _{simp} \to X,  $$ uniquely corresponding to $\Sigma$ via the Yoneda lemma.  
If we write $\Sigma ^{d}$ for a simplex of $X$, it is implied that it is a $d$-simplex.  

With the above identification if $f: X \to Y$ is a map of
simplicial sets then 
\begin{equation} \label{eq_identification}
f (\Sigma ) = f \circ \Sigma. 
\end{equation}

\subsection {Topological simplices and smooth
singular simplicial sets} 
Let $\D ^{d} $ be the topological $d$-simplex,
i.e.  $$\Delta^{d} := \{(x _{1}, \ldots, x _{d})
\in \mathbb{R} ^{d} \,|\, x _{1} + \ldots + x _{d} \leq 1,
\text{ and } \forall i:  x _{i} \geq 0  \}. $$  
The vertices of $\Delta^{d} $ will be assumed ordered in
the standard way.
\begin{definition} \label{def:collared}  Let $X$
be a smooth manifold with corners, in the diffeological
sense ~\cite{cite_IglesiasDiffForms}. 
We say that a map $\sigma: \Delta^{n}  \to X$  is
smooth if it is smooth as a map of
manifolds with corners.  In particular, $\sigma: \Delta ^{n} \to \Delta^{d} $ is smooth iff it has an extension 
to a smooth map $V \supset \Delta^{n} \to \mathbb{R} ^{d}
$, with $V$ open. (See ~\cite[Theorem
4.1]{cite_KarshonDomainMaps}).
\end{definition}

We denote by $\D ^{d} _{\bullet}  $ the
smooth singular simplicial set of $\D
^{d} $, i.e.  $\Delta^{d} _{\bullet } (k) $ is
the set of smooth maps $$\sigma: \D ^{k} \to \D ^{d}.  $$   

We call an affine map $\D ^{k} \to \D ^{d}$ taking vertices
to vertices, in an order preserving way,
\textbf{\emph{simplicial}}.  
And we denote by $$\D ^{d} _{simp} \subset \D ^{d}
_{\bullet} $$ the sub-simplicial set consisting
of these topological simplicial maps.  That is  $\D ^{d} _{simp} (k)
$ is the set of simplicial maps $\Delta^{k} \to
\Delta^{d} $.  

Note that $\D ^{d} _{simp}$  is naturally
isomorphic to the standard representable
$d$-simplex $\D ^{d} _{simp}  $ as previously
defined, so that this abuse of notation should not cause
issues. Thus we may also understand $\D$ as the category with objects topological simplices $\D ^{d} $, $d \geq 0$ and morphisms simplicial maps. 

\begin{notation} \label{notation:1}
A morphism $m \in hom _{\Delta} (n, k) $ uniquely
corresponds to a simplicial map $\D ^{n} _{simp} \to \D ^{k}
_{simp}$, which uniquely corresponds to a topological simplicial map $\D
^{n} \to \D ^{k}$ (as defined right above). The
correspondence is by taking the maps $\D ^{n} _{simp} \to \D
^{k} _{simp}$, $\D ^{n} \to \D ^{k}$, to be determined by
the set map $m: \{0, \ldots, n\} \to \{0, \ldots, k\}$. We will not notationally distinguish these
corresponding morphisms. So that $m$ may simultaneously refer
to all of the above morphisms.
\end{notation}

\subsection {The simplex category of a simplicial
set} \label{section_simplexcategory}
\begin{definition} \label{definition_simplexcategory} For $X$ a simplicial set,
$\Delta (X)$ will denote a certain category called the
\textbf{\emph{simplex category of $X$}}. This is the
category s.t.: 
\begin{itemize}
	\item The set of objects $\obj \Delta (X) $  is the set
of simplices 
$$\Sigma: \Delta ^{d} _{simp} \to X, \quad d \geq
   0. $$
	 \item Morphisms $f: \Sigma _{1} \to \Sigma
      _{2}$ are commutative diagrams in $s-Set$:
\begin{equation} \label{eq:overmorphisms1}  
\begin{tikzcd}
\D ^{d} _{simp} \ar[r, "\widetilde{f} "] \ar [dr,
   "\Sigma _{1}"] &  \D ^{n} _{simp} \ar
   [d,"\Sigma _{2}"] \\
    & X 
\end{tikzcd}
\end{equation}
with top arrow a simplicial map, which we denote
by $\widetilde{f} $.
\end{itemize}
An object $\Sigma: \Delta
^{d} _{simp} \to X $ is likewise called a $d$-simplex, and
such a $\Sigma $ will be said to have degree $d$. We may
specify the degree with a superscript, for example $\Sigma
^{d}$ for degree $d$.
\end{definition}
\begin{definition}\label{def:nondegeneratesimplex}
  We say that $\Sigma ^{m} \in \Delta (X) $   is
   \textbf{\emph{non-degenerate}}  if there is no
   morphism 
	 \begin{equation*}
\begin{tikzcd}
\D ^{m} _{simp} \ar[r, "\widetilde{f} "] \ar [dr,
   "\Sigma ^{m}"] &  \D ^{n} _{simp} \ar
   [d,"\Sigma ^{n} "] \\
    & X 
\end{tikzcd}
	 \end{equation*}
	  s.t. $m >n$. 
\end{definition}
There is a forgetful functor $$T: \Delta (X) \to
\D,$$ $T (\Sigma ^{d}) = \Delta^{d}  _{simp}$,
$T(f) = \widetilde{f}$.  
We denote by $\Delta ^{inj} (X)  \subset \D (X) $ the sub-category with same objects, and morphisms $f$
such that $\widetilde{f} $ are monomorphisms, i.e.
are face inclusions. 
\subsection{Geometric realization} 
\label{sub:geometricRealization} Let $Top$ be the
category of topological
spaces. Let $X$ be a simplicial set, then define as usual the
\textbf{\textit{geometric realization}} of $X$ by
the colimit in $Top$:
\begin{equation*}
|X| : = \colim _{\Delta (X)} T,
\end{equation*}
for $T: \Delta (X) \to \Delta \subset Top$  as
above, understanding $\Delta^{} $  as a
subcategory of  $Top$ as previously explained.   

\section {Smooth simplicial sets}  
If $$\sigma: \Delta ^{d} \to \Delta ^{n}  $$ is a smooth
map
then we have an induced map of simplicial sets
\begin{equation} \label{eq:induced}
 \sigma _{\bullet}: \Delta ^{d} _{\bullet} \to \Delta ^{n} _{\bullet},
\end{equation}
defined by $$\sigma _{\bullet } (\rho)  = \sigma
\circ \rho.$$  
%

We now give a pair of equivalent definitions of smooth
simplicial sets. The
first is more hands on, and has a close connection to the
definition of Chen/diffeological spaces, while the second is
more conceptual/categorical. 
\begin{definition}[First
   definition]
   \label{def:smoothsimplicialset1}  A \textbf{\emph{smooth
	 simplicial set}} consists of the following data: 
\begin{enumerate}
\item A simplicial set $X$.
\item \label{property:induced} For each $\Sigma: \Delta ^{n}
_{simp}  \to X$ an $n$-simplex there is an assigned map of simplicial sets 
\begin{equation*} g ({\Sigma}): \Delta ^{n} _{\bullet} \to
X.
\end{equation*}
This satisfies:
\begin{enumerate}
	\item \begin{equation} \label{axiom_gSigma1} g (\Sigma)| _{\D ^{n} _{simp}  }  = \Sigma.
\end{equation} 
We abbreviate $g ({\Sigma}) $ by
$\Sigma _{*} $, when there is no need to disambiguate 
which structure $g$ is meant. 
\item  \label{axiom:pushforward} The following
   property will be called
      \textbf{\emph{push-forward functoriality}}: 
      \begin{equation} \label{eq_pushforwardfunct}
         (\Sigma _{*} (\sigma)) _{*}  = \Sigma _{*} \circ \sigma _{\bullet}  
\end{equation} where $\sigma: \D ^{k} \to \D
      ^{d}$ is a $k$-simplex of $\Delta ^{d}
      _{\bullet}$, and  where $\Sigma$ as before is a $d$-simplex of $X$. 
\end{enumerate}

\end{enumerate}  
\end{definition}
Thus, formally a smooth simplicial set is a 2-tuple $(X,g)$, satisfying the axioms above.  When there is no need to disambiguate we omit specifying $g$.

%
%

\begin{definition} A \textbf{\emph{smooth}} map between smooth simplicial sets $$(X _{1}, g _{1} ),    (X _{2},  g _{2} )    $$ is a simplicial map $$f: X _{1} \to X _{2},
   $$ which satisfies the condition:
   \begin{equation} \label{eq:coherence}
 \forall n \in \mathbb{N} \,  \forall \Sigma  \in X
   _{1} (n):    g _{2} (f (\Sigma)) = f \circ g _{1} (\Sigma),
   \end{equation}
   or more succinctly: 
\begin{equation*}
\forall n \in \mathbb{N} \,  \forall \Sigma  \in X
   _{1} (n):   (f(\Sigma)) _{*} = f \circ \Sigma _{*}.
\end{equation*} 
\end{definition} 
 A \textbf{\emph{diffeomorphism}} between smooth
   simplicial sets is defined to be a smooth map,
   with a smooth inverse.  

Now let $\Delta ^{sm}$ denote the category:
\begin{enumerate}
	\item The set of
objects of $\Delta ^{sm}$ is $\mathbb{N} $.
\item $\hom _{\Delta ^{sm}} ({k}, {n}) $ is the set of smooth maps $\Delta ^{k} \to \Delta ^{n}$.
\item The composition of morphism is the natural composition.
\end{enumerate}
\begin{definition} [Second definition]
\label{def:smoothsimplicialset2} A \textbf{\emph{smooth
simplicial set}} $X$ is a functor $X: (\Delta ^{sm}) ^{op} \to Set$.
A smooth map $f: X \to Y$  of smooth simplicial sets is defined to be a natural transformation from the functor
$X$ to $Y$.
\end{definition}
The equivalence of the above definitions is established
further ahead, as we need certain preliminaries. In what
follows, we refer to the first definition unless specified
otherwise.
\begin{remark} \label{remark_}
There are respective advantages to both definitions. With the second
definition we can lean more on
category theory. In particular, some of the technical
results ahead are incarnations of the Yoneda
lemma and other such tools. For the first definition one can
work with the Kan condition more directly, and it is simpler to relate the first definition to the existing
theory of diffeological/Chen spaces, and to the existing
theory of differential forms on simplicial sets.
\end{remark}
\begin{example} [The tautological smooth simplicial set] \label{example:tautology} 
$\D ^{n} _{\bullet}$ has a  
tautological smooth simplicial set structure,
where $$g (\Sigma)  = \Sigma  _{\bullet},
$$ for $\Sigma: \D ^{k} \to \D ^{n}$ a smooth map, hence a $k$-simplex of $\D ^{n} _{\bullet}  $, and where $\Sigma _{\bullet} $ is as in \eqref{eq:induced}.  

\end{example} 
\begin{lemma} \label{lemma:simplicesaresmooth} Let $X$ be a smooth simplicial set and
   $\Sigma: \Delta ^{n} _{simp}  \to X $ an
   $n$-simplex. Let $\Sigma _{*}: \D ^{n}
   _{\bullet} \to X   $ be the corresponding simplicial
   map. Then $\Sigma _{*}  $ is smooth with
   respect to the tautological smooth simplicial
   set structure on $\D ^{n} _{\bullet}  $ as above.
\end{lemma}
\begin{proof} Let $\sigma$ be a $k$-simplex of $\D
   ^{n} _{\bullet}  $, so $\sigma: \D ^{k} \to \D ^{n}  $ is a smooth map. We need that  
   $$ (\Sigma _{*} (\sigma)) _{*} = \Sigma _{*} \circ 
   \sigma _{*}.    $$
   Now $\sigma _{*} = \sigma _{\bullet},  $ by
   definition of the tautological smooth structure
   on $\Delta^{n} _{\bullet }$. So we have: 
\begin{align*}
   (\Sigma _{*} (\sigma )) _{*} & = \Sigma _{*} 
    \circ \sigma _{\bullet}   \text{ by
   Axiom \ref{axiom:pushforward}} \\
   & = \Sigma _{*} \circ \sigma _{*}.
\end{align*}   
\end{proof}  

The following proposition in particular tells us that
the weak homotopy type of a smooth simplicial set (as
a plain simplicial set) is determined
by its complex of smooth simplices. For contrast, the weak
homotopy type of Chen/Diffeological space is not
generally determined by its complex of smooth simplices (cf.
~\cite[Example 3.12]{cite_homotopytheorydiffeological}).
\begin{proposition} \label{prop:smoothsimplices} 
   The set (just set) of $n$-simplices of a smooth simplicial set $X$ is naturally isomorphic to the set of smooth maps $\D ^{n} _{\bullet} \to X$.
In fact, define $X _{\bullet}$ to be the
simplicial set whose $n$-simplices are smooth
maps $\D ^{n} _{\bullet} \to X $,  and so that if $i: m \to n $  is a
morphism in $\Delta $ then $$X _{\bullet } (i): X
(n) \to X (m) $$   is the ``pull-back'' map: $$X
_{\bullet} (i) (\Sigma) = \Sigma \circ i
_{\bullet },  $$ for $i _{\bullet}: \Delta^{m}
_{\bullet }
\to \Delta^{n} _{\bullet}$  the induced map.
 Then $X _{\bullet}$ is naturally isomorphic to $X$.
\end{proposition}

\begin{proof} Let $\rho: \D ^{n}
   _{simp}  \to X$ be an $n$-simplex. By the lemma above, we have a uniquely associated to it smooth map $\rho
   _{*}: \D ^{n} _{\bullet} \to X  $. Conversely,
   suppose we are given a smooth map  
   $m: \D ^{n} _{\bullet} \to X.$ Then we get an
   $n$-simplex $\rho _{m}:=m| _{\D ^{n}
   _{simp}} $. 
   Let $id ^{n}: \Delta ^{n} \to \D ^{n}$ be the
   identity map.  We have that  
\begin{align*}
     m & = m \circ  id ^{n} _{\bullet} = m \circ
     id ^{n} _{*} \\ 
     & = (m (id ^{n})) _{*}, \text{ as $m$ is smooth } \\
     & = (\rho _{m} (id ^{n})) _{*}, \text{ trivially
     by definition of $\rho _{m}$ }   \\
   & = (\rho _{m}) _{*} \circ id ^{n} _{*}, \text{ as
     $(\rho _{m}) _{*}$ is smooth by Lemma
     \ref{lemma:simplicesaresmooth}}  \\
   & = (\rho _{m}) _{*}.
\end{align*} 
Thus, the map $I _{n}(\rho) = \rho _{*}$, from the set
of $n$-simplices of $X$ to the set of smooth maps $\D ^{n}
_{\bullet} \to X  $, is bijective.  

Given an element  $m \in hom _{\Delta } (n, d)$, 
let $m _{simp}: \Delta^{n} _{simp} \to \Delta^{d} _{simp}$
denote the corresponding natural transformation, also
identified with an element of $\Delta^{d} _{simp} (n)$.
Then the corresponding map $$X (m): X (d) \to X (n)$$ is $$\rho
\mapsto \rho \circ m _{simp},$$ for $\rho: \D ^{n}
   _{simp}  \to X$.

With that in mind, the diagram below commutes
\begin{equation*}
\begin{tikzcd}
X (d) \ar[r, "X(m)"] \ar [d, "I _{d}"] &  X (n) \ar [d,"I
_{n}"] \\
X _{\bullet } (d) \ar [r, "X _{\bullet } (m)"]
& X _{\bullet } (n),
\end{tikzcd}
\end{equation*}
as
\begin{align*}
	X _{\bullet} (m) \circ I _{d} (\rho) & = X _{\bullet } (m)
	(\rho _{*}) \\
	& = \rho _{*} \circ m _{\bullet} 
\end{align*}
while
\begin{align*}
I _{n} \circ X (m) (\rho)	& = (\rho \circ m _{simp}) _{*} \\
& = (\rho _* \circ m _{simp} ) _{*}, \text{ by
\eqref{axiom_gSigma1}} \\
& = (\rho _{*} (m _{simp})) _{*}, \text{ by
\eqref{eq_identification}} \\ 
& =  \rho _{*} \circ m _{\bullet }, \text{ by Axiom \ref{eq_pushforwardfunct}}.  
\end{align*}

Thus $I$  is a natural transformation and is an isomorphism
of simplicial sets $I: X  \to X _{\bullet }$.
\end{proof}
\begin{lemma}
   \label{lemma:inducedsmoothmap} Given a smooth
   $m: \Delta^{d}  _{\bullet} \to \Delta^{n}
   _{\bullet } $   there is a unique smooth map
   $f: \Delta^{d} \to \Delta^{n} $  such that
   $m=f_{\bullet}
    $.  
\end{lemma}
\begin{proof}
   Define $f$ by $m (id) $
   for $id: \Delta ^{d} \to \Delta ^{d}$     the
   identity. Then  
  \begin{align*}
     f _{\bullet}  & = (m (id)) _{\bullet} \\   
     & = (m (id)) _{*} \\
& = m \circ id _{*} \quad (\text{as $m$ is smooth}) \\
    & = m. 
   \end{align*}  
So $f$  induces $m$. Now if $g$ induces $m$  then $g _{\bullet } = m$ 
hence $g= g _{\bullet} (id) = m (id)  $.   
  \end{proof} 

\subsection{Smooth Kan complexes} \label{sec_Smooth Kan complexes}
\begin{definition} A smooth simplicial set whose underlying simplicial set is a Kan complex will be called a \textbf{\emph{smooth Kan complex}}.   
\end{definition}
The above notion will be crucial for us. 
\subsection{Some examples} \label{sec_ExampleKan}
Let $Y$ be a smooth manifold or more generally
a diffeological space and let $Sing^{sm}(Y)$ denote the simplicial set of
smooth singular simplices in $Y$ \footnote{This is often called the ``smooth singular simplicial set of $Y$''. However, for us ``smooth'' is reserved for  another purpose, so to avoid confusion we do not use such
terminology.}. That is $Sing ^{sm} (Y) (k)  $ is the set of smooth maps $\Sigma: \D
^{k} \to Y $, with its natural simplicial structure.
$Sing ^{sm} (Y) $  will often be abbreviated by $Y
_{\bullet}$. 
\begin{example} \label{example:smoothdfold} Let
$Y$ be a smooth manifold, or a diffeological space. Set $X=
Y_{\bullet}$,  then $X$ is naturally a smooth simplicial
set. When $Y$ is a finite dimensional manifold, $X$ is a smooth Kan complex, 
this is essentially
~\cite[Corollary 4.36]{cite_ChristensenWusmoothclassifying}.
Only ``essentially'', as the latter uses ``non-compact
simplices''. We will however not need this and so will not
elaborate.
\end{example} 
\begin{example} \label{example:loopspace} Here is one
special example. Let $M$ be a smooth manifold. Then there is a
natural smooth simplicial  set $LM ^{\Delta}$ whose
$d$-simplices $\Sigma$ are smooth maps $f
_{\Sigma}: \D ^{d} \times S ^{1} \to M$. The
maps $\Sigma _{*} $ are defined by $$\Sigma
_{*} (\sigma) = f _{\Sigma} \circ (\sigma
\times id),  $$ for $\sigma \in \Delta^{d} _{\bullet} (k)$
and $$\sigma \times id: \D ^{k}
\times S ^{1} \to \D ^{d} \times S ^{1},$$
the product map.
This $LM ^{\Delta^{} } $ is one simplicial model of the free loop space.
Naturally, the free loop space $LM$ also has the
structure of a Fr\'echet manifold, in particular we
have the smooth simplicial set $LM _{\bullet} $, whose
$n$-simplices are Gateaux $C ^{\infty}$ maps $\Sigma: \Delta ^{n}
\to LM $, see Hamilton ~\cite{cite_RichardHamiltonNashMoser}.  There is a natural simplicial map $LM ^{\D} \to LM
_{\bullet} $,  which is readily seen to be smooth.   (It is
  indeed a diffeomorphism.) 
	
The above smooth simplicial set structure $LM ^{\Delta^{} }$, in the language of diffeologies, is closely related to the functional diffeology
on $C
^{\infty }(Y , Z)$,  for which there are
diffeological diffeomorphisms:
\begin{equation*}
C ^{\infty }( X \times Y ,Z ) \to  C
^{\infty } (X ,C
^{\infty }(Y , Z)),
\end{equation*}
given another diffeological space $X$.
\end{example}
\subsection{Smooth simplex category of a smooth
simplicial set}
Given a smooth simplicial set $X$, there is an
extension of the previously defined simplex category
$\Delta^{} (X)$.  
\begin{definition}
   \label{def:extendedsimplexcategory} For $X$ a
   smooth simplicial set,
$\Delta ^{sm} (X)$ will denote the category whose
set of objects $\obj \Delta ^{sm} (X) $  is the
   set of smooth maps 
$$\Sigma: \Delta ^{d} _{\bullet} \to X, \quad d \geq
   0  $$ and morphisms $f: \Sigma _{1} \to \Sigma
      _{2}$, commutative diagrams:
\begin{equation} \label{eq:overmorphismssmooth}  
\begin{tikzcd}
   \D ^{d} _{\bullet} \ar[r, "\widetilde{f}
   _{\bullet}  "] \ar [dr,
   "\Sigma _{1}"] &  \D ^{n} _{\bullet} \ar
   [d,"\Sigma _{2}"] \\
    & X 
\end{tikzcd}
\end{equation}
%
with top arrow any smooth map (for the
tautological smooth simplicial set structure on
$\Delta^{d} _{\bullet}$), which we denote
by $\widetilde{f} _{\bullet} $.  By Lemma
\ref{lemma:inducedsmoothmap},  $\widetilde{f}  _{\bullet }$   is
induced by a unique smooth map $\widetilde{f}: \Delta ^{d} \to
\Delta^{n}$.
\end{definition}
By Proposition \ref{prop:smoothsimplices}  
we have
a natural faithful embedding $\Delta (X) \to \Delta
^{sm}(X)  $ that is an isomorphism on object
sets. We call elements of
$\Delta ^{sm} (X) $  $d$-simplices.
\begin{proposition} \label{prop:equivalenceofDefinitions}
   Definitions \ref{def:smoothsimplicialset1},
   \ref{def:smoothsimplicialset2} are equivalent.
\end{proposition}
\begin{proof}
Let $\mathcal{C} _{1}$ denote the category of smooth
simplicial sets as given by the Definition
\ref{def:smoothsimplicialset1}. And let $\mathcal{C}
_{2}$ denote the category of smooth simplicial sets as
given by the Definition \ref{def:smoothsimplicialset2}.

Given $X \in \mathcal{C}_{1}$, we define a functor $I(X):
(\Delta ^{sm}) ^{op} \to Set $  by setting $$I(X) ({k})
= \{\Sigma _{\bullet}: \Delta^{k} _{\bullet } \to X \,|\,
\Sigma _{\bullet } \text{ is smooth i.e. is a morphism in
$\mathcal{C}_{1}$}  \}. $$ And for $\sigma: \Delta^{k}
\to \Delta^{d} $ a smooth map setting $$I(X) (\sigma):
I(X) ({d}) \to I(X) ({k})  $$ to be the map
\begin{equation} \label{equation_defineIXmorphism} I(X)
(\sigma) (\Sigma _{\bullet}) = \Sigma _{\bullet } \circ
\sigma _{\bullet}. \end{equation} This defines $$I:
\mathcal{C}_{1} \to \mathcal{C}_{2}$$  on objects. 

Conversely, given  $F \in \mathcal{C}_{2}$,
define a simplicial set $I ^{-1} (F)$ by the rules: 
\begin{enumerate}
\item $I ^{-1} (F) (k) := F ({k})$.
\item  For
$\Sigma \in I ^{-1} (F) (k) $,  $\Sigma_*:
\Delta^{k}  _{\bullet} \to X$ is the map:
$$\Sigma _{*} (\sigma) = F(\sigma) (\Sigma).
$$
\end{enumerate}
  So that we get an element $I ^{-1} (F) \in \mathcal{C} _{1} $.
 This defines $$I ^{-1}: \mathcal{C} _{2} \to \mathcal{C} _{1}$$ on
 objects. By Proposition \ref{prop:smoothsimplices} $(I
 ^{-1} \circ I (X)) \simeq X$, an isomorphism in
 $\mathcal{C} _{1}$.
			
Suppose now we are given a morphism in $\mathcal{C} _{1}$: $f:
X _{0} \to X _{1} $  i.e. a simplicial map satisfying the condition:
\begin{equation} \label{eq_smoothnessf}
\forall n \in \mathbb{N} \,  \forall \Sigma  \in X
    (n):   (f (\Sigma)) _{*} = f \circ \Sigma _{*}.
\end{equation} 
Define a natural transformation:
\begin{equation*}
I ({f}): I (X _{0}) \to I (X _{1}),
\end{equation*}
by setting $I (f) _{{k} }: I (X _{0}) (k) \to I (X _{1}) (k) $
to be the map $I (f) _{{k}} (\Sigma _{\bullet}) = f \circ
\Sigma _{\bullet }$.

This is a natural transformation by the associativity of the
composition $f \circ (\Sigma _{\bullet } \circ \sigma
_{\bullet }) = (f \circ \Sigma _{\bullet }) \circ \sigma
_{\bullet }$.

It is clear that $I: \mathcal{C}_{1} \to \mathcal{C}_{2}$ is
a functor. We show that it is faithful on hom sets. If $f
_{0}, {f}': X _{0} \to X _{1}$ are a pair of morphisms in
$\mathcal{C}_{1}$ suppose that $I (f) = I (f')$. Then 
\begin{equation*}
\forall n \in \mathbb{N} \,  \forall \Sigma _{\bullet }  \in
I (X) (n):  f \circ \Sigma _{\bullet } = f' \circ \Sigma
_{\bullet }.  
\end{equation*} 
In particular, 
\begin{equation*}
\forall n \in \mathbb{N} \,  \forall \Sigma  \in
X (n):  f \circ \Sigma _{*} = f' \circ \Sigma
_{*},  
\end{equation*} 
as $\Sigma _{*} \in I (X) (n)$.
And so 
\begin{equation*}
\forall n \in \mathbb{N} \,  \forall \Sigma  \in
(X) (n):  f (\Sigma )  = f' (\Sigma ).  
\end{equation*} 
And thus $f=f'$.

We show that $I$ surjective on hom sets. Suppose that $N:
I (X _{0}) \to I (X _{1})$ is a morphism in $\mathcal{C}
_{2}$, i.e. a natural transformation of the corresponding
functors. So for $\sigma:
\Delta^{d} \to  \Delta^{k}$ smooth, we have a commutative diagram:
 \begin{equation} \label{eq_NaturalTransformN}
 \begin{tikzcd}
 I(X _{0}) (k)  \ar[r, "I(X _{0}) (\sigma )"] \ar [d, "N _{k}"]
 &   I(X _{0}) (d) \ar [d,"N _{d}"] \\
 I(X _{1}) (k) \ar [r, "I(X _{1}) (\sigma)"]   &  I(X _{2}) (d)
 \end{tikzcd}
 \end{equation}

Define a simplicial map
\begin{equation*}
f _{N}: X _{0} \to X _{1},
\end{equation*}
by $$f _{N} (\Sigma ) = N _{k} (\Sigma _{*})| _{\Delta^{k}
_{simp}} ,$$ for
$\Sigma \in X _{0} (k)$.

We check that $I (f _{N}) = N$. Let $\Sigma _{\bullet}:
\Delta^{d} _{\bullet } \to X _{0}$ be smooth. For $\sigma:
\Delta^{k} \to  
 \Delta^{d}$ smooth,
we have:
\begin{align*}
	I(f _{N}) _{d}   (\Sigma _{\bullet }) (\sigma)& = (f _{N} \circ
	\Sigma _{\bullet}) (\sigma),  \text{ by definition of $I$}  \\
	& = f _{N} (\Sigma _{\bullet } (\sigma)) \\
	& =N _{k} (\Sigma _{\bullet } (\sigma )_*)| _{\Delta^{k}
	_{simp}}, 
	 \text{ by definition of $f _{N}$} \\
	& = N _{k} (\Sigma _{\bullet } \circ  \sigma _{\bullet
	}) | _{\Delta^{k}
	_{simp}}, \text{ as $\Sigma _{\bullet }$ is smooth} \\
	& = N _{d} (\Sigma _{\bullet}) \circ  \sigma _{\bullet
	}| _{\Delta^{k}
	_{simp}},  \text{ by $N$ being a natural transformation,
	\eqref{eq_NaturalTransformN} and \eqref{equation_defineIXmorphism}} \\
	& =N _{d} (\Sigma _{\bullet }) \circ \sigma, \text{
	notation \ref{notation:1}}  \\
	& = N _{d} (\Sigma _{\bullet }) (\sigma ), \text{
	identification
	\eqref{eq_identification}}. 
\end{align*}
Since $\Sigma _{\bullet }, \sigma $ were general it follows
that $I (f _{N}) = N$.

We have proved that $I$ is a functor that is essentially
surjective on objects, and is fully-faithful on hom sets, it
follows by a classical theorem of category theory that $I$
is an equivalence of categories.
\end{proof}

\subsection{Products} \label{sec:products}
Given a pair of smooth simplicial sets $(X _{1}, g
_{1} ),   (X _{2}, g _{2}) $, the product  $X _{1}
\times X _{2}$ of the underlying simplicial sets,
has the structure of a smooth simplicial set $$(X
_{1} \times X _{2},  g _{1} \times g _{2}), $$
constructed as follows.
Denote by $\pi _{i}: X _{1} \times X _{2} \to X _{i}    $ the simplicial projection maps. 
Then for each $\Sigma \in (X _{1} \times X
_{2}) (d) $, $$(g _{1} \times g _{2}) (\Sigma): \Delta ^{d} _{\bullet} \to X _{1} \times X _{2}  $$ is defined by:
  \begin{equation*}
     (g _{1}  \times g _{2}) (\Sigma) (\sigma): = (g
     _{1} (\pi _{1} (\Sigma)) (\sigma), g _{2}
     (\pi _{2} (\Sigma)) (\sigma)).
  \end{equation*}  
\subsection {More on smooth maps}
\label{sec:moreOnsmoothmaps}
As defined, a smooth map $f: X \to Y$ of smooth simplicial
sets, induces
a functor $$\Delta ^{sm}f: \Delta ^{sm} (X)  \to
\Delta^{sm} (Y). $$     
This is defined by $\Delta ^{sm} f (\Sigma) = f \circ \Sigma $, where $\Sigma:
\Delta^{d} _{\bullet } \to X$ is in $\Delta
^{sm} (X)$. If $m: \Sigma _{1} \to \Sigma
_{2}$ is a morphism in $\Delta^{sm}  (X) $:
\begin{equation*}
\begin{tikzcd}
\Delta^{k} _{\bullet } \ar[r,
   "\widetilde{m} _{\bullet}"] \ar [dr,
   "\Sigma_1"] & \Delta^{d} _{\bullet }  \ar
   [d,"\Sigma _{2}"] \\
   &  X,
\end{tikzcd}
\end{equation*}
then  obviously the diagram below also
commutes:
\begin{equation*} 
\begin{tikzcd}
\Delta^{k} _{\bullet } \ar[r,
   "\widetilde{m} _{\bullet}"] \ar [dr,
   "h_1"] & \Delta^{d} _{\bullet }  \ar
   [d,"h _{2}"] \\
   &  Y,
\end{tikzcd}
\end{equation*}
where $h _{i}= \Delta ^{sm} f (\Sigma _{i}) = f \circ \Sigma
_{i}$,
$i=1,2$.   And so the latter diagram determines a morphism $\Delta
^{sm} f (m):  h _{1} \to h _{2}  $ in $\Delta^{sm}
(Y) $.  Clearly, this
determines a functor $\Delta ^{sm}f$ as needed.

%
%
%
\subsection {Smooth homotopy} 
\begin{definition}\label{def:smoothhomotopy}
   Let $X,Y$ be smooth simplicial sets. Set $I:=
   \Delta ^{1}
   _{\bullet }$ and let $I _{0}, I _{1}
   \subset I$  be the images of the pair of inclusions
   $\Delta^{0} _{\bullet } \to I $ corresponding
   to the pair of endpoints.
   A pair of smooth maps $f,g: X \to Y$   are called
   \textbf{\emph{smoothly homotopic}}  if  there
   exists a smooth map $$H: X \times I \to Y$$  such  that
	 $H| _{X \times I _{0}} = f$  and $H | _{X \times I _{1}
	 } = g $. $H$ will be called a smooth homotopy between $f,g$.

\end{definition} 
Let $X$ be a smooth simplicial set and $x _{0} \in X (0)$.
We say that a smooth $f: \Delta^{n} _{\bullet } \to X$ is
\textbf{\emph{relative to $x _{0}$}}  if $f| _{\partial \Delta^{n} _{\bullet }}$
has image in $x _{0, \bullet }$, where $x _{0, \bullet }$
denotes the image of $\Delta^{0} _{\bullet } \to X$
determined by $x _{0}$, and where $\partial \Delta^{n}
_{\bullet}$ is the sub-simplicial set corresponding to simplices
with image in $\partial \Delta^{n} $. We may analogously
define $F: \Delta^{n} _{\bullet } \times
I \to X$ to be relative to $x _{0}$, if $\partial \Delta ^{n}
_{\bullet } \times I$ has image in $x _{0, \bullet }$. We
call this a \textbf{\emph{relative homotopy}}. 
Then we have:
\begin{definition}\label{def:smoothpik}
Set $\pi ^{sm} _{k}
(X, x _{0}) $ to be the set of equivalence classes of smooth
relative to $x _{0}$ maps $f: \Delta^{k} _{\bullet } \to X $, 
where $f  \sim g$ if there is a smooth relative
homotopy $H:
\Delta^{k} _{\bullet} \times I \to X$, between $f,g$.
\end{definition}
\begin{remark} \label{rem_group} When $X$ is a Kan complex
$\pi ^{sm} _{k} (X, x _{0})$ can be shown to be a group, but
we will not need this.
\end{remark}

\subsection {Geometric realization} Geometric realization of
a smooth simplicial set $X$ is defined to be the geometric
realization of the underlying simplicial set.
\section{Differential forms on smooth simplicial
sets} \label{sec:Differential:forms}
The theory of differential forms on smooth simplicial sets
that we now present
is part of the standard abstract theory of differential forms on
simplicial sets. Some of the results of this are folklore,
for example the de Rham theorem can be credited to Sullivan
~\cite{cite_SullivanInfinitesimalcomputationsintopology},
but many much more detailed, subsequent expositions have been
made,  for example DuPont~\cite{cite_DupontCharClasses}. As such, the theory of differential forms here is a priori \emph{inequivalent}
to the theory of differential forms on diffeological spaces
in the sense of Souriau ~\cite{cite_SouriaGroupesDiff}.
If one wanted to translate our discussion of differential
forms into the language of
diffeological spaces, then probably it would be similar to
the work
of Katsuhiko~\cite{cite_KatsuhikoSimplicialCochainAlgebras},
see also ~\cite{cite_KatsuhikoOther}, ~\cite{cite_NobuyukiMayer-Vietoris}.

First we define smooth differential forms on the
topological simplices $\Delta^{d} $.
\begin{definition}\label{definition_}
Set $T \Delta^{d} := i ^{*} T\mathbb{R} ^{d}  $ for $i:
\Delta^{d} \to \mathbb{R} ^{d} $ the natural inclusion. Let
$T ^{*} \Delta^{d} $ denote the dual vector bundle. A \textbf{\emph{smooth differential $k$-form}}  $\omega$ on $\D ^{d} $ is
a section of $\Lambda ^{k} (T ^{*} \Delta^{d})$,
having a smooth extension to a section of $\Lambda ^{k} (T
^{*} N)$ for $N \supset \Delta^{d} 
$ an open subset of $\mathbb{R} ^{d} $.
\end{definition}
The above is equivalent to various other possible
definitions. For example we may take $\Delta^{d} $ to be
a special case of a smooth manifold with corners, and use a more
general theory of differential forms. This can be done, for
example, using theory of diffeological spaces
~\cite{cite_IglesiasDiffForms}. See also Karshon-Watts
~\cite{cite_KarshonDomainMaps}, which establishes one kind
of ``uniqueness of notions of smooth structures'' for the
case of simplices, so that our chosen model is canonical up
to suitable equivalence.

\begin{definition} \label{def:coherence:differentialform}
Let $X$ be a simplicial set. 
   A \textbf{simplicial differential $k$-form} $\omega$, or
	 just differential $k$-form where there is no possibility
	 of  confusion,  is an assignment for each $d$-simplex
	 $\Sigma$ of $X$ a smooth differential $k$-form $\omega (\Sigma )  = \omega _{\Sigma} $ on $\D ^{d} $, such that 
  \begin{equation} \label{eq:sigmaForm}
  i ^{*}\omega _{\Sigma _{2} }= \omega _{\Sigma _{1} }, 
  \end{equation}
for every morphism $i: \Sigma _{1} \to \Sigma _{2}
   $ in $\D (X)$, (see Section
   \ref{section_simplexcategory}).   
If in addition $X$ is a smooth simplicial set, and if in
addition:
   \begin{equation}
      \label{eq:formcoherence} 
    i ^{*} \omega _{\Sigma _{2}}= \omega _{\Sigma _{1} }, 
\end{equation}
for every morphism $i: \Sigma _{1} \to \Sigma _{2}$ in $\Delta^{sm}
(X)$ then we say that $\omega$ is
\textbf{\emph{coherent}}.      
\end{definition}
\begin{remark} \label{rem_coherence}
In the main applications here coherence will be unnecessary,
and so will not be assumed.
We can sharpen our constructions to get universal 
Chern-Weil forms that are coherent, but this will add much
length to the paper, and is only interesting in more in
depth applications so is postponed. (To get coherence, in
the construction of the universal bundles we must include
connections as part of the data, somewhat akin to what is
done in Freed-Hopkins ~\cite{cite_FreedHopkins}.)
\end{remark}

A simplicial differential form $\omega$ may be denoted
simply as $\omega = \{\omega  _{\Sigma }\}$. It may also be
convenient to use the
anonymous function notation $\Sigma \mapsto \omega _{\Sigma
}$.
%

\begin{example} \label{exampleInducedForm} If $Y$ is
a smooth $d$-fold, and if $\omega$ is a differential $k$-form on $Y$, then $\Sigma \mapsto  \Sigma ^{*} \omega $ is a coherent simplicial differential $k$-form on $Y _{\bullet }$ called the \textbf{\emph{induced simplicial differential form}}. And this determines a $dg$ map:
\begin{equation} \label{equation_G}
\Theta: \Omega ^{\bullet } (Y, \mathbb{R} ^{} ) \to \Omega
^{\bullet } (Y _{\bullet }, \mathbb{R} ^{} ).
\end{equation}
\end{example}
\begin{example} Let $LM ^{\Delta^{} }   $ be the
   smooth Kan complex of Example
   \ref{example:loopspace}. Then Chen's iterated
   integrals \cite{cite_ChenIteratedPathIntegrals} naturally give
   coherent differential forms on $LM ^{\Delta^{}
   }$. More specifically, each $d$-simplex of $LM
	 ^{\Delta^{} }   $ corresponds to a smooth ``plot'' of the
	 form
	 $\Delta^{d} \to LM$ (in Chen's language). Chen's iterated
	 integrals as differential forms on $LM$, amount to a rule
	 in particular giving a differential forms on $\Delta^{d} $, for each
	 such plot. The
	 coherence condition in our language is implied by the condition 
	 ~\cite[Definition 1.2.2]{cite_ChenIteratedPathIntegrals} for this rule. So that this
	 exactly gives a coherent differential form in our
	 language. 
\end{example}

Let $X$ be a simplicial set. We denote by $\Omega ^{k} (X) $ the $\mathbb{R}$-vector space of differential $k$-forms on $X$. 
Define $$d: \Omega ^{k} (X) \to \Omega ^{k+1}
(X)$$ so that $d (\omega)$ abbreviated by   $d \omega $ is:
$$\Sigma \mapsto d\omega _{\Sigma }.
$$  
Clearly we have $$d ^{2}=0. $$ 
A $k$-form $\omega$ is said to be
\textbf{\emph{closed}} if $d \omega=0$, and
\textbf{\emph{exact}} if for some $(k-1)$-form
$\eta$, $\omega = d \eta$. \begin{definition} The \textbf{\emph{wedge
   product}} on $$\Omega ^{\bullet}  (X) =
   \bigoplus _{k \geq 0} \Omega ^{k} (X)   $$ is
   defined by $$\omega \wedge \eta (\Sigma) = \omega
   _{\Sigma} \wedge \eta _{\Sigma}.
   $$ Then $\Omega ^{\bullet}  (X) $  has the
   structure of a differential graded $\mathbb{R}
   ^{}$-algebra with respect to $\wedge$.
\end{definition}

We then, as usual, define the \textbf{\emph{de Rham cohomology}} of $X$:
\begin{equation*}
   H ^{k} _{DR} (X) = \frac{\text{closed k-forms}}{\text{exact k-forms} },
\end{equation*}
then $$H ^{\bullet } _{DR} (X)  = \bigoplus _{k \geq 0}
H ^{k} _{DR} (X)$$ is a graded commutative $\mathbb{R}$-algebra. 

Versions of the simplicial de Rham complex have been used by Whitney and perhaps most famously by Sullivan
~\cite{cite_SullivanInfinitesimalcomputationsintopology}. In
particular, the proof of the de Rham theorem (next section) is due to Sullivan.
\subsection {Homology and cohomology of a 
simplicial set}    
\label{sec:Homology:and:cohomology}  
We go over this mostly to establish notation. 
For a simplicial set $X$, we define an
abelian group $$C _{k} (X, \mathbb{Z}), $$  as the
free abelian group generated by the set of
$k$-simplices $X (k) $. Elements of $C _{k} (X, \mathbb{Z}) $ are called $k$-\textbf{\emph{chains}}.
The boundary operator: 
\begin{equation*}
\partial: C _{k} (X, \mathbb{Z}) \to C _{k-1} (X, \mathbb{Z}),  
\end{equation*}
is defined on a $k$-simplex $\sigma$ by
$$\partial \sigma = \sum _{i=0} ^{n} (-1) ^{i} d _{i}
\sigma,$$ where $d _{i}: X (k) \to X (k-1) $ are the face maps, this is then extended by linearity to general chains.
Then clearly $\partial ^{2}=0 $. 

The homology
of this complex is denoted by $H _{k}
(X,\mathbb{Z}) $, called integral homology. 
The integral cohomology is defined analogously to the
standard topology setting, using dual chain
groups $C ^{k} (X, \mathbb{Z} ) = hom
(C _{k} (X, \mathbb{Z}), \mathbb{Z} )$.
The corresponding coboundary operator
is denoted by $d$ as usual:
\begin{equation*}
   d: C ^{k} (X, \mathbb{Z} )  \to C ^{k+1} (X,
   \mathbb{Z}).
\end{equation*}
Homology and
cohomology with other ring coefficients (or
modules) are defined analogously.  
Given a simplicial map $f: X \to Y$  there are
natural induced chain maps $f ^{*}: C ^{k} (Y, \mathbb{Z})
 \to C ^{k} (X, \mathbb{Z} )  $, and   
$f _{*}: C _{k} (X, \mathbb{Z})
\to C _{k} (X, \mathbb{Z} )  $.

We say that a pair of simplicial maps $f,g: X \to Y$  are 
\textbf{\emph{homotopic}}  if there is
a simplicial map $H: X \times \Delta ^{1} _{simp} \to 
Y$ so that $f= H \circ i _{0}$, $g=H \circ i _{1}$
for $i _{0}, i _{1}: X  \to X
\times \Delta ^{1} _{simp}$  corresponding to the
pair of end point inclusions  $\Delta^{0} _{simp}
\to \Delta^{1} _{simp}$. A \textbf{\emph{simplicial homotopy 
equivalence}}  is then defined analogously to the 
topological setting.

As is well known if $f,g$ are homotopic then $f
^{*}, g ^{*}$   and $f _{*}, g_*$  are chain
homotopic.
\subsection{Integration} Let $X$ be a
simplicial set. Given a chain $$\sigma =
\sum _{i} a _{i} \Sigma _{i} \in C _{k} (X,
\mathbb{Z}) $$ and a smooth differential form
$\omega$, we define:
\begin{equation*}
\int _{\sigma}  \omega= \sum _{i} a _{i} \int _{\D ^{k} } \omega _{\Sigma _{i} } 
 \end{equation*}
where the integrals on the right are the standard
integrals of differential forms. 
Thus, we obtain a homomorphism:
$$\int: \Omega ^{k}  (X)  \to C ^{k} (X,
\mathbb{R}), $$ $\int (\omega) $ is the
$k$-cochain defined by:  $$\int (\omega)  (\sigma) := \int
_{\sigma} \omega, $$ where $\sigma$ is a
$k$-chain. We will abbreviate $\int (\omega)$ by 
$\int \omega$. 
The following is well known.
\begin{lemma} \label{lemma:integration} For a
   simplicial set $X$,  the
   homomorphism $\int$ commutes with $d$, and
   so induces a
   homomorphism:
 \begin{equation*}
DR: \Omega ^{\bullet} (X)  \to C ^{\bullet} (X, \mathbb{R}),
\end{equation*} 
with the induced map on cohomology denoted by the same
symbol.
\end{lemma}

\begin{proof} We need that 
\begin{equation*}
    \int d \omega = d \int \omega.
 \end{equation*} 
Let $\Sigma: \Delta^{k} _{simp} \to X$ be a
   $k$-simplex. Then 
\begin{align*}
   (\int d \omega) _{\Sigma} & =  \int _{\Delta^{k}
	 } d \omega _{\Sigma},
   \quad \text{ by definition }  \\
   & = \int _{\partial 
   \Delta^{k}} \omega _{\Sigma},
   \quad \text{ by Stokes
   theorem} \\   
   & = d (\int \omega) _{\Sigma},
   \quad \text{ by the
   definition 
   of $d$ on co-chains.} 
\end{align*}
%
\end{proof}
The de Rham theorem  tells us that $DR$ is
a quasi-isomorphism see ~\cite{cite_DupontCharClasses}, but
we will not need this.

\subsection {Pull-back} \label{section_pullback}
Given a (smooth) map $f: X _{1} \to X _{2}  $ of
(smooth) simplicial sets, we define 
\begin{equation*}
f ^{*}: \Omega ^{k} (X _{2}) \to \Omega ^{k} (X _{1} )
\end{equation*}
naturally by 
\begin{equation} \label{equation_pullback}
f ^{*} (\omega) (\Sigma ) :=  \omega (f({\Sigma})).
\end{equation}

Let's check that
$f ^{*}$ commutes with $d$. We have:
\begin{align*}
	\forall \Sigma : f ^{*} (d \omega ) (\Sigma ) & = d \omega
	(f (\Sigma )) \\ 
	& = d (f ^{*} \omega (\Sigma) ) \\
	& =d (f ^{*} \omega) (\Sigma).
\end{align*}

So we have 
an induced differential graded $\mathbb{R} ^{}
$-algebra homomorphism:
\begin{equation*}
  f ^{*}:  \Omega ^{\bullet} (X _{2})  \to \Omega
  ^{\bullet}
  (X _{1}).
\end{equation*}
And in particular an induced $\mathbb{R} ^{}
$-algebra homomorphism:
\begin{equation*}
   f ^{*}:  H ^{\bullet} _{DR} (X _{2})  \to H
   ^{\bullet } _{DR}
  (X _{1}).
\end{equation*}

\subsection{Relation with ordinary homology
and cohomology} 
\label{sec:Relation with ordinary homology cohomology} 
Let $s-Set$ denote the category of simplicial sets
and $Top$ the category of  
topological spaces.
Let $$| \cdot|: s-Set \to Top $$  be the geometric
realization functor as defined in Section
\ref{sub:geometricRealization}.
Let $X$ be a (smooth)  simplicial set.
Then for any ring $K$ and any $d \in \mathbb{N} $ we have natural  chain maps
\begin{equation} \label{eq:chainRmaps} 
\begin{split}
   & CR: C _{d} (X, K)  \to C
    _{d} (|X|, K), \\
   &  CR ^{\vee}: C ^{d} (|X|,K)  \to C
    ^{d} (X, K). 
\end{split}
 \end{equation}
The chain map $CR$ is defined as follows.   
A $d$-simplex $\Sigma: \Delta^{d}  _{simp} \to
X$, by construction of $|X|$ naturally induces  a
continuous map $\Sigma _{top}: \Delta^{d} \to |X| $. Denote
by also by $\Sigma _{top}$ the corresponding generator of $C _{d} 
(|X|,K)$, then we set 
$CR (\Sigma) = \Sigma _{top} $ in this notation.      
Then $CR ^{\vee}$ is the dual chain map.

$CR$ and $CR ^{\vee}$  are quasi-isomorphisms, i.e. induce isomorphisms 
\begin{equation} \label{eq:chainRmapshomology} 
\begin{split}
   & R: H _{d} (X, K)  \to H
    _{d} (|X|, K), \\
   &  R ^{\vee}: H ^{d} (|X|,K)  \to H
    ^{d} (X, K). 
\end{split}
\end{equation}
\begin{remark} \label{rem_}  This works as follows. Let $|X| ^{f}$ denote the
geometric realization of $X$ omitting the degeneracies in
the colimit construction, (that is we take the colimit over
the category of face maps). Then $|X| ^{f}$ is an infinite
dimensional $\Delta^{} $-complex, and as shown by Hatcher
~\cite[Section 2.1]{cite_HatcherAlgebraic}, the
$\Delta$-complex homology of a $\Delta^{} $-complex is 
isomorphic to its singular homology. On the other hand, for
$|X| ^{f}$ the $\Delta^{} $-complex homology is naturally
identified with $H _{d} (X,K)$. Since $|X| ^{f}$ is weakly
equivalent to $|X|$
~\cite[Appendix A]{cite_SegalCategoriesAndCohomology}, this readily implies the claim.
\end{remark}

%
Now let $Y$ be a (finite dimensional, possibly with corners) smooth manifold and $X=Y _{\bullet
 } = Sing ^{sm} (Y) $.  The
natural map 
\begin{equation} \label{equation_h}
h: |Y_{\bullet }| \to Y
\end{equation}
is a weak homotopy equivalence. To see this let $f: S ^{n}
\to Y$ represent a class in $\pi _{n} (Y, y _{0})$, then
there is a smooth $f': S ^{n} \to Y$ representing the same
class. It readily follows that the class $[f']$ is in the
image of $h _{*}: \pi _{k}  (|Y _{\bullet}|, h ^{-1} (y
_{0})) \to \pi _{k}  (Y, y _{0})$. Injectivity of $h _{*}$
is verified similarly.
So $h$ is a homotopy equivalence by the Whitehead's theorem. 

Let us denote by
\begin{equation}
   \label{eq:n}
   N: Y \to |Y_{\bullet }|,  
\end{equation}
a homotopy inverse.
 
Define
\begin{equation*}
I: H _{d} (Y _{\bullet}, K) \to H _{d} (Y,K)
\end{equation*}
to be the map induced by the chain map $$CI: C _{d} (Y
_{\bullet}, K) \to C _{d} (Y,K)$$ sending the
generator of $C _{d} (Y _{\bullet}, K) $,
corresponding to a simplex $\Sigma  \in Y _{\bullet
} (d) $, to the generator of $C _{d} (Y) $, corresponding to
the smooth map $\Sigma _{top}: \Delta^{d} \to Y$  (as
$\Sigma  \in Y _{\bullet
} (d) $ by definition uniquely corresponds to such a
smooth map).

Then factor $R$ as:
\begin{equation*}
\begin{tikzcd}
 H _{d} (Y _{\bullet}, K)
\ar[r, "I"] \ar [dr, "R"] & H _{d}
    (Y,K)   \ar [d,"N _{*}"] \\
    & H_{d} (|Y_{\bullet}|, K).
\end{tikzcd}
\end{equation*}

We may factor $R ^{\vee }$  as: 
\begin{equation} \label{eq:RlowerstarSpaceSmooth2}
\begin{tikzcd}
H ^{d} (|Y _{\bullet}|,K) \ar[r, "N ^{*}"] \ar [dr, "R
^{\vee}"] &  H ^{d} (Y,K)  \ar [d,"I ^{\vee }"] \\
    & H ^{d} (Y _{\bullet}, K),
\end{tikzcd}
\end{equation}
where $I ^{\vee }$ is induced by
the dual $CI ^{\vee}$ of $CI$. 

%

%
%
\begin{notation} \label{notation:geomRealizationClass}  
Let $\alpha \in H ^{d} (X, K) 
$.
\begin{enumerate}
	\item   We set
\begin{equation*}
|\alpha| :=(R ^{\vee}) ^{-1} (\alpha) \in H ^{d} (|X|, 
K). 
\end{equation*}
\item If $Y$	is a smooth manifold,  and $X=Y _{\bullet }$.  We set \begin{equation*}
|\alpha| _{sm}:=  (I ^{\vee}) ^{-1} (\alpha)  \in H ^{d} (Y, K).
\end{equation*}
\label{notation:geomRealizationClasspart2}
\end{enumerate}
 \end{notation}
Given a map of simplicial sets $f: X _{1} \to X
_{2}$  we let $|f|: |X _{1}| \to |X _{2}|$
denote the induced map of geometric realizations.  
\begin{lemma}
   \label{lemma:naturalrealization} Let 
   $f: X _{1} \to  X _{2}$ be a simplicial map of 
   simplicial sets.
  Let $f ^{*}: H ^{d} (X _{2}, {K} ) \to H
   ^{d} (X _{1}, {K} ) $ be the
   induced homomorphism  then:
\begin{equation*}
   |f ^{*} (\alpha)| = |f| ^{*} (|\alpha|).
\end{equation*}
\end{lemma}
\begin{proof}
   We have a clearly commutative diagram of chain 
   maps (omitting the coefficient ring):
   \begin{equation*}
   \begin{tikzcd}
   C _{d} (X _{1})  \ar[r, "CR"] \ar 
      [d, "f _{*}"] &  C _{d} (|X _{1}|)  \ar 
      [d,"|f| _{*}"] \\
   C _d (X _{2})  \ar [r, "CR"]   & C _{d} (|X 
      _{2}|),
   \end{tikzcd}
   \end{equation*}
  from which the result immediately  follows.
\end{proof}
The following is immediate from definitions.
\begin{lemma} \label{lemma_Iveetrivial}
If $K=\mathbb{R} ^{} $ then $$I ^{\vee} \circ DR ^{ord} = DR
\circ H\Theta,$$ 
where: 
\begin{itemize}
\item $DR ^{ord}: H ^{d} _{DR} (Y, \mathbb{R} ^{}) \to
H ^{d} (Y, \mathbb{R} ^{} ) $ is the ordinary de Rham
integration isomorphism. 
\item $DR$ is as in Lemma \ref{lemma:integration}.
\item $H \Theta: H ^{d} _{DR} (Y, \mathbb{R} ^{} ) \to H
^{d} _{DR} (Y _{\bullet }, \mathbb{R} ^{} )$ is
the cohomology map induced by the map $\Theta $ as in
\eqref{equation_G}.
\end{itemize}
\end{lemma}



%
\section{Smooth simplicial $G$-bundles}    
\label{sec:Simplicial bundles and connections} 

In what follows $G$ may be assumed to be either a locally
convex Lie group or a diffeological Lie group, with
smoothness interpreted in the corresponding categories, and
where the diffeology on $\Delta^{n} $ is the subspace
diffeology. 

In Section \ref{sec:Chern-Weil} we specialize to $G$ being
a generalized Lie group, for some basics on
the subject of infinite dimensional Lie
groups we refer the reader to Neeb~\cite{cite_Neeb}.
We now introduce the basic building blocks for
simplicial $G$-bundles.   
\begin{definition} \label{remark:concreteness}  Let $P$ be
a topological principal $G$-bundle over $\Delta^{n}$, with
the embedding $\Delta^{n} \subset \mathbb{R} ^{n}$ as
previously. Suppose we have a choice of a maximal atlas of
topological $G$-bundle trivializations $\phi _{i}: U _{i} \times G \to P $, $U _{i}
\subset \Delta^{n} $ open, s.t. the transitions maps
\begin{equation*}
(U _{i} \cap U _{j}) \times G \xrightarrow{\phi _{ij} = \phi
_{j} ^{-1} \circ \phi _{i}, } (U _{i} \cap U _{j}) \times G
\end{equation*}
extend to smooth maps $N \times G \to N \times G$, for $N
\supset U _{i} \cap U _{j} $ some open set in $\mathbb{R} ^{n} $.
Then with such a choice of an atlas we call $P$
\textbf{\emph{a smooth $G$-bundle over $\Delta^{n} $}}.
Smooth bundle maps, and isomorphisms are then defined as with standard smooth bundles. 
\end{definition}


At this point our  
terminology may partially clash with common 
terminology, in particular a simplicial  $G$-bundle
will \emph{not}  be a presheaf on $\D$ with values in the
category of smooth $G$-bundles. Instead, it will
be a functor (not a co-functor!)  on $\Delta ^{sm}
(X) $ with additional
properties. Presheafs of this type will not appear in the paper so that
this should not cause confusion.   

In the definition of simplicial
differential forms we omitted coherence. In the
case of simplicial $G$-bundles, the analogous
condition (full functoriality on $\D ^{sm} (X) $) turns out to be necessary if we want universal simplicial $G$-bundles with
expected behavior.   
\begin{notation}
   \label{notation:cG} Let $G$ be as above, we denote by $\mathcal{G}$ the category of smooth
	 principal $G$-bundles over the simplices $\Delta^{n} $,
	 ($n$ not fixed) with morphisms smooth $G$-bundle maps. 
\end{notation}
Let $\mathcal{F}  _{1}: \Delta^{sm} (X) \to \Delta^{sm} $ be the
natural forgetful functor. And $\mathcal{F} 
_{2}: \mathcal{G} \to \Delta^{sm} $ the 
functor taking a $G$-bundle $P \to \Delta^{k} $ to
${k} $ and defined on morphisms as follows.
If $\widetilde{\phi} : P _{1} \to P _{2} $ is a morphism in
$\mathcal{G} $ over a smooth map $\phi: \Delta^{k} \to
\Delta^{n} $ then $\mathcal{F} _{2} (\widetilde{\phi}) (k,
n) = \phi$.
\begin{definition} \label{def:simplicialGbun} Let 
   $G$ be as above and $X$  a smooth
   simplicial set. A \textbf{\emph{smooth
   simplicial $G$-bundle}} $P$ over $X$ is a 
   functor $P: \Delta ^{sm} (X) \to
   \mathcal{G}  $, so that the diagram:
	 \begin{equation*}
	 \begin{tikzcd}
	 \Delta ^{sm} (X) \ar[r, "P"] \ar [dr, "\mathcal{F}
	 _{1}"] & \mathcal{G}  \ar [d,"\mathcal{F}  _{2}"] \\
	      &  \Delta^{sm},
	 \end{tikzcd}
	 \end{equation*}
commutes. We will call this the \textbf{\emph{compatibility condition}}.
\end{definition}

We will only deal with smooth simplicial
$G$-bundles, and so will usually say
\textbf{\emph{simplicial $G$-bundle}}, omitting
the qualifier `smooth'. 
\begin{notation}
	 \label{notation:} To reduce the use of the parenthesis, we often use notation $P _{\Sigma}$ for $P (\Sigma)$. Note that this notation is used exclusively for objects.  If we write a simplicial $G$-bundle
	 $P \to X$, this means that $P$ is a simplicial
	 $G$-bundle over $X$ in the sense above.  So that $P \to
	 X$ is just notation not a morphism.
\end{notation}

\begin{example} 
\label{example:product} 
If $X$ is a smooth
simplicial set and $G$ is as above, we denote by $X \times
G$  the simplicial $G$-bundle, $$ \forall n \in \mathbb{N},
\forall \Sigma ^{n} \in \Delta (X): (X \times G)_{\Sigma
^{n}} \text{ is the trivial bundle $\Delta^{n}  \times
G \to \Delta^{n} $}. $$ And where for $\sigma: \Delta^{n}
\to \Delta^{k} $, $(X \times G) (\sigma): \Delta^{n}  \times G \to \Delta^{k}  \times G$ is the map $\sigma \times id$.

This is called
the \textbf{\textit{trivial simplicial $G$-bundle over
$X$}}. 
\end{example}   

\begin{example} \label{ex:inducedbundle}
   Let $Z \to Y$ be a diffeological $G$-bundle over a
   diffeological space $Y$.  Or a smooth $G$-bundle over
	 a smooth manifold $Y$ with $G$ locally convex.
	 Then we have a simplicial
   $G$-bundle $Z ^{\Delta} $ over $Y _{\bullet} $
   defined by the conditions: 
\begin{enumerate}
		\item $Z ^{\Delta^{} } _{\Sigma} = \Sigma
   ^{*} Z. $
	 \item  For $f: \Sigma _{1} \to
   \Sigma _{2}$ a morphism,  the bundle map $$Z ^{\D}
   (f): Z ^{\Delta^{}} _{\Sigma _{1}} \to   Z
   ^{\Delta^{}}_{\Sigma _{2}} 
   $$ is the 
	 universal map $u:  \Sigma ^{*} _{1}Z \to \Sigma
	 ^{*}_{2} Z$ corresponding to the universal pull-back
	 property of $\Sigma ^{*} _{2} Z$.
	\end{enumerate} 
	The uniqueness of the universal maps readily implies that
	$Z ^{\Delta^{}}$ is a functor. We say that $Z ^{\Delta} $ is \textbf{\emph{the simplicial $G$-bundle induced by $Z$}}. 
\end{example}

\begin{definition} \label{def:simplicialGbundlemap}
Let $P _{1} \to X _{1} , P _{2} \to X _{2}    $ be
a pair of simplicial $G$-bundles. 
Let $h: X _{1}
   \to X _{2}$ be a smooth map. A
\textbf{\emph{smooth simplicial $G$-bundle map
   over $h$}} from $P _{1}$ to $P _{2}$ is a 
 natural transformation of functors:
\begin{equation*}
   \widetilde{h}: P _{1} \to P _{2} \circ \Delta
   ^{sm} h,
\end{equation*} 
such that the following additional
property is satisfied. For each $d$-simplex $\Sigma \in \Delta^{sm}  (X
   _{1}) $ the natural transformation
   $\widetilde{h} $ specifies a morphism in
   $\mathcal{G}$:
\begin{equation*}
    \widetilde{h} _{\Sigma}: P _{1} (\Sigma) \to P
   _{2} (h \circ \Sigma),  
 \end{equation*}
and we ask that this is a bundle map over the
identity, so that the following diagram commutes:
  \begin{equation*}
   \begin{tikzcd}
  P _{1} (\Sigma)   \ar[r, "\widetilde{h} _{\Sigma} "] \ar [d,
      "p _{1}"] &  P _{2} (h \circ \Sigma)  \ar [d,"p
      _{2}"] \\
   \Delta ^{d} \ar [r, "id"]   & \D
      ^{d}.
   \end{tikzcd}
   \end{equation*}
\end {definition} 
We will usually say  
simplicial $G$-bundle map instead of smooth
simplicial $G$-bundle map, (as everything is
   always smooth)  when $h$ is not
specified it is assumed to be the identity.
\begin{definition}
   Let $P _{1}, P _{2}$ be simplicial $G$-bundles 
   over $X
   _{1}, X _{2}$ respectively. A 
   \textbf{\emph{simplicial
   $G$-bundle isomorphism}} is a simplicial 
   $G$-bundle map
   $$\widetilde{h}:P _{1} \to P _{2}$$  s.t.
   there is a
   simplicial $G$-bundle map $$ \widetilde{h}  
   ^{-1}: P _{2}
   \to P _{1}$$ with $$\widetilde{h}  ^{-1} \circ
   \widetilde{h} = id.$$   
\end{definition}
Usually $X _{1} = X _{2} $ and in this case, unless specified otherwise, it is assumed $h=id$. A simplicial $G$-bundle isomorphic to the trivial simplicial $G$-bundle is called
\textbf{\emph{trivializeable}}.
\begin{definition} \label{def:inducedGbundle}
If $X= Y _{\bullet}$ for $Y$  a smooth manifold,
we say that a simplicial $G$-bundle $P$ over $X$
is \textbf{\emph{inducible by a smooth $G$-bundle}}
$N \to Y$ if there is a simplicial $G$-bundle
isomorphism $N ^{\D} \to P$.
\end{definition}
The following will be one of the crucial ingredients later on.
(Recall also that we are treating two distinct cases
simultaneously: a diffeological $G$ and locally convex $G$,
the proof of the following theorem works the same way in both cases.)
\begin{theorem} \label{lemma:inducedSmoothGbundle}
Let $G$ be as above and let $P \to Y _{\bullet }$ be a simplicial $G$-bundle, for $Y$ a smooth
   $d$-manifold (or a manifold with corners, understood as
	 previously). Then $P$ is inducible by some smooth
	 $G$-bundle $N \to Y $.
\end{theorem} 
\begin{proof}  
We need to introduce an auxiliary notion. Let $Z$
be a smooth $d$-manifold with corners, as before understood
as a diffeological space.  And let
$\mathcal{D}  (Z)  $ denote the category whose objects are
smooth (diffeological) embeddings $\Sigma: \Delta^{d} \to Z$, (for the same fixed $d$). 
A morphism $f \in hom _{\mathcal{D} (Z) } (\Sigma _{1}, \Sigma _{2}) $ is a 
commutative diagrams:  
\begin{equation} \label{eq:overmorphisms2}  
\begin{tikzcd} \D ^{d} 
 \ar[r, "\widetilde{f}"] \ar [dr, "\Sigma
   _{1}"] & \D ^{d}  \ar [d,"\Sigma
   _{2}"] \\
 &  Z.
\end{tikzcd} 
\end{equation}
Note that the map $\widetilde{f}$ is unique, when
   such a diagram exists, as $\Sigma _{i}$
   are embeddings. Thus $hom _{\mathcal{D} (Z) } (\Sigma _{1}, \Sigma
      _{2})$ is either empty or consists of a
      single element.

Although, we state the result for manifolds with corners,
for simplicity we assume here that $Y$ is a manifold.
Let $\{O _{i} \} _{i \in I}$ be a locally finite open cover of $Y$, closed under 
intersections, with each $O _{i}$
diffeomorphic to an open ball in $\mathbb{R} ^{d} $. Such a cover is often called a good cover of a manifold. Existence of such a cover is a folklore theorem, but a proof can be found in
~\cite[Prop A1]{cite_GoodCover}. 

Let $\mathcal{O}$ denote the category with the set
of objects $\{O _{i} \} $ and with morphisms inclusions.    
And set $C _{i} = \mathcal{D} (O _{i}) $, then clearly $C
_{i}$ is a full subcategory of  $\D ^{sm} (Y _{\bullet}) $.   For each
$i$, we have the functor $$F _{i} = P|_{C _{i}}: C _{i} \to \mathcal{G}. $$

By assumption that each $O _{i}$ is diffeomorphic
to an open ball,  $O _{i}$ has an exhaustion
by embedded $d$-simplices. This means that there
is a sequence of smooth embeddings $\Sigma
_{j}: \Delta^{d} \to O _{i} $ satisfying:
\begin{itemize}
   \item  $\interior \image
(\Sigma _{j+1}) \supset  \image (\Sigma _{j})
$  for each $j \in \mathbb{N} $.
\item $\bigcup_{j}
\image (\Sigma _{j}) = O _{i} $.
\end{itemize}
As each element of $C _{i}$ is contained in some $\Sigma
_{j}$, $$\Sigma _{0} \to \ldots \to \Sigma _{j} \to \Sigma
_{j+1} \to \ldots $$ forms a final sub-category of $C
_{i}$.
Thus, for each $i$, the colimit in $\mathcal{G}$:   
\begin{equation} \label{eq:colimPi}
      P _{i} := \colim _{C _{i}} F_{i}
\end{equation}
is the colimit of the sequence $$P (\Sigma _{0}) \to \ldots
\to P (\Sigma _{j}) \to P (\Sigma  _{j+1}) \ldots  $$
And this colimit is naturally a topological 
$G$-bundle over $O _{i} = \cup _{i} \image (\Sigma _{i})$. 
      
We may give $P _{i}$ the structure of a smooth $G$-bundle,
with $G$-bundle charts defined as follows. For each $\Sigma
\in C _{i}$, pick a smooth trivialization:
$$\xi _{\Sigma}: (\Delta^{d}) ^{\circ }  \times G \to
(P ^{\circ} _{\Sigma}: = P _{\Sigma}| _{(\Delta^{d})
^{\circ}}).$$

Then set $\phi _{\Sigma,i}$ to be the composition map
$$(\Delta^{d}) ^{\circ } \times G \xrightarrow{\xi _{\Sigma
}} P _{\Sigma} ^{\circ } \xrightarrow{c _{\Sigma}} P _{i}, $$     
where $c _{\Sigma }: (P _{\Sigma } = F _{i} (\Sigma)) \to P _{i}$ is the natural map in the colimit diagram of \eqref{eq:colimPi}.
\begin{lemma} \label{lem_atlas}
The collection $\{\phi 
_{\Sigma,i}\}$ forms a smooth $G$-bundle
atlas for $P _{i}$. 

\end{lemma}
\begin{proof} [Proof]
Suppose that $z
\in (\image \phi _{\Sigma, i}) \cap (\image \phi
_{\Sigma', i}) $. Then there is a morphism
$\Sigma'' \xrightarrow{m} \Sigma $ and a morphism $\Sigma ''
\xrightarrow{m'} \Sigma'$ such that $z \in \image \phi
_{\Sigma'', i}$. And such that
the following
composition maps coincide:
\begin{align*}
& P ^{\circ} _{\Sigma''} \xrightarrow{P (m)} P ^{\circ
} _{\Sigma } \xrightarrow{c _{\Sigma }} P _{i} \\
& P ^{\circ} _{\Sigma''} \xrightarrow{P (m')} P ^{\circ
} _{\Sigma' } \xrightarrow{c _{\Sigma' }} P _{i}.
\end{align*}
Hence, $c _{\Sigma' } ^{-1} \circ c _{\Sigma } = P (m')
\circ P (m) ^{-1} $ where the inverses are defined on the
suitable open sub-domains. Thus $c _{\Sigma' } ^{-1} \circ c _{\Sigma }$ is smooth, which clearly implies our claim.
%
%

\end{proof}

So we obtain a functor $$D: \mathcal{O} \to \mathcal{G},$$  
defined by $$D (O _{i}) = P _{i},$$ and defined
naturally on morphisms.  Specifically, a morphism
$O _{i _{1}} \to O _{i _{2}}$  induces a functor
$C _{i _{1}} \to C _{i _{2}}  $   and hence a
smooth $G$-bundle map
$P _{i _{1}} \to P _{i _{2}} $, by the naturality 
of the colimit.

      Let $t: \mathcal{O} \to Top$ denote the tautological
			functor, sending the subspace $O$ to the
			corresponding topological space, so that  $Y = \colim _{\mathcal{O}} t  $, 
  where for simplicity we write equality for
natural isomorphisms here and further on in
this proof.   

Now,
\begin{equation}
   \label{eq:colimitN}
   N := \colim _{\mathcal{O}} D, 
\end{equation}
is naturally a
topological $G$-bundle $$N \xrightarrow{p} \colim
_{\mathcal{O}} t = Y.$$ Let $c _{i}: P _{i}  \to N$
denote the natural maps in the colimit diagram
of \eqref{eq:colimitN}.
The collection of charts $\{c _{i} \circ \phi  _{\Sigma,i}\}
_{i,\Sigma \in C _{i}}$ forms a smooth $G$-bundle atlas on
$N$ (repeat the argument of Lemma \ref{lem_atlas}). Let us rename these charts as $\{\rho _{k}\}$, for $k$ elements of the index set implicit above.

We now prove that $P$ is induced by $N$.
Let $\Sigma: \Delta^{n}  \to Y$ be smooth, then $\{V _{i} :=  \Sigma
^{-1} (O _{i}) \} _{i \in I}$ is a
locally finite and hence finite open cover of $\Delta^{n}
$ closed under intersections. 
Let $N ^{\D}$ be the simplicial $G$-bundle over $Y _{\bullet }$
induced by $N$. So $$N ^{\D}_{\Sigma} := \Sigma ^{*}N. $$ 


As $\Delta^{n}$ is a convex subset of
$\mathbb{R} ^{n} $, the open metric balls in
$\Delta ^{n}$,  for the
induced metric,
are convex as subsets of $\mathbb{R} ^{n}$.
Consequently, as each $V _{i} \subset \Delta^{n} $
is open, it has a basis of convex metric balls, with respect to
the induced metric.  
By Rudin~\cite{cite_RudinMetricParacompact} there is
then a locally finite cover of $V _{i}$  by
elements of this basis. In fact, Rudin shows
any open cover of $V _{i}$ has
a locally finite refinement  by elements of such a
basis. 

Let $\{W ^{i} _{j}\}$ 
consist of elements of this cover and all 
intersections of its elements, (which must then be
finite intersections).    
So  $W ^{i}_{j} \subset   V _{i} $  are
open convex subsets and $\{W
^{i}_{j}\}$  is a locally finite open cover of $V
_{i}$, closed under finite intersections.

As each $W ^{i}_{j} \subset \Delta^{n} $ is open and convex it has an exhaustion by nested images of embedded simplices. That is 
\begin{equation} \label{equation_WijExhaustion}
W ^{i} _{j}
= \bigcup_{k \in \mathbb{N} } \image \sigma
^{i,j}_{k}  
\end{equation}
for $\sigma ^{i,j} _{k}:
\Delta^{d}  \to W ^{i}_{j}$ smooth and
embedded, with $\image \sigma ^{i,j} _{k}
\subset \image \sigma ^{i,j} _{k+1}$ for each $k$.  


Let $C$ be the small category with objects $I
\times J \times \mathbb{N} $, so that there is
exactly one morphism from $a= (i,j,k) $  to $b=
(i',j',k') $  whenever $\image \sigma ^{i,j} _{k}
\subset \image \sigma ^{i',j'} _{k'},$   and no
morphisms otherwise.
Let $$F: C \to \mathcal{D} (\Delta^{d} ) $$ be the
functor $F (a) =  \sigma ^{i,j}_{k}$ for
$a=(i,j,k) $, (the definition on morphisms is
forced). For brevity, we denote $\sigma _{a} := F (a) $.

For a smooth manifold with corners $Y$, if $\mathcal{O} (Y) $  denotes the category of topological subspaces of $Y$ with morphisms inclusions,  then there is a forgetful functor $$T: \mathcal{D} (Y)  \to \mathcal{O} (Y) $$ which takes $f$ to $\image (\widetilde{f})$. 
With all this in place, we have: 
\begin{lemma} \label{lemma_}
\begin{equation} \label{eq_colimitDd}
   \Delta^{d}  = \colim _{C} T \circ F,
\end{equation}
as a colimit in $Top$.
\end{lemma}
\begin{proof} [Proof]
First recall that 
a general topological space $X$
is the colimit of any open cover $\{O _{i}\} $ of $X$ closed
under intersections. In particular, $\Delta^{d} $ is the
colimit of the cover $\{W ^{i} _{j}\} _{i,j}$. On the other
hand $W ^{i} _{j} = \colim _{S _{i,j}} T \circ F$, for  $S
_{i,j} \subset C$ a full subcategory corresponding to
the exhaustion \ref{equation_WijExhaustion}. Moreover, $C
= \cup _{i,j} S _{i,j}$. The result
readily follows.
\end{proof}

It follows that 
\begin{equation*}
N ^{\D} _{\Sigma} = \colim _{C} N ^{\D} \circ \Delta
^{sm} \Sigma  \circ F. 
\end{equation*}

Now, by construction for each $a \in C$,
$\Sigma \circ  \sigma _{a}$ is contained in an open set 
$O _{i}$ diffeomorphic to the standard open ball in $\mathbb{R} ^{d}
$.
It follows that we may express:
\begin{equation}
   \label{eq:Sigmaacirc}
   \Sigma \circ  \sigma _{a} =
\Sigma _{a} \circ m _{a} \circ  \sigma _{a},
\end{equation}
for some $\Sigma _{a}: \Delta^{d}  \to O _{i} \subset
Y$ a smooth embedded $d$-simplex.  And $m _{a}: \Delta^{n}
\to \Delta^{d}  $ smooth. 

So for all $a \in C$, $$N ^{\D} \circ \Delta ^{sm} \Sigma  \circ F (a)
= (m _{a} \circ  \sigma _{a}) ^{*} P _{\Sigma_a},$$
after naturally identifying $P _{\Sigma _{a} }$ with $N
^{\Delta^{} } _{\Sigma _{a}}$.
More precisely, there is a natural isomorphism 
$\phi _{a}: P 
_{\Sigma _{a} } \to N ^{\D} _{\Sigma _{a}}$ given by the
composition:
\begin{equation} \label{eq:composition}
   P _{\Sigma _{a}} \to P _{i} \to N,
\end{equation}
with the first map the bundle map in the colimit diagram
of \eqref{eq:colimPi}, and the second map the
bundle map in the colimit diagram of
\eqref{eq:colimitN}.  The composition 
\eqref{eq:composition} gives a bundle map over 
$\Sigma _{a}$. And so, by the
defining universal property of the pull-back, there
is a uniquely induced universal map  $$P _{\Sigma _{a}} \to  (\Sigma _{a}) ^{*} N = N ^{\D} _{\Sigma _{a}}, $$    which is a $G$-bundle isomorphism. 

%

Also,
\begin{equation*}
P _{\Sigma } = \colim _{C} P \circ \Delta
^{sm} \Sigma  \circ F.
\end{equation*}
Similarly to the above discussion we have that for all $a \in C$:
\begin{equation*}
P \circ \Delta ^{sm} \Sigma  \circ F (a) = P _{\Sigma _{a}
\circ m _{a} \circ \sigma _{a}},
\end{equation*}
and by functoriality of $P$ there is a morphism:
\begin{equation*}
P (m _{a} \circ \sigma _{a}): P _{\Sigma _{a} \circ m _{a}
\circ \sigma _{a}} \to P _{\Sigma _{a}},
\end{equation*}
over $m _{a} \circ \sigma  _{a}$ and hence an induced
natural morphism:
\begin{equation*}
P _{\Sigma _{a} \circ m _{a} \circ \sigma _{a}} \to (m _{a}
\circ \sigma _{a}) ^{*} P _{\Sigma _{a}},
\end{equation*}
which is also a $G$-bundle isomorphism.

To summarize, we obtain for all $a \in C$ a natural isomorphism 
\begin{equation*}
N \circ \Delta ^{sm} \Sigma  \circ F (a) \xrightarrow{\phi
_{a}} P \circ \Delta ^{sm} \Sigma  \circ F (a).
\end{equation*}
These fit into a natural transformation of functors:
$$\phi: N \circ \Delta ^{sm} \Sigma  \circ F 
\to P \circ \Delta ^{sm} \Sigma  \circ F.$$ So that $\phi$ induces a map
of the colimits: 
\begin{equation*}
 h _{\Sigma}:  P _{\Sigma} \to  N ^{\Delta}_{\Sigma},
\end{equation*}
by naturality,
and this is an isomorphism of these smooth
$G$-bundles.
It is then clear that $\{h _{\Sigma}\} _{\Sigma}$ determines
the bundle isomorphism $h: P \to N ^{\Delta} $  we are
looking for.
%
\end{proof}
\subsection {Pullbacks of simplicial bundles}
\label{sec:pullbacksimplicialbundle}
Let $P \to X$  be a simplicial $G$-bundle over a
smooth simplicial set $X$. And let $f: Y \to X$
be a smooth map of smooth simplicial sets.
We define the pull-back simplicial $G$-bundle $f
^{*}P \to Y$ to be the functor $f ^{*}P := P \circ
\Delta ^{sm}f$.

Note that the analogue of the following lemma
is not true in the category of topological
fibrations. The pull-back by the composition
is not the composition of pullbacks (except
up to
a natural isomorphism).
\begin{lemma}
   \label{lemma:pullbackisfunctorial}
The pull-back is functorial. So that if $f: X
   \to Y$  and $g: Y \to Z$  are smooth maps of
   smooth simplicial sets, and $P \to Z$ is a
   smooth simplicial $G$-bundle over $Z$   then
   $$(g \circ f) ^{*}P = f ^{*} (g ^{*} (P)) \text{ this is an actual equality.} $$ 
\end{lemma}
\begin{proof}
  This is of course elementary, as functor
   composition is associative: $$(g \circ f) ^{*}P =
   P \circ \Delta ^{sm}(g \circ f) = P \circ
   (\Delta ^{sm}g \circ \Delta ^{sm}f) = (P \circ
   \Delta ^{sm} g) \circ \Delta ^{sm}f = f
   ^{*} (g ^{*} P). $$ 
   \end{proof}

\section{Connections on simplicial $G$-bundles} \label{sec:Connections}
In this section $G$ is a generalized Lie group.
A $G$-connection on a smooth $G$-bundle $P$ over a finite
dimensional smooth manifold $X$ is an Ehresmann $G$-connection,
that is a smooth, right
$G$-invariant horizontal distribution. Existence of such
connections is proved as in the case of finite dimensional
bundles. One notes that in trivializations connections form
an affine space, and then uses partitions of unity over the base,
\footnote{This is where we used that $X$ is finite
dimensional, this can likely be relaxed.}
see for instance ~\cite{cite_completeconnectionsonfiberbundles}.

In our setting, we only need to treat the case of $G$-bundles 
$P$ over a simplex $\Delta^{n} $. As such a $G$-bundle is
trivializable, the space of $G$-connections on $P$ is in
correspondence with the space of
Lie algebra valued 1-forms. The regularity condition also ensures that there is
a good theory of parallel transport, see Section
\ref{sec_Curvature 2-form}. Thus, we can completely avoid the
generalities. 
\begin{definition}  A \textbf{\emph{simplicial $G$-connection}} $D$ on a simplicial $G$-bundle $P$ over a smooth simplicial set $X$ is for each $d$-simplex $\Sigma$ of $X$, 
a smooth $G$-invariant Ehresmann $G$-connection $D
(\Sigma ) = D _{\Sigma} $ on $P _{\Sigma} $. 
%
%
%
This data is required to satisfy: if $f: \Sigma _{1} \to \Sigma _{2}  $ is a morphism in $\D (X)$  then  
\begin{equation} \label{equation_functorialityDefConnection}
P (f) ^{*} D _{\Sigma _{2} } = D
   _{\Sigma _{1} }.    
\end{equation}
We say that $D$ is \textbf{\emph{coherent}} if the
   same holds for
all morphisms $f: \Sigma _{1} \to \Sigma _{2}$ in
   $\D ^{sm} (X) $.  
We will often say $G$-connection instead of
simplicial $G$-connection, where there is no need to disambiguate.
\end{definition}
As with differential forms the coherence condition
is very restrictive, and is not part of the basic
definition. 

Let $P \to X$ be a simplicial $G$-bundle. Define $P \times
I \to X \times I$, for $I := [0,1] _{\bullet} $, 
to be the simplicial $G$-bundle $pr ^{*} P $,  for $pr:
X \times I \to X$ the natural projection.
\begin{lemma} \label{lemma:connections:exist}
$G$-connections on simplicial $G$-bundles exist
and any pair of $G$-connections $D _{1}, D _{2}  $
on a simplicial $G$-bundle $P$ are
\textbf{\emph{concordant}}. The latter means that there is
a $G$-connection on $\widetilde{D} $ on $P\times I \to
X \times I$, which restricts to $D _{1}, D _{2}  $ on $P \times I _{0} $, respectively on $P \times I _{1} $,  for $I _{0}, I _{1}  \subset I$ denoting the images of the two end point inclusions $\Delta ^{0} _{\bullet} \to I$.
\end{lemma}
\begin{proof} Suppose that $\Sigma: \Delta^{d}
   _{simp} \to X$ is a degeneracy of a 0-simplex
   $\Sigma _{0}: \Delta^{0} _{simp} \to X$,
   meaning that there is a morphism from $\Sigma$ to
$\Sigma _{0}$ in $\Delta (X) $.      Then $P
   _{\Sigma}= \Delta^{d} \times P _{\Sigma _{0}}$
   (as previously equality indicates natural
   isomorphism) 
   and we fix the corresponding trivial connection
   $D _{\Sigma}$ on $P _{\Sigma }$. This assignment satisfies the
	 condition that for all morphisms $m: \Sigma _{1} \to
	 \Sigma _{2}$ in $\Delta{(X)}
	 $, for $\Sigma _{1}, \Sigma _{2}$ degeneracies of
	 $0$-simplices, $D _{\Sigma _{1}} = P (m) ^{*} D _{\Sigma
	 _{2}}$.
	 We then proceed inductively.


Suppose we have constructed connections $D _{\Sigma }$ for
all $k$-simplices, $0 \leq k \leq n$, and all their
degeneracies, satisfying
the condition $S (n)$, which is as follows. For all morphisms $m: \Sigma _{1} \to
\Sigma _{2}$ in $\Delta(X) $, for $\Sigma _{1}, \Sigma
_{2}$ $k$-simplices or their degeneracies with $0 \leq k \leq n$, $D _{\Sigma _{1}}
= P (m) ^{*} D _{\Sigma _{2}}$. We construct an extension $D
_{\Sigma }$ for all $(n+1)$-simplices and their
degeneracies, so that this extension satisfies $S (n+1)$.

If $\Sigma$ is a non-degenerate
$(n+1)$-simplex then $D _{\Sigma}$ is already
determined over the boundary of $\Delta^{n+1}
$ by the defining condition
\eqref{equation_functorialityDefConnection}. For by the
hypothesis, $D _{\Sigma}$ is already
defined on all $n$-simplices. Then extend $D _{\Sigma}$ over
all of $\Delta^{n+1} $ arbitrarily (since the corresponding
bundle is trivializable this amounts to choosing a smooth extension of
a Lie algebra valued 1-form). Given a non-identity morphism of
non-degenerate $k$-simplices $m: \Sigma _{0} \to \Sigma $,
$0 \leq k \leq n+1$, $\degree \Sigma _{0} < n+1$ and hence $m$
maps to the boundary of $\Sigma $, i.e. to a subsimplex of
degree $n$ or less and hence by the inductive
hypothesis we have that $P (m) ^{*} D _{\Sigma
} = D _{\Sigma }$.

Thus, we have extended $D _{\Sigma}$  to all   
$(n+1)$-simplices, as such a simplex is either
non-degenerate or is a degeneracy of an
$n$-simplex, and in the latter case $D _{\Sigma }$  is
defined by the hypothesis. 

Now, suppose we have a degeneracy $mor: \Sigma ^{m} \to \Sigma
^{k}$, $k < m$, $k \leq n+1$ ($\Sigma ^{k}$ can itself be degenerate). Then we have bundle map:
\begin{equation*}
P (mor): P (\Sigma ^{m}) \to P (\Sigma ^{k}).
\end{equation*}
And we define $D _{\Sigma ^{m}} = P (mor) ^{*} D _{\Sigma ^{k}}$. 
The property $S (n)$ ensures that this is well defined.
And so we have constructed an
assignment $D _{\Sigma }$ for all degeneracies of
$(n+1)$-simplices. By construction this satisfies $S (n+1)$.
And so we have completed the inductive step.

The second part of the lemma follows by an
analogous argument, since we may extend $D
_{1}, D _{2}  $ to a concordance connection
$\widetilde{D}$, using the above inductive
procedure.


\end{proof}
\begin{example} \label{example:pullbackconnection}
Given a smooth $G$-connection $D$ on a
smooth principal $G$-bundle $Z \to Y$, we naturally get a
simplicial $G$-connection on the induced simplicial
$G$-bundle $ Z ^{\D}$. Concretely, this is
defined by setting $D
_{{\Sigma}}$  on $ Z ^{\D} _{\Sigma} = \Sigma ^{*} Z $
to be $\widetilde{\Sigma}^{*}D$, for
$\widetilde{\Sigma}: \Sigma ^{*} Z \to Z $  the
natural map (in the pull-back diagram). The pull-back
$\widetilde{\Sigma }^{*}D$, is the pre-image by
$ \widetilde{\Sigma}$ of the corresponding distribution.
This is called the
\textbf{\emph{induced simplicial connection}},
and it will be denoted by $D ^{\D}$. 
Going in the other direction is always possible if
   the given simplicial $G$-connection in addition satisfies
   coherence,  but we will not elaborate.
\end{example}

\section {Chern-Weil homomorphism}
\label{sec:Chern-Weil} 
\subsection {The classical case}
\label{sec:classicalChernWeil}
To establish notation we first
discuss the standard Chern-Weil homomorphism. In this section
again $G$ will be a generalized Lie group.
Let $\mathfrak g $ denote its
Lie algebra. Let $P$ be a $G$-bundle over
a smooth finite dimensional manifold $Y$.
Fix a $G$-connection $D$ on $P$.
\subsubsection{Curvature 2-form} \label{sec_Curvature 2-form}
Associated to $D$ we have the
curvature 2-form $R ^{D}$ on $Y$,
understood as a 2-form valued in the vector bundle
$\mathcal{P} \to Y $, whose fiber over $y \in Y$ is   
$\mathcal{R} (P _{y})$ - the Lie algebra of right $G$-invariant
vector fields on $P _{y}$. 
In the infinite dimensional setting the definition of this
2-form is more subtle. 
\begin{remark} \label{rem_curvatureform}
No general reference is
known to me. But it might be possible to adapt the finite
dimensional approach of Kobayashi-Nomizu
~\cite{cite_bookKobayashiNomizuVol1},
\cite{cite_bookKobayashiNomizuVol2},
to the locally convex infinite dimensional setting. The
potential difficulty may be in differential geometric
details like the Bianchi identity in infinite dimensions.
The approach below relies on regularity, but on the other
hand it is intuitive, and there is no differential geometry
just calculus. It is also, at least implicitly, the approach
one takes in symplectic geometry for curvature of Hamiltonian fibrations,
see
~\cite[Section
6.4]{cite_McDuffSalamonIntroductiontosymplectictopology}. 
%
%
\end{remark}

Suppose first we
have a trivial bundle $(U \subset \mathbb{R} ^{d} ) \times
G$, with $0 \in U$, $U$ open. Define a smooth $\mathfrak g$ valued one form
$\alpha ^{D}$ on $U$ by:
\begin{equation*}
\alpha ^{D} (v) = pr _{G} (\widetilde{v}),
\end{equation*}
where $\widetilde{v}$ is the $D$-horizontal lift of $v$, 
$pr _{G}: (U \subset \mathbb{R} ^{d}) \times G \to G  $ is
the projection, and where we identify $\mathfrak g$ with the
space of right invariant vector fields.

Then for any smooth path $\gamma: [0,1] \to U$, we get
a smooth path in $\mathfrak g$, $$t \mapsto \xi _{t} = \alpha
^{D} (\gamma' (t)).$$ 

We use the defining property of the regularity of $G$ to find the unique smooth solution
curve $\widetilde{\gamma}: [0,1] \to G$ satisfying:
\begin{equation*}
\widetilde{\gamma}' (t)  = \xi _{t} (\gamma
_{t}).
\end{equation*}
Set $\phi ^{D} ({\gamma}) = \widetilde{\gamma} (1)$. 

This determines a map, called the holonomy map, on the
smooth based loop space:
\begin{align*}
& Hol ^{D}:  \Omega _{p} U \to G, \\
& Hol ^{D} (\gamma) = \phi ^{D} ({\gamma}).
\end{align*}
The regularity of $G$ gives that $Hol ^{D}$ is smooth,
taking the standard Fr\'echet manifold structure on $\Omega _{p} U$.

Now, for $v, w \in T _{0} U $ and $h,k \in [0,1) $ let $\gamma
_{hv, hw} \in \Omega _{0} U$, parametrize the oriented boundary of the
parallelogram,
three of whose vertices are $0, hv, hw$, where the
orientation on the parallelogram is $\{v,w\}$.  (We need to perturb
the natural piecewise linear parametrization 
to be smooth, with the same image, basically making it constant near corners of the parallelogram.)


Then the mapping $(h,k)
\mapsto Hol ^{D} (\gamma _{hv, hw}) \in G$ is smooth (it is
well defined with respect to any choice of the smoothing
mentioned above), and of
course $(0,0) \mapsto e$ (the unit).
And so we may set:   
\begin{equation} \label{def_RD}
 \forall v,w \in T _{0} U: R ^{D} _{0}
 (v,w) = \frac{\partial }{\partial h \partial k} \vert 
 _{(0,0)} Hol ^{D} (\gamma _{hv, hw}) \in \mathfrak g,
\end{equation}
where $\mathfrak g$ is now understood as $T _{e} G$. 
Define $R ^{D} _{p}$ analogously at other points of $p \in U$.

Given a $G$-bundle map $f: T	_{0} \to T _{1} $ of
$G$-bundles over a point let $$\operatorname {lie} f: \mathcal{R} (T _{0}) \to
\mathcal{R} (T _{1})$$ denote the induced mapping of the Lie
algebras. Note that for $T _{0} = T _{1} = G$ any such map
$f$ uniquely corresponds to left multiplication by some $g \in
G$.
\begin{lemma} \label{lem_} Let $\widetilde{f}: U \times G \to U \times
G$ be a smooth $G$-bundle map over $f: U  \to U$. Let
$\widetilde{f} _{p}$ denote the 
restriction of $\widetilde{f}$ to the map of the
fiber over $p$ to the fiber over $f (p)$.  
And let $D'= \widetilde{f} ^{*}D$ denote the pull-back
connection, then for $v,w \in T _{p} U$: 
\begin{align} \label{eq_pullback}
(f ^{*} R ^{D}) _{p}  & = \operatorname {lie}
\widetilde{f} _{p} \circ R ^{D'} _{p} \\
& = Ad _{g _{p}} (R ^{D'} _{p}),   
\end{align}
where $g _{p} \in G$ corresponds to $\widetilde{f} _{p}$
as above, and 
where $Ad _{g}$ denotes the adjoint action by $g$
\footnote{As it is natural here
to work with right invariant vector fields, we define the
adjoint action on the Lie algebra as the left $G$ action on 
right invariant vector fields.}.
\end{lemma}
\begin{proof} [Proof] 
For simplicity suppose that $p=0$ and $f (0) = 0$.
By definition of the pull-back connection we have:
\begin{equation*}
Hol ^{D'} (\gamma _{hv, hw})  = \widetilde {f} _{p} ^{-1} \circ
Hol ^{D} (f \circ \gamma _{hv, kw}).
\end{equation*}
So \begin{align*}
R ^{D'} _{0} (v,w) & = \frac{\partial }{\partial h \partial k}  \vert 
 _{(0,0)} Hol ^{D'} (\gamma _{hv, kw}) \\ 
 & = \operatorname
 {lie} \widetilde{f} _{p} ^{-1} (\frac{\partial }{\partial
 h \partial k} \vert _{(0,0)} Hol ^{D} (f \circ \gamma _{hv, kw}) \\ 
 & = \operatorname
 {lie} \widetilde{f} _{p} ^{-1} (\frac{\partial }{\partial
 h \partial k} \vert _{(0,0)} Hol ^{D} (\gamma _{h f_
 *v, kf _{*}w})  \\
 & = \operatorname {lie} \widetilde{f} _{p} ^{-1}
 (R ^{D} _{0} (f _{*}v,f _{*}w)) \\
 & = \operatorname {lie} \widetilde{f} _{p} ^{-1} (f ^{*}
 R ^{D} _{0} (v,w)).
\end{align*}
Here the third equality from the top is obtained as follows.
We have the composition:
\begin{equation*}
(V \subset  \mathbb{R} ^{} \times \mathbb{R} ^{}) \to \Omega _{0}
U \xrightarrow{\Omega f} \Omega _{0} U \xrightarrow{Hol
^{D}} G,
\end{equation*}
where $V \ni (0,0) $ is open, $\Omega f: \Omega _{0} U \to \Omega _{0} U$ is just
the map $\gamma \mapsto f \circ \gamma$, and the first map
is the map $(h,k) \mapsto \gamma _{hv, kw}$.
Note that the differential $D (\Omega f) _{0}: \Omega _{0} T _{0} U \to \Omega
_{0} T	_{0} U$, at the constant loop at zero (just
denoted as $0$), is the map $\eta \mapsto Df _{0} \circ
\eta$, i.e. it is the map $\Omega Df _{0}$. Then apply chain rule to this
composition.

\end{proof}
The curvature 2-form $R ^{D}$ on $Y$ is then defined as
follows. Let $\widetilde{f}:
U \times G \to P$, be a smooth $G$-bundle parametrization map
over a parametrization $f: U \to Y$.
Then define $R ^{D}$ by the condition:
$$\forall \widetilde{f}: \; (f ^{*}R ^{D}) _{p} = \operatorname {lie}
\widetilde{f} _{p} \circ R ^{\widetilde{f} ^{*} D} _{p}.  $$
It is an elementary verification using the lemma above that $R ^{D}$ is well
defined. 

\subsubsection{The algebra $\mathcal {I} (G)$ and the
homomorphism for smooth $G$-bundles}
\label{sec_The algebraI}

We denote by $\mathcal I ^{d} (G)$ the space of continuous,
symmetric multilinear functionals 
$$\prod_{i=1} ^{i=d}  \mathfrak g \to \mathbb{R} ^{}, $$
we will just call $d$-tensors, fixed by the adjoint $G$
action. Meaning that for $\rho \in \mathcal{I} ^{k} (G) $:
$$ \rho (Ad _{g}(\xi _{1}), \ldots, Ad _{g}(\xi _{k}) ) = \rho (\xi _{1}, \ldots, \xi _{k}), \quad \forall g \in G, \,
\xi _{i} \in  \mathfrak g.$$
And set $$\mathcal{I} (G) = \bigoplus _{d \geq 0} \mathcal{I}
^{d}(G).$$ This forms an algebra under the symmetric product, see for instance ~\cite[Chapter 12]{cite_bookKobayashiNomizuVol2}. 

As mentioned in the
introduction, in the infinite dimensional
setting $\mathcal{I} (G)$ may not be freely generated,
and is possibly very intricate
algebraically.

Now, let $\rho \in \mathcal{I} ^{k} (G)$.
As $\rho$ is $Ad$ invariant, it uniquely determines  
a multilinear map with the same name: $$\rho: \prod_{i=1} ^{i=k}
\mathcal{R} (P _{y}) \to \mathbb{R},$$
by taking any $G$-bundle map over a point $P
_{y} \to G$ and pulling back the functionals.  We may
now define a closed $\mathbb{R}$-valued
$2k$-form $\omega ^{\rho,D}$ on $Y$: 
\begin{equation} \label{eq:omegaclassical}
   \omega ^{\rho,D} (v _{1}, \ldots, v _{2k}) =
   \frac{1}{2k!} \sum _{\eta \in {S}
   _{2k}} \sign \eta \cdot \rho (R ^{D} (v _{\eta (1)}, v _{\eta ({2})}),
\ldots, R ^{D}(v _{\eta
(2k-1)}, v _{\eta _{2k}})),
\end{equation} 
for ${S} _{2k} $ the permutation group of a set
with $2k$ elements, and where $v _{1}, \ldots, v
_{2k} \in T _{y}Y$.    

Set $$cw ^{P,D} (\rho) = \omega ^{\rho,D}.$$ 

In this way we get a $dg$ map:
\begin{equation*}
cw ^{P,D}: \mathcal {I} (G) \to \Omega
^{\bullet } (Y, \mathbb{R} ^{} ).
\end{equation*}

Set $$\alpha ^{\rho,D}:=
\int \omega ^{\rho,D} \in C ^{2k} (Y, \mathbb{R} ^{} ). $$ 
Then we define the Chern-Weil characteristic class:  
\begin{equation}
   \label{eq:classicalChernWeil}
  c ^{\rho} (P) := [\alpha ^{\rho,D}] \in H ^{2k} (Y,
\mathbb{R}). 
\end{equation}

\subsection {Chern-Weil homomorphism for 
simplicial $G$-bundles}  
Now let $P$ be a simplicial $G$-bundle over a smooth
simplicial set $X$.
Fix a simplicial $G$-connection $D$ on $P$. 

For each simplex $\Sigma ^{d}$, we have the curvature 2-form $R ^{D} _{\Sigma}$ of
the connection $D_{\Sigma} $ on $P _{\Sigma}$,
defined as in the section just above. 
For concreteness: 
\begin{equation*}
 \forall v,w \in T _{z} \Delta^{d}: R ^{D}
_{\Sigma} (v,w)  \in \mathcal{R}  (P _{y}),
\end{equation*}
for $P _{z}$  the fiber of $P _{\Sigma}$ over $z
\in \Delta^{d} $. 

As in the previous section, let $\rho \in \mathcal{I} ^{k}
(G)$. 
We may now define a closed, $\mathbb{R}$-valued, 
simplicial differential $2k$-form $\omega ^{\rho,D}$ on $X$: 
\begin{equation*}
   \omega ^{\rho,D} _{\Sigma}  (v _{1}, \ldots, v _{2k}) =
   \frac{1}{2k!} \sum _{\eta \in {S}
   _{2k}} \sign \eta \cdot \rho (R ^{D} _{\Sigma }  (v _{\eta (1)}, v _{\eta ({2})}),
\ldots, R ^{D} _{\Sigma}(v _{\eta
(2k-1)}, v _{\eta _{2k}})).
\end{equation*} 
Set $cw ^{P,D} (\rho) = \omega ^{\rho,D}$.
This defines a $dg$ map:
\begin{equation*}
cw ^{P,D}: \mathcal {I} (G) \to \Omega
^{\bullet } (X, \mathbb{R} ^{} ).
\end{equation*}

\begin{definition}\label{def_dghomotopy}
Let 
\begin{equation*}
f _{i}: A \to \Omega ^{\bullet} (X, \mathbb{R}), \, i=0,1
\end{equation*}
be $dg$ maps of differential graded $\mathbb{R} ^{} $ algebras where $X$ is
a simplicial set (we work over $\mathbb{R}$ for
simplicity). We say
that $f _{i}$ are \textbf{geometrically homotopic}  if
there is a $dg$ map: $$\widetilde{f}: A \to \Omega
^{\bullet} (X \times I, \mathbb{R}),$$ 
satisfying $e _{0} \circ \widetilde{f} = f _{0}$, and $e
_{1} \circ \widetilde{f} = f _{1}$, where $e _{i}: \Omega
^{\bullet} (X \times I, \mathbb{R}) \to \Omega ^{\bullet} (X)$ are the $dg$ maps induced by
the end point inclusions $pt \to \Delta^{1} $.
\end{definition}
It is not hard to see that being geometrically homotopic is an
equivalence relation, but we will not need this.
%

\begin{lemma} \label{lemma:equality:chernweilclass}
For $P \to X$ as above $$cw ^{P,D _{0}} \simeq cw ^{P,D _{1}},$$ for any pair of simplicial $G$-connections $D _{0},D _{1}$ on $P$, where
$\simeq$ is the geometric homotopy relation.
The latter homotopy can be made to depend solely
on the choice of a concordance connection $\widetilde{D}$.
\end{lemma}
\begin{proof} 

For $D _{0},D _{1}$ as in the statement, fix
a concordance simplicial $G$-connection $\widetilde{D} $,
between $D _{0},D _{1}$, on the simplicial $G$-bundle $P \times I  \to X \times I$, 
as in Lemma \ref{lemma:connections:exist}. 

We have a diagram of $dg$ maps:
\begin{equation*}
\begin{tikzcd}
 \mathcal I (G) \ar[r, "cw
 ^{P,\widetilde{D}}"]   &  \Omega ^{\bullet} (X \times I,
 \mathbb{R} ^{}) \ar [r, "r _{i}"] & \Omega
^{\bullet} (X, \mathbb{R} ^{} ),
\end{tikzcd}
\end{equation*}
where 
\begin{itemize}
\item ${r} _{i}$ are the restriction maps corresponding to
the pair of natural inclusions $X \to X \times I$.
\item $r _{i} \circ cw ^{P,\widetilde{D}} = cw ^{P,
D _{i}}$, $i=0,1$.
\end{itemize}
\end{proof}
\begin{definition}\label{def_AinftyHomotopy}
Let 
\begin{equation*}
f _{i}: A \to B, \, i=0,1
\end{equation*}
be $dg$ maps of differential graded $\mathbb{R} ^{} $ algebras $A,B$. We say
that $f _{i}$ are \textbf{\emph{$A _{\infty}$ homotopic}}  if there is an $A
_{\infty}$ map: $$\widetilde{f}: A \to B \otimes \Omega
^{\bullet} (I, \mathbb{R} ^{}),$$ 
satisfying $\widetilde{e} _{0} \circ \widetilde{f}
= f _{0}$, and $\widetilde{e}
_{1} \circ \widetilde{f} = f _{1}$, where $\widetilde{e} _{i}: B \otimes
\Omega ^{\bullet} (I) \to B$ are the natural $dg$ maps
induced by the end point inclusions $pt \to I$, formally
defined in the proof of Proposition
\ref{prp_ainftyhomotopy}.
\end{definition}

A geometric homotopy induces an $A _{\infty}$ homotopy of
$dg$ maps. We relegate this to the Appendix
\ref{appendix_Ainfhomotopy}.

Set $$\alpha ^{\rho,D}:= \int \omega ^{\rho,D} \in C ^{2k}
(X, \mathbb{R} ^{} ). $$
Then in particular, the cohomology class:  
$$c ^{\rho} (P) := [\alpha ^{\rho,D}] \in H ^{2k} (X,
\mathbb{R}),
$$
is well defined, and this is called the Chern-Weil characteristic class.

\begin{notation}
\label{notation:represent} Let us denote by $cw ^{P}$ any
representative of the homotopy class $[cw ^{P,D}]$. (For smooth
or simplicial $G$-bundles $P$.)
\end{notation}
We have the expected naturality:
\begin{lemma}
   \label{lemma:pullbackChernWeil}  
Let $P$ be a simplicial $G$-bundle over $Y$,
$\rho$ as above and $f: X \to Y$ a smooth simplicial map.
Then $$f ^{*} \circ cw ^{P} \simeq  cw ^{f ^{*} P},
$$ where $\simeq $ as before means homotopic.
\end{lemma}
\begin{proof}
  Let $D$ be a simplicial $G$-connection on $P$. 
   Define the pull-back connection $f ^{*}D$ on  
   $f ^{*}P$ by $f ^{*} D ({\Sigma}) = D _{f 
   (\Sigma) } $.   
  Then $f ^{*}D$ is a simplicial $G$-connection on
   $f ^{*}P$. Now,  
\begin{align*}
	\forall \Sigma: \omega ^{\rho, f ^{*} D} (\Sigma) & = 
  \omega ^{\rho, D} (f(\Sigma)), \text{ by definition of $f
	^{*}D$} \\ 
	& = f ^{*} \omega ^{\rho,D}(\Sigma),
	\text{ definition \eqref{equation_pullback}}.
	\end{align*}	 
And consequently, $\omega ^{\rho, f ^{*} D} = f ^{*} \omega
^{\rho,D}$. It follows that:
\begin{equation*}
f ^{*} \circ cw	^{P,D} = cw ^{f ^{*}{P}, f ^{*}D }.
\end{equation*}
The result then readily follows by Lemma
\ref{lemma:equality:chernweilclass}.
\end{proof}

\begin{proposition} \label{prop:expected}   
Let $G \hookrightarrow Z \to Y$ be an ordinary smooth
principal $G$-bundle, and $\rho$ as above. Let $Z ^{\D}$  be the induced simplicial $G$-bundle over $Y _{\bullet}$  as in Example \ref{ex:inducedbundle}.
Then: 
\begin{enumerate}
\item  The form $\omega ^{\rho,D ^{\D}}$ is the simplicial differential form induced by
$\omega ^{\rho,D}$, where induced is as in Example
\ref{exampleInducedForm}. In particular, $$cw ^{Z ^{\Delta^{}
}} \simeq  \Theta \circ cw ^{Z}, $$ where $\Theta $ is as in
\eqref{equation_G}.
\item If $c ^{\rho} (Z) \in H ^{2k}
(Y, \mathbb{R}) $ is the Chern-Weil
characteristic class as in \eqref{eq:classicalChernWeil}, then   
\begin{equation} \label{eq:5}
 |c ^{\rho}  (Z ^{\D})| _{sm} =  c ^{\rho} (Z),
\end{equation} 
where $|c ^{\rho} (Z ^{\D})| _{sm}$ is as
in part \ref{notation:geomRealizationClasspart2} of Notation \ref{notation:geomRealizationClass}.
\end{enumerate}
\end{proposition}
\begin{proof} 
Fix a smooth $G$-connection $D$ on $Z$.  This induces
a simplicial $G$-connection $D ^{\D}$ on  $Z ^{\D}$, as in Example \ref{example:pullbackconnection}. Let $\omega ^{\rho,D} $ denote the smooth Chern-Weil differential $2k$-form on $Y$, as in \eqref{eq:omegaclassical}.

Now, 
\begin{align*}
	\forall \Sigma: \omega ^{\rho,D ^{\D}} (\Sigma) & = \omega ^{\rho,
	\widetilde{\Sigma} ^{*}D} \text{ by definitions
} \\
& = \Sigma ^{*} \omega ^{\rho, D} \text{ by standard naturality of
Chern-Weil forms}. 
\end{align*}
So we obtain the first part of the Proposition. 

Let $\alpha ^{\rho,D}  = \int \omega ^{\rho,D} \in C ^{2k} (Y, \mathbb{R}) $.  
It readily follows by Lemma \ref{lemma_Iveetrivial} that:
\begin{equation*}
 |c ^{\rho}  (Z ^{\D})| _{sm} = (I ^{\vee }) ^{-1} ([\alpha ^{\rho,D ^{\D}}]) = [\alpha ^{\rho,D}] = c ^{\rho} (Z),  
\end{equation*} 
where $I ^{\vee }$ is as in \eqref{eq:RlowerstarSpaceSmooth2}. 
\end{proof}

\section{The universal simplicial $G$-bundle}
\label{sec_universalsimplicialGbundle}
Briefly, a Grothendieck universe is a set $\mathcal{U}$ forming a model for set theory. 
That is if we interpret all terms of set theory
as elements of $\mathcal{U} $, then all the set
theoretic constructions keep us within
$\mathcal{U} $.
We will assume Grothendieck's axiom of universes which says
that for any (pure) set $X$  there is a Grothendieck
universe $\mathcal{U} \ni X$. Intuitively, such a
universe $\mathcal{U} $ is formed by taking
all possible set theoretic constructions starting with $X$.  For example if $\mathcal{P} (X) $ denotes the power set of $X$, then
$\mathcal{P  } (X)  \in \mathcal{U} $. 
Note that this axiom is beyond $ZFC$, and the resulting
axiomatic system is sometimes denoted as $ZFCG$. This is now a
common framework of modern set theory,
especially in the context of category
theory, c.f. ~\cite{cite_LurieHighertopostheory}.
In some contexts one works with universes
implicitly. This is impossible here, as we need to
establish certain universe independence.

Let $G$ be a locally convex Lie group. Let $\mathcal{U}$
be a Grothendieck universe satisfying:
$$G \in \mathcal{U}, \quad \forall n  \in
\mathbb{N}:  \Delta^{n} \in \mathcal{U},
$$ where $\Delta^{n} $ are the usual topological
$n$-simplices.  Such a $\mathcal{U}$ will be called \textbf{\emph{$G$-admissible}}.

We will construct smooth Kan complexes $BG
^{\mathcal{U}} $ for each $G$-admissible $\mathcal{U}$.
Moreover, we will construct a weak equivalence $|BG
^{\mathcal{U}}| \to BG$ for each $\mathcal{U} $, where $BG$
the standard Milnor classifying space.

If $G$ has the homotopy type of a CW complex, then $BG$ has the homotopy 
type of a CW complex.  In particular, by Whitehead's theorem
the homotopy type of $|BG ^{\mathcal{U}}| $ is
the independent of $\mathcal{U} $, and in fact $|BG
^{\mathcal{U}}| $ is $BG$ up to homotopy.


\begin{definition} \label{def:usmall}
A \textbf{\emph{$\mathcal{U}$-small set}} is  an
element of $\mathcal{U}$. For $X$  a smooth
simplicial set, a smooth simplicial
$G$-bundle $P \to X$ will be called 
\textbf{\emph{$\mathcal{U}$-small}} if for each $n$ $X (n)
$ is $\mathcal{U} $-small, and
for each 
simplex $\Sigma$ of $X$ the bundle $P _{\Sigma} 
$ is $\mathcal{U}$-small.  
\end{definition}

\subsection {The classifying spaces $BG
^{\mathcal{U} }$} \label{sec:classifyingspace} Let
$\mathcal{U} $ be $G$-admissible.
We define  a simplicial set $BG ^{\mathcal{U}}$,
whose set of $k$-simplices  
$BG ^{\mathcal{U} } (k) $ is the set 
of  $\mathcal{U}$-small smooth simplicial
$G$-bundles over $\D ^{k} _{\bullet} $. 
The simplicial maps are defined by
pull-back so that given a map $i \in hom _{\D} (m, n) $ the map $$BG ^{\mathcal{U}} (i): BG ^{\mathcal{U}} (n) \to BG ^{\mathcal{U}} (m)  $$ is the natural pull-back:
\begin{equation*}
   BG ^{\mathcal{U}}(i) (P) =  i _{\bullet
   } ^{*}P,
\end{equation*}
for $i _{\bullet}$, the induced map $i _{\bullet}:
\D ^{m} _{\bullet } \to \D ^{n} _{\bullet}  $, $P
\in BG ^{\mathcal{U} } (n) $ a simplicial
$G$-bundle over $\Delta^{n} _{\bullet }$,  and
where the pull-back map $i _{\bullet} ^{*}$  is as in Section
\ref{sec:pullbacksimplicialbundle}.  Then Lemma
\ref{lemma:pullbackisfunctorial} ensures that $BG
^{\mathcal{U}}: \Delta^{op} \to s-Set $ is a functor, 
so that we get a simplicial set $BG ^{\mathcal{U}}$.  

We define a smooth simplicial set structure $g$ on $BG
^{\mathcal{U}}$ as follows. Given a $d$-simplex
$P \in BG ^{\mathcal{U}} (d) $ the induced map  
$$(g (P) = P _{*}): \D ^{d}
_{\bullet}  \to BG ^{\mathcal{U}}, $$ is defined
naturally by 
\begin{equation} \label{equation_Pstar}
P _{*} (\sigma) := \sigma ^{*} _{\bullet} P.
\end{equation}
where $P$ on the right is the initial simplicial
$G$-bundle $P \to \Delta^{d} _{\bullet
}$.  
More explicitly,  $\sigma \in \D ^{d} _{\bullet}
(k) $
is a smooth map $\sigma: \Delta ^{k} \to
\Delta^{d}  $, $\sigma _{\bullet}: \Delta ^{k}
_{\bullet} \to \Delta ^{d}_{\bullet }$  denotes
the induced map and the pull-back is as previously 
defined.
We need to check the push-forward functoriality
Axiom \ref{axiom:pushforward}.

Let $\sigma \in  \Delta^{d} _{\bullet}
(k) $, then for all $j \in \mathbb{N}, \rho \in 
\Delta^{k} _{\bullet} (j)$:
\begin{align*}
   (P _* (\sigma)) _{*} (\rho) & = (\sigma ^{*} 
    _{\bullet} P) _{*} (\rho)   \\
    &  = \rho ^{*} _{\bullet} (\sigma ^{*} 
    _{\bullet} P), \text{ by definition of $g$}. \\
\end{align*} 
And
\begin{align*}
   P _{*} \circ \sigma
   _{\bullet} (\rho) & = (\sigma _{\bullet} 
   (\rho)) _{\bullet} ^{*} P \\ 
   & = (\sigma _{\bullet} 
    \circ \rho _{\bullet }) ^{*}P,  \text{ as 
    $\sigma _{\bullet}$ is smooth }   \\
   & = \rho ^{*} _{\bullet} (\sigma ^{*} 
    _{\bullet} P).  \\
\end{align*}
And so 
\begin{equation*}
    (P _* (\sigma)) _{*} =  P _{*} \circ \sigma
   _{\bullet},
\end{equation*}
so that $BG ^{\mathcal{U} }$   is indeed a smooth
simplicial set.
%
%

\subsection {The universal smooth simplicial
$G$-bundle $EG ^{\mathcal{U}} \to BG ^{\mathcal{U} }$}
To abbreviate already complex notation, in what follows $V$ denotes $BG ^{\mathcal{U}}$
for a general, $G$-admissible
$\mathcal{U}$.  There is a tautological functor 
\begin{equation} \label{equation_E}
E: \D
^{sm} (V) \to
\mathcal{G}
\end{equation}
that we now describe. 

A smooth map $P: \Delta^{d} _{\bullet }
\to V$, uniquely corresponds to a $d$-simplex $P ^{b}$  of $V$ via
Proposition \ref{prop:smoothsimplices},  
i.e. a simplicial $G$-bundle  $P ^{b} \to  \Delta^{d}
_{\bullet}$. In other words  $P ^{b}$ is the bundle:
\begin{equation} \label{equation_Pbequality}
P
^{b} = P (id ^{d}),
\end{equation}
for $id ^{d}: 
\Delta^{d} \to \Delta^{d} $ the identity, and 
where the equality is an equality of simplicial $G$-bundles, in other words functors.
\begin{notation}
   \label{notation:} Although we disambiguate 
   in the discussion just below, later on we may conflate the notation $P,P ^{b}$ with just $P$. 
\end{notation}
Recalling that $P ^{b}$ is a certain functor
$\Delta ^{sm} (\Delta^{d}_{\bullet}) \to \mathcal{G}$ we then
set: 
$$E (P) = P ^{b} ({id} ^{d} _{\bullet}).$$

We now define the action of $E$ on morphisms. Suppose we have a morphism   
$m \in \Delta^{sm}
(V) $:
\begin{equation*}
\begin{tikzcd}
\Delta^{k} _{\bullet } \ar[r, "\widetilde{m} _{\bullet}"] \ar [rd, "P _{1}"]
   &  \Delta^{d} _{\bullet} \ar [d,"P _{2}"]
   \\
    & V, 
\end{tikzcd}
\end{equation*}  
then we have an equality:
\begin{equation}\label{eq:Pb}
   \begin{split}
      P _{1}^{b} & = P _{1} (id ^{k}) \text{
			\eqref{equation_Pbequality}} \\
			& = (P _{2} \circ 
   \widetilde{m} _{\bullet }) (id ^{k})  \\
   & = P _{2} (\widetilde{m}) \\
   & = (P _{2}^{b}) _{*} (\widetilde{m}), \text{ thinking of
	 $P _{2} ^{b}$ as a simplex of $V$}  \\
   & = P _{2} ^{b} \circ \Delta ^{sm}\widetilde{m} 
   _{\bullet}, \text{ by \eqref{equation_Pstar}}.
   \end{split}
\end{equation}

%
So that
\begin{equation*}
   P _{1} ^{b} (id ^{k} _{\bullet}) = P _{2} ^{b}
   (\widetilde{m} _{\bullet} \circ id^{k} 
   _{\bullet})  = P _{2} ^{b}
   (\widetilde{m} _{\bullet}).  
\end{equation*}
We have a tautological morphism $e_m \in \Delta ^{sm}
(\Delta^{d} _{\bullet}) $ corresponding to the
diagram:
\begin{equation*}
\begin{tikzcd}
\Delta ^{k} _{\bullet} \ar[r, "
{\widetilde{m} _{\bullet} }"] \ar [rd, "\widetilde{m}
   _{\bullet }"] & \Delta^{d} _{\bullet}  \ar
   [d,"id ^{d} _{\bullet}"] \\
   & \Delta ^{d} _{\bullet }.
\end{tikzcd}
\end{equation*}
So we get a smooth $G$-bundle map:
$$P _{2} ^{b} (e _{m}): (E (P _{1}) = P _{2} ^{b}
(\widetilde{m} _{\bullet})) \to (E (P _{2}) = P
_{2} ^{b} (id ^{d} _{\bullet})),   $$ which is over the
smooth map $\widetilde{m}: \Delta^{k} \to
\Delta^{d}  $  induced by $\widetilde{m}
_{\bullet}$. 
And we set $E (m) = P ^{b} _{2} (e _{m}) $. 

We need to
check that with these assignments $E$ is a functor. Suppose we have a
diagram:
\begin{equation*}
 \begin{tikzcd}
    \Delta ^{l} _{\bullet} \ar [r, "\widetilde{m}
    ^{0} _{\bullet }"] \ar [rrd, bend
    right, "P
    _{0}"]  & \Delta^{k} _{\bullet } \ar[r,
    "\widetilde{m} ^{1} _{\bullet}" ] \ar [rd, "P _{1}"]
    &  \Delta^{d} _{\bullet} \ar [d,"P _{2}"] \\
  &   & V. 
 \end{tikzcd}
 \end{equation*}
In other words, we have a diagram for the composition $m = {m
^{1}} \circ {m ^{0}}$ in $\Delta^{sm} (V)$.
Then $e _{m} = e _{m^1} \circ e' _{m ^{0}} $ where
$e' _{m^0}$ is the diagram: 
\begin{equation*}
\begin{tikzcd}
\Delta ^{l} _{\bullet} \ar[r, "
{\widetilde{m} ^{0}_{\bullet} }"] \ar [rd, "\widetilde{m}
   _{\bullet }"] & \Delta^{k} _{\bullet}  \ar
   [d,"\widetilde{m} ^{1}_{\bullet} "] \\
   & \Delta ^{d} _{\bullet }, 
\end{tikzcd}
\end{equation*}
and $e _{m ^{1}}$  is the diagram:
\begin{equation*}
\begin{tikzcd}
\Delta ^{k} _{\bullet} \ar[r, "
{\widetilde{m} ^{1}_{\bullet} }"] \ar [rd,
   "\widetilde{m} ^{1}
   _{\bullet }"] & \Delta^{d} _{\bullet}  \ar
   [d,"id ^{d}_{\bullet} "] \\
   & \Delta ^{d} _{\bullet }. 
\end{tikzcd}
\end{equation*}
So $$E (m) = P _{2} ^{b} (e _{m}) = P _{2} ^{b} (e _{m 
^{1}})
\circ  P _{2} ^{b} (e' _{m ^{0}}) = E (m ^{1}) \circ P _{2} ^{b} (e' _{m ^{0}}). $$
Now, 
\begin{align*}
E (m _{0}) 
& =P ^{b}
_{1}   (e _{m ^{0}}) \\
& = (P _{2} ^{b} \circ \Delta 
^{sm} \widetilde{m} ^{1} _{\bullet}) (e _{m ^{0}}), \text{
analogue of \eqref{eq:Pb}}   \\   
& = P _{2} ^{b} (e' _{m _{0}}). 
\end{align*}
And so we get:
$E (m) = E (m _{1}) \circ E (m _{0}).  $   Thus, $E$ is a functor.


By construction the
functor ${E} $
satisfies the compatibility condition, and hence
determines a simplicial $G$-bundle.

\begin{definition}\label{definition_EG}
Given $G,\mathcal{U} $ as previously, 	the universal
simplicial $G$-bundle $EG ^{\mathcal{U}}$ is defined to be
the functor $E$ as constructed above.
\end{definition}
 \begin{proposition} \label{prop:BGisKan}
$BG ^{\mathcal{U}}$ is a Kan complex.
\end{proposition}                                           
\begin{proof} 

In what follows we again abbreviate $BG ^{\mathcal{U} }$ as
$V$. Recall that $\Lambda ^{n} _{k} \subset \Delta ^{n} _{simp}$, denotes the sub-simplicial set corresponding to the
``boundary'' of $\Delta ^{n} $ with the
$k$'th face removed, where by $k$'th face we mean
the face opposite to the $k$'th vertex. Let $h: \Lambda ^{n}
_{k} \to V$, $0 \leq k \leq n $,
be a simplicial map, this is also called a horn.   We need
to construct an extension of $h$ to $\Delta^{n}
_{simp}$.

For simplicity we assume $n=2$, and $k=1$ as the general case
is identical. 
There are three natural inclusions $$i _{j}: \D ^{0} _{simp}
 \to \Delta^{2} _{simp},$$ $j=0,1,2$, with $i _{1}$ corresponding to the inclusion of the horn vertex. The corresponding 0-simplices will be denoted by $0,1,2$.
Let $$\sigma _{i,j}: \D ^{1} _{simp} \to \Delta^{2}  $$  be the edge between vertexes $i,j$, that is $\sigma  _{i,j} (0) = i$, $\sigma  _{i,j} (1)=j$.

Let us denote by $L$ the smooth sub-simplicial set of
$\Delta^{2} _{\bullet } $  corresponding to the simplices whose
images lie in $\image \sigma _{0,1}$  or $\image \sigma
_{1,2}$. There are then smooth maps $\sigma _{i,j}:
\Delta^{1}  _{\bullet} \to L$ extending $\sigma _{i,j}$
above.

The map $h$ above, induces a smooth map $h _{\bullet}: L \to
V$, and we denote $P := h _{\bullet } ^{*} E $. So $P \to L$
is a simplicial $G$-bundle. The extension of $h$ to
$\Delta^{2} _{simp}$ will be accomplished once we extend the
$P$ over $\Delta^{2} _{\bullet }$.


%

\begin{lemma}
The bundle $P \to L$ is trivializable.
\end{lemma} 

\begin{proof} 
Set $P _{i,j}:= \sigma _{i,j} ^{*} P$,  
then by Theorem \ref{lemma:inducedSmoothGbundle}
$P _{i,j}$ is induced by a smooth $G$-bundle over
$\Delta^{1} $. The latter is trivializable, and hence $P
_{i,j}$ is trivializable as a simplicial $G$-bundle. 
Denote by $\phi _{i,j}: \Delta^{1} _{\bullet} \times G \to
P _{i,j}$ the corresponding trivialization over the $id:
\Delta^{1}  _{\bullet} \to \Delta^{1}  _{\bullet}$.

Denoting $0,1 \in \Delta^{1} _{\bullet} (0)$ the end-point
vertices as previously, $\phi _{0,1}$ induces
a $G$-equivariant smooth map:
\begin{equation*}
G = (\Delta^{1} _{\bullet} \times G) (1) \to P _{0,1} (1),
\end{equation*}
and we denote this map by $\phi  _{1}$.
Likewise, $\phi _{1,2}$ induces a smooth map:
\begin{equation*}
G = (\Delta^{1} _{\bullet} \times G) (0) \to P _{1,2} (0),
\end{equation*}
and we denote this map by $\phi  _{2}$.

Now the map $\phi _{1} ^{-1} \circ \phi _{2}$ may  not be the
$id: G \to G$, but clearly we may adjust $\phi _{1,2}$  so that it is. And so we may assume this holds.

Then $\phi _{0,1}$ and $\phi _{1,2}$ clearly induce
a trivialization $tr: T \to P$, for $T \to L$ the trivial
simplicial $G$-bundle.

\end{proof}
We have the trivial extension of $T$ to the trivial
simplicial $G$-bundle over $\Delta^{2} _{\bullet } $.
And so by the lemma above it should be clear that $P$
likewise has an extension $\widetilde{P} $ over $\Delta^{2}
_{\bullet }$, but we need this extension to be
$\mathcal{U} $-small so that we must be explicit.

We proceed inductively. Let $D ^{0}$ denote the full
sub-category of $\Delta^{sm} (\Delta^{2} _{\bullet})$ with
the set objects $\obj L \cup \Delta^{2} _{\bullet } (0)
$ (non-disjoint union).

We extend $P$ to a functor $\widetilde{P} ^{0}: D ^{0} \to
\mathcal{G}$. For ${\sigma} \in \Delta^{2} _{\bullet} (d)$, if $\sigma $ has image in the horn $\Lambda ^{2} _{1} \subset
\Delta^{2} $, then set $\widetilde{P} ^{0} (\sigma) = P (\sigma
)$.  
The extension of $\widetilde{P} ^{0}$ to morphisms in $D
^{0}$ is then taken to be the trivial extension. 

 Let $T ^{0}: D ^{0} \to
\mathcal{G}  $ be the trivial functor (as in the definition
of a trivial bundle in Example \ref{example:product}). 
Then in addition, there is clearly a natural
transformation $tr ^{0}: T ^{0} \to \widetilde{P} ^{0}$ extending the natural transformation $tr$ of the lemma above.

Let $S (n)$ be the statement:
\begin{enumerate}
	\item There is an extension $\widetilde{P} ^{n} $ of $P$
	over the full-subcategory $D ^{n} \subset \Delta^{sm}
	(\Delta^{2} _{\bullet})$, defined analogously to $D ^{0}$,
	with objects $$\obj L \cup \bigcup _{0 \leq k \leq n}
	\Delta^{2} _{\bullet} (k). $$ \label{property_extend1}
	\item $\widetilde{P} ^{n} $ satisfies compatibility.
	\item There is a natural
transformation $tr ^{n}: T ^{n} \to \widetilde{P} ^{n}$ extending the natural transformation $tr$ of the lemma above.
Where $T ^{n}: D ^{n} \to \mathcal{G} $ is the trivial
functor defined analogously to $T ^{0}$ above.
\label{property_extend}
\end{enumerate}

We prove $$S (n) \implies S (n+1),$$ and that moreover
the corresponding functor $\widetilde{P} ^{n+1}$ can be
chosen to extend $\widetilde{P} ^{n}$. Then natural induction
implies the existence of the needed extension $\widetilde{P}
$ over $\Delta^{2} _{\bullet}$.


Let $\sigma \in \Delta^{2} _{\bullet } (n+1) $. 
By the hypothesis $S (n)$ the functor:
\begin{equation*}
\sigma _{\bullet} ^{*}\widetilde{P} ^{n}: C _{n+1} \to \mathcal{G},  
\end{equation*}
is defined, where $C _{n+1}$ denotes the sub-category of $\Delta^{}
(\partial \Delta^{n+1} _{simp})$ with objects all non-degenerate
objects and with morphisms injections. (The pull-back is as
in Section \ref{sec:pullbacksimplicialbundle}). 

We then have a topological bundle $$p': N' _{\sigma}	\to
\partial \Delta^{n+1} $$ defined as the colimit of
$\sigma _{\bullet} ^{*}\widetilde{P} ^{n}$ over $C _{n+1}$.


Let $$inc: \partial \Delta ^{n+1} \to \Delta^{n+1} $$ denote
the inclusion.
We first construct a principal $G$-bundle with discrete topology
$$N _{\sigma} \xrightarrow{p} \Delta^{n+1}, $$
by the following conditions: 
\begin{align} \label{equation_identification1}
N  _{\sigma}| _{\partial \Delta^{n+1} } := p ^{-1} (\partial
\Delta^{n+1} ) & = N' _{\sigma}, \\ 
N  _{\sigma}| _{(\Delta^{n+1}) ^{\circ}} &
= (\Delta^{n+1}) ^{\circ} \times G, \label{equation_identification2}
\end{align}
where the projection map $p$ is determined by the maps
$p': N' _{\sigma } \to \partial \Delta^{n+1}  $,  and the
projection map $(\Delta^{n+1}) ^{\circ} \times G \to (\Delta^{n+1}) ^{\circ}$.

By the inductive hypothesis $S (n)$ part \ref{property_extend},
there is a distinguished trivialization 
$h _{n}: \partial \Delta^{n+1} \times G \to N' _{\sigma}$, corresponding to $tr ^{n}$.
The map $h _{n}$ and the identity map $(\Delta^{n+1})
^{\circ} \times G \to (\Delta^{n+1})
^{\circ}$ induce a discrete $G$-bundle isomorphism $\Delta^{n+1}
\times G \to {N}_{\sigma}$. 
Then push-forward the smooth $G$-bundle structure and the
topology along this
map. The resulting smooth $G$-bundle is then set to be
$\widetilde{P} ^{n+1} ({\sigma})$. 

We have thus defined the extension $\widetilde{P} ^{n+1}$ on
objects. We now need to
treat morphisms in $D ^{n+1} \subset \Delta^{} (\Delta^{2} _{\bullet }) $. 
For any $d$-simplex $\rho$ of $\Delta^{2} $ let $\rho
_{i}$ denote the $i$'th face of $\rho$, $0 \leq
i \leq d$. For clarity, this means that we have an inclusion
morphism $inc _{i}: \rho _{i} \to \rho $ corresponding to the diagram:
\begin{equation*}
\begin{tikzcd}
\Delta^{n} \ar[r, "\widetilde{inc}  _{i}"] \ar [dr, "\rho _{i}"]
&  \Delta^{d}  \ar [d,"\rho"] \\
& \Delta^{2},
\end{tikzcd}
\end{equation*}
where $\widetilde{inc}  _{i}$ is the topological face
inclusion map corresponding to the face opposite the vertex $i$, also called the $i$-face of $\Delta^{d} $.

By construction we have natural maps:
\begin{equation*}
\widetilde{P} ^{n+1} (inc _{i}): \widetilde{P} ^{n}
(\sigma _{i}) = \widetilde{P} ^{n+1}
(\sigma _{i}) \to \widetilde{P} ^{n+1} (\sigma ),
\end{equation*}
for each $i \in \{0, \ldots, n+1\}$.

If $m: \sigma ^{n+1} \to \rho ^{n+1}$ is an identity
morphism, then set $\widetilde{P} ^{n+1} (m)$ to
be the $id$.

Suppose we given a morphism $m: \sigma \to \rho$, $\sigma \in
\Delta^{2} _{\bullet } (n+1)$, $\rho \in
\Delta^{2} _{\bullet } (n)$. 
We then define $\widetilde{P} ^{n+1} (m)$ as follows.
First we define a map:
\begin{equation*}
\partial: inc ^{*} \widetilde{P} ^{n+1} (\sigma)
\to \widetilde{P} ^{n+1} (\rho),
\end{equation*}
for $inc: \partial \Delta ^{n+1} \to \Delta^{n+1} $ the
inclusion. 

Over a $j$-face of $\Delta^{n+1} $, $\widetilde{P} ^{n+1}
(\sigma)$ is naturally identified with $\widetilde{P}
^{n} (\sigma _{j})$. 
The composition $m _{j}= m \circ inc _{j}$, 
\begin{equation*}
\sigma _{j} \xrightarrow{inc _{j}} \sigma  \xrightarrow{m}
\rho,
\end{equation*}
is a morphism in $D ^{n}$.
So we have the map
$$\widetilde{P} ^{n} (m _{j}): \widetilde{P} ^{n} ({\sigma} _{j}) \to \widetilde{P} ^{n+1}  ({\rho}).$$
The collection of these maps for $0 \leq j \leq n+1$ then naturally induces the map
$\partial$.

We then define $\widetilde{P} ^{n+1} (m)$ using the
identifications \eqref{equation_identification1}, \eqref{equation_identification2} as follows.
Set ${\widetilde{P} ^{n+1} (m)} $ to be  the map
$\partial$ on  $$\widetilde{P} ^{n+1} (\sigma)|_{\partial
\Delta^{n+1}}= p ^{-1} (\partial \Delta^{n+1}). $$ 
The hypothesis $S (n)$ part \ref{property_extend} ensures
that ${\widetilde{P} ^{n+1} (m)} $ has an extension to a map
$\widetilde{P} ^{n+1} (\sigma) \to \widetilde{P} ^{n+1}
(\rho)$.


As $\widetilde{P} ^{n+1}$ must extend $\widetilde{P} ^{n}$,
combined with the construction above, we have thus specified the functor
$\widetilde{P} ^{n+1}$ on a generating set of morphisms in
$D ^{n+1}$, which defines $\widetilde{P} ^{n+1}$ completely.
For example, given a degeneracy $m: \sigma ^{n} \to
\rho^{n-2}$, we may factorize it as $\sigma ^{n}
\xrightarrow{m'} \rho
^{n-1} \xrightarrow{pr} \rho ^{n-2} $, and then set:
$$\widetilde{P} ^{n+1} (m) = (\widetilde{P} ^{n+1} (pr)
= \widetilde{P} ^{n} (pr)) \circ
\widetilde{P} ^{n+1} (m').$$
Functoriality of $\widetilde{P} ^{n}$, implies that this is
well defined. And by construction $\widetilde{P} ^{n+1}$
will be a functor satisfying compatibility.
So we are done.

\end{proof}

\begin{theorem} \label{thm:classifyKan} 
Let $X$ be a smooth simplicial set. $\mathcal{U}$-small simplicial $G$-bundles $P \to X$ are ``classified by'' smooth maps $$f _{P}: X \to BG ^{\mathcal{U} }.$$ Specifically:
\begin{enumerate}
   \item  For every
   $\mathcal{U}$-small $P$ there is a natural smooth map $f
   _{P}: X \to BG ^{\mathcal{U} }   $ so that 
   $$f _{P}^{*} EG ^{\mathcal{U} }  = P$$ as simplicial 
   $G$-bundles. 
	 We say in this case that $f _{P}$
   \textbf{\emph{classifies}}  $P$. 
   \item  If $P _{1}, P _{2}$ are isomorphic
   $\mathcal{U}$-small smooth simplicial
   $G$-bundles over $X$ then the classifying
   maps $f _{P _{1} }, f _{P _{2}}$ are smoothly homotopic,
	 as in Definition \ref{def:smoothhomotopy}.
	 \label{part_isomorphicbundles}
\item  If $X = Y _{\bullet }$  for $Y$ a smooth
   manifold and $f,g: X \to BG ^{\mathcal{U} }$
   are smoothly homotopic then $P _{f} = f ^{*}
   EG ^{\mathcal{U} }, P _{g} = g ^{*} EG
   ^{\mathcal{U} }$   are isomorphic simplicial
   $G$-bundles. \label{part_homotopicMaps}
\end{enumerate}
     
\end{theorem}
Note that the above is a partly stronger (because of
equality in Part 1) and partly weaker
than just saying that isomorphism classes of $\mathcal{U}
$-small bundles over $X$ are in correspondence with smooth
homotopy classes of maps $X \to BG ^{\mathcal{U}}$. It is
strictly stronger when $X = Y _{\bullet }$ for $Y$ a smooth
manifold.
 
\begin{proof}
Set $V=BG ^{\mathcal{U} }$, $E = EG ^{\mathcal{U} }$. 
Let $P \to X$ be a $\mathcal{U}$-small simplicial $G$-bundle.  
Define $f _{P}: X \to V$  by: 
\begin{equation}
   \label{eq:fPmap}
   f _{P} (\Sigma) = \Sigma _{*} ^{*} P, 
\end{equation}
where $\Sigma \in \Delta^{d}  (X) $, $\Sigma _*:
   \Delta^{d} _{\bullet } \to X$, the induced map,
and the pull-back $\Sigma _{*} ^{*}P$ our usual
simplicial $G$-bundle pull-back.  
We check that the map $f _{P}$  is simplicial.  

Let $m: k  \to d $  be a morphism in
$\Delta$. 
We need to check that the following diagram
   commutes:
\begin{equation*}
\begin{tikzcd}
X (d)  \ar[r, "X (m) "] \ar [d, "f _{P}"] & X (k)
   \ar [d,"f _{P}"] \\
V (d)  \ar [r, "V (m) "]    &  V (k). 
\end{tikzcd}
\end{equation*}
Let $\Sigma \in X (d) $, then by push-forward 
functoriality Axiom \ref{axiom:pushforward} 
$(X (m) (\Sigma))
_{*} = \Sigma _{*} \circ m _{\bullet} $ where
$m _{\bullet }: \Delta^{k} _{\bullet } \to
\Delta^{d} _{\bullet } $   
is the simplicial map induced
by $m: \Delta ^{k} \to \Delta ^{m}$.   And so
$$f _{P} (X (m) (\Sigma)) = (\Sigma _{*} \circ m _{\bullet})
^{*} P = m _{\bullet} ^{*}
(\Sigma _{*} ^{*}P) = V (m) (f _{P} (\Sigma)),
$$ where the second equality uses Lemma
\ref{lemma:pullbackisfunctorial}.  
Hence the diagram commutes.

We now check that $f _{P}$ is smooth. Let
$\Sigma \in X (d) $, then  we have:
\begin{align*}
   (f _{P} (\Sigma)) _{*} (\sigma)  & = \sigma
   _{\bullet } ^{*} (\Sigma ^{*} _{*} P)     \\
   & = (\Sigma _{*} \circ \sigma _{\bullet} ) ^{*}
   P, \quad \text{ Lemma \ref{lemma:pullbackisfunctorial}} \\
  & = (\Sigma _{*} (\sigma)) ^{*} _{*} P,
   \quad \text{as $\Sigma _{*}$ is smooth, Lemma
   \ref{lemma:simplicesaresmooth}}  \\ 
	 & = f _{P} (\Sigma _{*} (\sigma)),  \quad \text{by \eqref{eq:fPmap} }  \\ 
	 & = (f _{P} \circ \Sigma _{*}) (\sigma),  
\end{align*}
and so $f _{P}$ is smooth. 

We check that $f _{P}^{*}E = P$.
Let $\Sigma: \Delta^{d} _{\bullet } \to X$ be
smooth, and $\sigma \in \Delta^{d} _{\bullet }$.   First, we
need the identity: 
\begin{equation} \label{eq:firstparagraph} 
\begin{aligned}  
\Delta ^{sm} f _{P} (\Sigma) (\sigma) = (f _{P} \circ \Sigma) (\sigma) = f _{P}  (\Sigma (\sigma)) & = (\Sigma
   (\sigma)) ^{*} _{*}P \text{ by
   definition of $f _{P}$ }   \\
   & = (\Sigma ^{*} \circ
   \sigma _{\bullet}) ^{*} P   \text{ as $\Sigma$
   is smooth} \\ 
   & = \sigma ^{*} _{\bullet } (\Sigma ^{*}P)
   \text{ Lemma \ref{lemma:pullbackisfunctorial}}  \\ 
   & = g(\Sigma ^{*}P)  (\sigma). 
\end{aligned}
\end {equation} 
So 
\begin{equation} \label{equation_DeltaSM}
\Delta ^{sm} f _{P} (\Sigma) = g(\Sigma ^{*}P).
\end{equation}
Then 
\begin{align*}
   f _{P} ^{*} E (\Sigma) = (E \circ \Delta ^{sm}f
   _{P}) (\Sigma) & = E
   (g(\Sigma ^{*}P)), \text{ by \eqref{eq:firstparagraph} }  \\
   & = (\Sigma ^{*}P)  (id ^{d} _{\bullet}), \text{
      definition of $E$ } \\ & =  P
   (\Sigma). 
\end{align*}
So $f _{P} ^{*}E=P$   on objects.

Now let $m$ be a morphism:
  \begin{equation*}
   \begin{tikzcd}
   \Delta^{k} _{\bullet }  \ar[r,
      "\widetilde{m} _{\bullet }"] \ar [rd,
      "\Sigma _{1}"] &  \Delta^{d} _{\bullet } \ar
      [d,"\Sigma _{2}"] \\
     &  X,
   \end{tikzcd}
   \end{equation*}
in $\Delta ^{sm} (X) $.      
We then have, for $e _{m}$ is as in the definition of $E$:
\begin{align*}
      f _{P}^{*}E (m) & =  E (\Delta ^{sm} f _{P} (m) )  \\
      & = (\Delta ^{sm} f _{P} (\Sigma _{2})) ^{b} (e 
      _{m})  \text{ by definition of  $E$} \\ 
    & = \Sigma ^{*} _{2} P (e _{m}) \text{ by
     \eqref{equation_DeltaSM}} \\
     & = (P \circ \Delta ^{sm}
     \Sigma _{2}) (e _{m}).
\end{align*} 
But $\Delta ^{sm} \Sigma _{2} (e _{m}) $   is the diagram: 
\begin{equation*}
\begin{tikzcd}
   \Delta^{k} _{\bullet }  \ar[r,
      "\widetilde{m} _{\bullet }"] \ar [rd,  bend
   right, "\Sigma _{2} \circ \widetilde{m} _{\bullet
   }"] &  \Delta^{d} _{\bullet } \ar
      [d,"\Sigma _{2} \circ id _{\bullet} ^{d}"] \\
     &  X,
\end{tikzcd} 
\end{equation*}
   i.e. it is the diagram $m$.  So $(P \circ
   \Delta ^{sm}
   \Sigma _{2}) (e _{m}) = P (m) $. 
Thus, $f _{P} ^{*} E =
   P$ on morphisms.

So we have proved the first part. We now prove the second
part. Suppose that $\phi: P _{1} \to  P _{2}$ is an
isomorphism of $\mathcal{U} $-small simplicial $G$-bundles over $X$.     
We construct a $\mathcal{U} $-small simplicial $G$-bundle
$\widetilde{P} $ over $X \times I $ 
as follows, where $I= \Delta^{1}   _{\bullet }$ as
before. 

Let $\sigma$ be a $k$-simplex of
$X$. Then $\phi$ specifies a $G$-bundle
   diffeomorphism $\phi _{\sigma}: P _{1}(
   \sigma)
   \to P _{2}(\sigma)  $ over the identity map $ \Delta ^{k} \to
\Delta ^{k}  $. 
Let $M
_{\sigma} $    be the
   mapping cylinder of $\phi
    _{\sigma }$. So that 
\begin{equation}
      \label{eq:mappingcylinder}
     M _{\sigma} = (P _{1} (\sigma) \times
   \Delta^{1}   \sqcup P _{2}
   (\sigma))/\sim,
   \end{equation}
    for $\sim$  the
   equivalence relation  generated by the
   condition $$(x,1) \in P _{1} (\sigma) \times
   \Delta^{1}  \sim \phi (x)
   \in P _{2} (\sigma).  $$ 
   Then $M_{\sigma} $ is a smooth $G$-bundle over
$\Delta ^{k} \times \Delta^{1}$.  

Let $pr_{X}, pr _{I _{ }} $ be the natural projections of $X \times I $, to $X$ respectively $I$.
Let $\Sigma$ be a $d$-simplex of $X \times I
$, for any $d$. 
Set $\sigma _{1} = pr _{X} \Sigma $, and $\sigma
    _{2} = pr _{I} (\Sigma)  $. 
Let $id ^{d}: \Delta ^{d} \to \Delta ^{d} $ be the
    identity, so $$(id ^{d}, \sigma _{2}): \Delta
    ^{d} \to \Delta ^{d} \times \Delta^{1} ,$$ 
is a smooth map, where $\sigma _{2}$ is the
    corresponding smooth
map
$\sigma _{2}: \Delta^{d} \to \Delta^{1} = [0,1]  $.
We then define $$\widetilde{P}_{\Sigma} := (id ^{d}, \sigma _{2}) ^{*} M _{\sigma _{1}},$$   
which is a smooth $G$-bundle over $\D ^{d} $.

Notice that if $\Sigma $ is in $X \times 0 _{\bullet
} \subset X \times I$, then we do \emph{not}  have
$\widetilde{P} (\Sigma ) = P _{1} (\Sigma )$, instead there
is a natural isomorphism. This is for
the same reason that fixing the
standard construction of the set theoretic pull-back,
a bundle $P \to B$ is not set theoretically equal to the bundle $id
^{*}P \to B$, for $id: B \to B$ the identity, (but
they are of course naturally isomorphic.)
However, we can adjust the construction of $\widetilde{P}
_{\Sigma }$ so that $\widetilde{P} (\Sigma ) = P _{1} (\Sigma )$ does
hold,  similarly to the inductive procedure in
the proof of Proposition \ref{prop:BGisKan}. In what follows, we
ignore this minor ambiguity.

Suppose that $\rho: \sigma  \to \sigma'$ is a
morphism in $\D ^{sm} (X) $, for $\sigma$ a
   $k$-simplex and $\sigma'$ a $d$-simplex. 
As $\phi$ is a simplicial $G$-bundle map, 
we have a commutative diagram:
\begin{equation}
   \label{eq:commatutiveNaturalphi}
\begin{tikzcd}
P _{1} (\sigma) \ar[r, "P _{1} (\rho) "] \ar [d,
   "\phi _{\sigma}"] & P _{1} (\sigma')    \ar [d,"\phi _ {\sigma'}"]
   \\
P _{2} (\sigma)  \ar [r, "P _{2} (\rho)  "]    &
P _{2} (\sigma'). 
\end{tikzcd}
\end{equation}
And so we get a naturally 
induced (by the pair of maps $P _{1} (\rho), P _{2} (\rho) $) bundle map:
\begin{equation}
     \label{eq:MprSigma} 
     \begin{tikzcd}
    M _{\sigma} \ar[r, "g _{\rho}"] \ar [d, ""] &
          M _{\sigma'} \ar [d,""] \\
  \Delta ^{k}  \times \Delta ^{1}  \ar [r, "\widetilde{\rho}
        \times id "]    &  \Delta ^{d} \times \D
        ^{1}.
     \end{tikzcd}
\end{equation} 
More explicitly,  let $q _{\sigma}: P _{1} (\sigma) \times
   \Delta^{1}   \sqcup P _{2}
   (\sigma) \to M _{\sigma}$ denote the
quotient map. Define 
$$\widetilde{g} _{\rho}: P _{1} (\sigma) \times
   \Delta^{1}   \sqcup P _{2}
   (\sigma) \to M _{\sigma'}$$  by: $$\widetilde{g}  (x,t) = q
_{\sigma'}
 ((P _{1} (\rho)
(x),t)) \in M _{\sigma'}, $$ for $$(x,t) \in P _{1} (\sigma)
\times \Delta^{1},  $$  
while $\widetilde{g}  _{\rho} (y) = q _{\sigma'}(P _{2} (\rho)
(y))  $ for $y
\in P _{2} (\sigma). $ 
By commutativity of \eqref{eq:commatutiveNaturalphi} $\widetilde{g} _{\rho}$
induces the map $g _{\rho}:M _{\sigma} \to M
_{\sigma'}$, appearing in
\eqref{eq:MprSigma}.

Now suppose we have a morphism $m: \Sigma \to
\Sigma'$ in
   $\Delta ^{sm} (X \times I) $, where $\Sigma$
is a $k$-simplex and $\Sigma'$ is a $d$-simplex.  
Then we have a commutative diagram:
 \begin{equation}
    \label{diagram:wtPm} 
 \begin{tikzcd} 
 M _{\sigma} \ar[r, "g _{pr _{X}} (m) "] \ar [d, ""] &
           M _{\sigma'} \ar [d,""] \\
 \Delta ^{k}    \times \D ^{1} \ar[r, "\widetilde{m}  \times
    id"]  &  \Delta ^{d} \times \D ^{1} \\
     \Delta ^{k} \ar [r, "\widetilde{m}"] \ar [u,
    "h _{1}"]    & \Delta ^{d} \ar [u, "h _{2}"] \\
   \widetilde{P} _{\Sigma} \ar[u] \ar[uuu, bend
     left=120]  & \widetilde{P}_{\Sigma'} \ar [u] \ar
     [uuu, bend right=120]  \\
 \end{tikzcd}
 \end{equation} 
where $h _{1} = (id ^{k}, pr _{I} (\Sigma) ) $
and $h _{2} = (id ^{d}, pr _{I} (\Sigma') )  $.
We  then readily get
an induced natural bundle map:
\begin{equation*}
   \widetilde{P} (m): \widetilde{P} _{\Sigma} \to \widetilde{P}
   _{\Sigma'},
 \end{equation*}
as left most and right most arrows in the above
commutative diagram are the
natural maps in pull-back squares, and so by
universality of the pull-back such a map exists
and is uniquely determined.
  Of course $\widetilde{P} (m) $ is   the unique map making the whole
diagram \eqref{diagram:wtPm} commute. 

With the above assignments, it is 
immediate that $\widetilde{P}
$ is indeed a functor, by the uniqueness of
the assignment
$\widetilde{P} (m) $. And this determines our
$\mathcal{U} $-small smooth simplicial $G$-bundle $\widetilde{P} \to X \times I$.   
By the
first part of the theorem, we have an induced smooth
classifying map $f _{\widetilde{P}}: X \times I
    _{ } \to V$. By construction, it is a homotopy
between $f _{P _{1}}, f _{P _{2} } $.  So we have verified the
second part of the theorem.

We now prove the third part of the theorem. Let $X
= Y _{\bullet }$. 
Suppose that $f,g: X \to V$ are smoothly homotopic, and let $H: X \times  I \to V$  be the corresponding
smooth homotopy.  Now $P _{H} = H ^{*} E$ is a
simplicial $G$-bundle over $X
\times I = (Y \times [0,1] ) _{\bullet }
$  and hence by  Theorem
\ref{lemma:inducedSmoothGbundle} $P _{H}$  is
induced by a smooth $G$-bundle $P' _{H}$ over $Y
\times [0,1] $. 

Now by construction $$P _{f} \simeq (P' _{H}| _{Y \times
\{0\}})
^{\Delta^{} }$$ and $$P _{g} \simeq (P' _{H}| _{Y \times \{1\}})
^{\Delta^{} }.$$ And $$P' _{H}| _{Y \times
\{0\}} \simeq P' _{H}| _{Y \times
\{1\}}$$ by standard smooth bundle theory and hence $$(P' _{H}| _{Y \times
\{0\}}) ^{\Delta^{} } \simeq (P' _{H}| _{Y \times
\{1\}}) ^{\Delta^{} }.$$
And so $P _{f} \simeq P _{g}.$

\end{proof}

Since Theorem \ref{lemma:inducedSmoothGbundle} works for
manifolds with corners, and since $\Delta^{k} _{\bullet
} \times I \simeq (\Delta^{k} \times \Delta^{1} ) _{\bullet }$ the proof of the theorem above readily extends to give the
following theorem.  We say that
a smooth $G$-bundle $P$ over $\Delta^{k}$ is
trivial over $\partial \Delta^{k}$,  if there is
a distinguished trivialization of $P$ over $\partial
\Delta^{k} $.
A \textbf{\emph{relative isomorphism}} of $P _{0}, P _{1}$
as above, is an isomorphism that is trivial 
trivial over $\partial \Delta^{k} $, in the respective
distinguished trivializations.

Let $v _{0} \in BG ^{\mathcal{U}} (0)$ correspond
to the trivial simplicial $G$-bundle $G \times \Delta ^{0}
_{\bullet } \to \Delta^{0}  _{\bullet }$.
\begin{theorem} \label{thm_smoothHomotopyGroupsBG} The set $\pi _{k} ^{sm} (BG
^{\mathcal{U}}, v _{0}) $ (Definition \ref{def:smoothpik})
is naturally isomorphic to the set $\mathcal{P}
^{\mathcal{U}} _{k}$
of equivalence classes of smooth, $\mathcal{U} $-small $G$-bundles $P$ over
$\Delta^{k}$ trivial over $\partial \Delta^{k} $, where $P _{0} \sim P _{1}$ if there is
a relative bundle isomorphism from $P _{0}$ to $P _{1}$.
The map 
\begin{equation} \label{eq_clk}
cl_k: \pi _{k} ^{sm} (BG
^{\mathcal{U}}, v _{0})  \to \mathcal{P}
^{\mathcal{U} } _{k} 
\end{equation}
is given by $[f] \mapsto [P _{f}]$, where
$P _{f} = f ^{*} EG ^{\mathcal{U} } (id ^{k})$, $id ^{k}:
\Delta^{k} \to \Delta^{k} $  the identity.
\end{theorem}

We now study the dependence on a Grothendieck universe $\mathcal{U}$. 
\begin{theorem} \label{thm:homotopytypeBGU} Let
   $G$ be a locally convex Lie group. Let $\mathcal{U}$ be a
   $G$-admissible universe, let $|BG
   ^{\mathcal{U}} |$ denote the geometric
   realization of $BG ^{\mathcal{U}} $ and let $BG
   ^{top}$ denote the classifying space
   of $G$ as defined by the Milnor construction
   \cite{cite_Milnoruniversalbundles}.
Then there is a weak homotopy equivalence $$e
^{\mathcal{U}}: |BG ^{\mathcal{U}}| \to   BG
^{top},$$  which is natural in the sense that
if $\mathcal{U} \in  \mathcal{U}' $ then 
\begin{equation} \label{eq:naturality}
   [e ^{\mathcal{U}'} \circ |i ^{\mathcal{U},
   \mathcal{U} '}|] = [e ^{\mathcal{U}}], 
\end{equation}
where $|i ^{\mathcal{U},
\mathcal{U}'}|: |BG ^{\mathcal{U} }| \to  |BG
^{\mathcal{U}' }|$ is the map of geometric
   realizations, induced by the natural
inclusion $i ^{\mathcal{U},
\mathcal{U}'}: BG ^{\mathcal{U} } \to  BG
^{\mathcal{U}' }$ and where $[\cdot] $ denotes
   the homotopy class. 
In particular, if $G$ has the homotopy type of a CW complex,
then for all $\mathcal{U} $ $BG ^{\mathcal{U}} $ has the
homotopy type of $BG ^{top}$. 
\end{theorem}
\begin{proof} To cut down on notation set $V:= BG ^{\mathcal{U}}$, and $E:=EG
^{\mathcal{U}}$, $v _{0} \in V$ will be as above, and we
will not disambiguate by decoration with $\mathcal{U} $.


\begin{lemma} \label{lemma_colimitbundle}
Let $P \to X$ be a simplicial $G$-bundle.
Set $|P| = \colim _{\Delta (X)} P$, where the
colimit is understood to be in the category of
topological $G$-bundles, and recalling that $P$ is a functor
$\Delta^{sm} (X) \to \mathcal{G} $, and so restricts to
a functor $\Delta (X) \to \mathcal{G} $.
Then the natural map $|P| \to |X|$ is
a topological principal $G$-bundle. We call this
the \textbf{\emph{geometric realization of $P$}}.
\end{lemma}

\begin{proof} [Proof]
Set $P' = \mathcal{F} _{1}  \circ P$, where $\mathcal{F}
_{1} $ is as in Definition \ref{def:simplicialGbun}. So there is a natural
transformation $N: P \to P'$ of $Top$ valued functors.
By naturality of the colimit there is an induced map:
\begin{equation*}
\colim _{\Delta (X)} P \to \colim _{\Delta
(X)} P',
\end{equation*}
i.e. a continuous map $|P| \to |X|$.

We check that the map $|p|$ is a locally trivial fibration. 
By a topological $d$-simplex $\Sigma: \Delta^{d} \to |X|$ we
shall mean the natural map (in the colimit
diagram) $\Delta^{d} \to |X|$
corresponding to a non-degenerate $d$-simplex of $X$.
If $v \in |X|$ is a vertex, i.e. the image of a topological 0-simplex $\Delta^{0} \to |X|$, we construct
a contractible open neighborhood 
$U _{v} \ni v$  as follows.  

Let $S _{d}$ be the set of the
topological $d$-simplices with images containing $v$.
For $\Sigma \in S _{d}$, let $\widetilde{\Sigma} = \image
\Sigma - \image \Sigma _{v}  $, where $\Sigma _{v}:
\Delta^{d-1}  \to |X|$ is the face of $\Sigma $ not
containing $v$.
Then define:
\begin{equation*}
U _{v} = \bigcup _{d} \bigcup _{\Sigma \in S _{d}}
\widetilde{\Sigma}. 
\end{equation*}

Over each $\widetilde{\Sigma} $ the bundle $|p|$ is
obviously a trivializeable topological $G$-bundle. Moreover,
any trivialization over the boundary of $\widetilde{\Sigma}$
(the collection of
the remaining faces of ${\Sigma} $) may be extended
to a trivialization over $\widetilde{\Sigma} $.
We may then proceed inductively.

Set $$U ^{k} = \bigcup _{d \leq k} \bigcup _{\Sigma \in
S _{d}} \widetilde{\Sigma},$$
by convention set $U ^{0} = \{v\}$.
We construct $G$-bundle trivialization maps $\forall k \in \mathbb{N}$:
\begin{equation*}
f ^{k}: |p| ^{-1} (U ^{k}) \to U ^{k} \times G,
\end{equation*}
with the property that each $f ^{k+1}$ extends $f ^{k}$.
Clearly, $f ^{0}$ exists. Suppose we have constructed $f ^{0},
\ldots, f ^{k}$ for some $k \geq 0$ with the property above.
We then construct $f ^{k+1}$. Let $\Sigma \in S _{k+1}$,
then over the boundary of $\widetilde{\Sigma}$ we already
have a trivialization determined by $f ^{k}$. Clearly the
set $J _{\Sigma}$ of extensions of this trivialization
over $\widetilde{\Sigma}$ is nonempty. 

Using axiom of choice fix an element of $J _{\Sigma }$ for
each $\Sigma \in S _{k+1}$, and this determines our extension $f
^{k+1}$. \footnote{Strictly speaking, to formalize the
existence of the infinite sequence $\{f ^{k}\}$ (not just
any fragment) requires a separate invocation of the axiom of
choice, or more specifically the so called axiom of
dependent choice, but this is standard.}
%
\end{proof}

By the lemma above we have a topological $G$-bundle
$$|E|  \to |V|. $$ 
Then 
\begin{equation} \label{eq:EsimeqEG}
  |E|  \simeq  e ^{*} EG ^{top},   
\end{equation}
where 
\begin{itemize}
\item $EG ^{top} $ is the universal $G$-bundle
over $BG ^{top} $.
	\item $$e  = e ^{\mathcal{U}}: |V| \to BG ^{top}$$ is uniquely determined up
to homotopy.
\item  $\simeq$ here and the rest of this
argument will mean $G$-bundle isomorphism or
simplicial $G$-bundle isomorphism, depending on context. 
\end{itemize}
We say that $e$ \textbf{\emph{classifies}}  $|E| \to |V|$.
We will show that $e$ induces an
isomorphism of all homotopy groups. 

We first prove an auxiliary lemma. Let $\mathcal{U}'$ be
a universe enlargement of $\mathcal{U}$, that is
$\mathcal{U}'$ is a universe with $\mathcal{U} \in
\mathcal{U}' $.
There is a natural inclusion map $$i=i ^{\mathcal{U},
\mathcal{U}'}: BG ^{\mathcal{U}} \to BG ^{\mathcal{U}' },$$
and $$i ^{*} EG ^{\mathcal{U} '} = E  \, \footnote
{This is indeed an equality, not just a natural
isomorphism.}.$$
\begin{lemma} \label{lemma:isomorphismhomotopy}
Let $G$ be a locally convex Lie group, then
$$i_*: \pi ^{sm} _{k} (BG ^{\mathcal{U}}, v _{0})  \to \pi ^{sm}
   _{k} (BG ^{\mathcal{U}'}, v _{0})  $$ is a set isomorphism for all
   $k \in \mathbb{N}$. 
\end{lemma}
\begin{proof}
We show that $i _{*} $ is injective. Let's abbreviate $V = BG
^{\mathcal{U}}$, $V' = BG ^{\mathcal{U}'}$, $E= EG
^{\mathcal{U} }$, $E' = EG ^{\mathcal{U}'}$.
Let $f,g: \Delta ^{k} _{\bullet} \to V$ be a pair of 
smooth maps relative to $v _{0}$. Let $P _{f}, P _{g} $   denote
the smooth bundles over $\Delta^{k} $ as given by the
correspondence of Theorem \ref{thm_smoothHomotopyGroupsBG}.
Set $f' = i \circ f$, $g' = i \circ g$ and suppose
that  $f', g'$ are smoothly homotopic. Then by Theorem
\ref{thm_smoothHomotopyGroupsBG},
$P _{f'}, P _{g'}$ are relatively isomorphic and so $P _{f'} $ and $P
_{g'}$ are relatively isomorphic. (Since we may identify $P
_{f'}, P _{g'}$ with $P _{f}, P _{g}$.) 

We now show surjectivity of $i _{*} $.  
Let $f: \Delta ^{k} _{\bullet} \to V'  $ be a smooth
relative to $v _{0}$ map. Let $P _{f} \to \Delta^{k} $ be the corresponding
smooth bundle trivial over $\partial \Delta^{k} $, via Theorem \ref{thm_smoothHomotopyGroupsBG}. 

It is elementary that any such smooth bundle is isomorphic to a bundle obtained by
the clutching construction corresponding 
to some smooth relative to $e$ map $\phi: \Delta^{k-1}  \to G$
(meaning as before the boundary is taken to $e$).
Specifically, $P'$ is isomorphic as a smooth
$G$-bundle to a bundle of the form: $$C _{\phi}
=  \Delta^{k} _{-} \times G  \sqcup  \Delta ^{k} _{+}
\times G/\sim, \,    $$ where:
\begin{enumerate}
	\item $\Delta ^{k} _{+}, \Delta ^{k} _{-}    $ are connected and closed subsets of $\Delta^{k}
$ diffeomorphic to $\Delta^{k} $, covering $\Delta^{k} $, whose intersection is the
image of a smooth embedding $i: \Delta^{k-1}  \to \Delta^{k} $ mapping
boundary to boundary, and mapping the interior to the interior
of $\Delta^{k} $. 

\item  $\sim$ is the
equivalence relation generated by the relation:  
for $(i(x), g) \in \Delta ^{k} _{-} \times G $,
$$(i (x), g) \sim (i (x), \phi (x) \cdot g) \in D ^{k} _{+} \times G.$$ 
\end{enumerate}

This gluing construction can be carried out in $\mathcal{L}$, for any $G$-admissible Grothendieck universe
$\mathcal{L}$. In particular, $C _{\phi}$  is $\mathcal{U}
$-small and $C _{\phi }$ and $P _{f}$ are relatively
isomorphic $\mathcal{U} '$-small smooth $G$-bundles. 


If we denote by $f _{\phi}$ a representative for $cl _{k}
^{-1} ([C _{\phi }])$, then 
by Theorem \ref{thm_smoothHomotopyGroupsBG} $f _{\phi }: \Delta^{k}
_{\bullet} \to V'$ is smoothly relatively homotopic to $f$. 
But $C _{\phi}$ is $\mathcal{U}$-small,   and hence $[C
_{\phi}] \in \mathcal{P} ^{\mathcal{U}} _{k}$ is  $cl _{k} ([f'])$
for a smooth based map $f': \Delta _{\bullet }  ^{k}
\to V$. It is immediate that $[i \circ
f'] = [f _{\phi}] $, since $i ^{*}E'=E$.  And so $i _{*} ([f']) = [f]$. 
\end{proof}
\begin{corollary} \label{cor:classify}  Let $G$ be
a locally convex Lie group, and let $\mathcal{U}, \mathcal{U} '$
be as in the previous lemma.  Then the natural map
$$j: \mathcal{P} ^{\mathcal{U}} _{k} \to 
\mathcal{P} ^{\mathcal{U}'} _{k},  $$ is a set
bijection.
\end{corollary}
We now prove that $e _{*}: \pi _{k}  (|V|, v _{0}) \to \pi
_{k} (BG ^{top}, x _{0})$ is an isomorphism, where $x _{0}$ is any fixed point s.t. $EG
^{\mathcal{U} }$ is given a fixed trivialization over $x
_{0}$.

Let $f: \Delta^{k}  \to BG ^{top}   $ be a continuous based
at $x _{0}$ map. By
{M\"uller}-{Wockel}~\cite[Theorem II.11]{cite_MullerSmoothVsContinousBundles}, the bundle $P _{f}: = f ^{*}
EG ^{top}  $ is topologically relatively isomorphic to a
smooth $G$-bundle $P' \to \Delta^{k}    $ trivial over the
boundary. By the axiom
of universes $P'$ is $\mathcal{U}
_{0}$-small for some $G$-admissible $\mathcal{U}
_{0} \ni \mathcal{U} $. 

By Theorem \ref{thm_smoothHomotopyGroupsBG} $[P'] 
= j (cl _{k} ([g])) $ for some smooth relative to $v _{0}$
$$g: \Delta  ^{k} _{\bullet} \to V.  $$  
Hence $P _{f}$ is relatively isomorphic as a topological
$G$-bundle to $|g| ^{*} |E|$, where $$|g|:
\Delta^{k} \simeq |\Delta^{k} _{simp}| \to |V|$$ is
the natural map induced by $g$, and $\Delta^{k} \simeq
|\Delta^{k} _{simp}|$ is the natural topological homeomorphism.
And so ${P} _{f}$  is relatively isomorphic to $|g| ^{*}
e ^{*} EG ^{top}$, and so as relative classes $[f]
= e _{*} [g]$, so that we are
done with surjectivity.

We now prove injectivity. Let $f _{0}, f _{1}: \Delta  ^{k} \to |V|  $ be
continuous and relative to $v _{0}$. We may represent the
relative classes $[f _{i}]$ by $|{g} _{i}|$,
where $$|{g} _{i}|: |\Delta^{k} _{simp}| \to |V|$$ are
induced by some relative to $v _{0}$ maps $g _{i}:
\Delta ^{k} _{simp} \to V$. This is by generalities of
homotopy groups of Kan complexes, see for instance
~[Chapter 1]\cite{cite_joyal1999introduction}. 

Let $P _{i}  \to \Delta^{k}  $ be the smooth trivial over
the boundary $G$-bundles corresponding to $g _{i}$,
meaning:
take the simplices corresponding to $g _{i}$, these
represent simplicial $G$-bundles over $\Delta^{k} _{\bullet
}$ by construction of $V$, then evaluate on $id ^{k}: \Delta^{k} \to \Delta^{k} $.

%

Now suppose that $[e \circ f _{0}] = [e \circ f _{1}] $.
Then $P _{i}$ are relatively isomorphic as topological
$G$-bundles, and so by
~\cite[Theorem II.12]{cite_MullerSmoothVsContinousBundles} $P _{i}$ are
smoothly relatively isomorphic $G$-bundles over $\Delta^{k}
$. And so by Theorem \ref{thm_smoothHomotopyGroupsBG} $g _{i}$ are smoothly homotopic. Consequently, $|g _{i} | $ are homotopic and so $[f _{0}] = [f _{1}] $. 

It follows that:
\begin{equation} \label{equation_eisom}
\forall k \in \mathbb{N}: e _{*}:
\pi _{k} (|V|, v _{0}) \to \pi _{k} (BG ^{top}, x _{0})  
\end{equation}
is an isomorphism.  



Finally, we show naturality.
Let $$|i ^{\mathcal{U} , \mathcal{U}
'}|: |V| \to  |V'|$$
denote the map induced by the inclusion $i
^{\mathcal{U}, \mathcal{U} '}$. Since  $E = 
(i ^{\mathcal{U}, \mathcal{U}'}) ^{*} E'$, 
we have
that $$|E| \simeq  |i ^{\mathcal{U} , \mathcal{U}
'}| ^{*} |E'|$$  and
so $$|E| \simeq |i ^{\mathcal{U} , \mathcal{U}
'}| ^{*} \circ (e ^{\mathcal{U}'}) ^{*} EG
^{top},$$ 
by \eqref{eq:EsimeqEG},
from which the conclusion immediately follows. And we are
done with the proof of the theorem.
\end{proof} 

\section{The
universal Chern-Weil homomorphism}
\label{sec:universalChernWeil}
In this section $G$ is a generalized Lie group  and
$\mathfrak g$ its Lie algebra.
Pick any simplicial $G$-connection $D$ on $EG ^{\mathcal{U}} \to
BG ^{\mathcal{U}}$. By construction of Section
\ref{sec:Chern-Weil} we obtain a $dg$ homomorphism:  
\begin{equation} \label{equation_dguniversal}
cw ^{D} := cw ^{EG ^{\mathcal{U}},D}: \mathcal{I} (G)  \to \Omega ^{\bullet } (BG ^{\mathcal{U}}, \mathbb{R} ).
\end{equation}
Let $cw$ represent $[cw ^{D}]$ as in Notation
\ref{notation:represent}.
This satisfies the following property:
\begin{proposition} \label{prop:abstractChernWeil}
Let $\mathcal{U} $ be a
$G$-admissible Grothendieck universe.  
Let $P \to X$ be a $\mathcal{U} $-small simplicial 
$G$-bundle and let $$cw ^{P}: \mathcal{I} (G)
\to \Omega ^{\bullet } (X, \mathbb{R} ),$$ be
as in Notation \ref{notation:represent}.  Then
$$f _{P} ^{*} \circ cw \simeq  cw ^{P},
$$  where $f _{P}: X \to BG ^{\mathcal{U} }$  is the
classifying map. 
\end{proposition} 
\begin{proof}
   This follows immediately from Lemma
   \ref{lemma:pullbackChernWeil}.
\end{proof}
%

Let $e ^{\mathcal{U} }$ be as in
Theorem \ref{thm:homotopytypeBGU}, then this is a weak
equivalence.
And so induces an isomorphism $$e ^{\mathcal{U}
} _{*}: H ^{\bullet}(|BG ^{\mathcal{U} }|, \mathbb{R} ^{})
\to H ^{\bullet } (BG ^{top}, \mathbb{R} ^{} ),$$ Hatcher
~\cite[Proposition 4.21]{cite_HatcherAlgebraic}.



Set $$c ^{\rho, \mathcal{U} } := c ^{\rho} (EG ^{\mathcal{U}
}) =  [\int (cw (\rho))] \in H ^{2k}
(BG ^{\mathcal{U}},
\mathbb{R}).$$
Then we define the cohomology class
$$c ^{\rho} := e
^{\mathcal{U}} _{*} (|c ^{\rho,
\mathcal{U} }|)  \in H ^{2k}
(BG ^{top}, \mathbb{R} ^{} ), $$ where the $G$-admissible
universe $\mathcal{U} $ is chosen
arbitrarily and where  $$|c ^{\rho,
\mathcal{U} }| \in H ^{2k} (|BG ^{\mathcal{U} }|,
\mathbb{R} ^{} )
 $$ is as in Notation
\ref{notation:geomRealizationClass}.
\begin{lemma}
   \label{lemma:independent} 
The cohomology class $c ^{\rho} $  is well-defined.
\end{lemma} 
\begin{proof} Given another choice of a
   $G$-admissible universe
$\mathcal{U} '$, let $\mathcal{U} '' \supset \{
   \mathcal{U}, \mathcal{U} ' \} $ be a common universe
enlargement.     
   By Lemma \ref{lemma:pullbackChernWeil} and
   Lemma \ref{lemma:naturalrealization} $$|i ^{\mathcal{U},
   \mathcal{U} ''}| ^{*} (|c ^{\rho, \mathcal{U}
   ''}|)
= |c ^{\rho, \mathcal{U}}|.$$ 
Now $|i ^{\mathcal{U},
\mathcal{U} ''}|$ is a weak equivalence by Theorem
\ref{thm:homotopytypeBGU}.
Let $$|i ^{\mathcal{U}, \mathcal{U} ''}| ^{*}: 
 H ^{\bullet } (|BG ^{\mathcal{U}' }|, \mathbb{R} ^{} ) \to
 H ^{\bullet}(|BG ^{\mathcal{U} }|, \mathbb{R} ^{}) $$ be the
 corresponding algebra isomorphism, and let $|i
 ^{\mathcal{U}, \mathcal{U} ''}| _{*}$ denote its inverse.
 
Then we have:
\begin{equation} \label{equation_iUpullback}
|i ^{\mathcal{U},
   \mathcal{U}''}| _{*} (|c ^{\rho, \mathcal{U}}|)  =
   |c ^{\rho, \mathcal{U}''}|.
\end{equation}
Consequently,   
\begin{align*}
   e^{\mathcal{U}} _{*} (|c ^{\rho, \mathcal{U} }|) & =  e ^{\mathcal{U} ''} _{*} \circ |i ^{\mathcal{U},
   \mathcal{U}''}| _{*} (|c ^{\rho, \mathcal{U}})|, \text{ by
	 the naturality part of Theorem \ref{thm:homotopytypeBGU}}
	 \\ 
   & = e ^{\mathcal{U} ''} _{*} (|c ^{\rho,
	 \mathcal{U}''}|), \text{ by
	 \eqref{equation_iUpullback}}.
\end{align*}
In the same way we have:
\begin{equation*}
   e^{\mathcal{U}'} _{*} (|c ^{\rho,
      \mathcal{U}' }|)  =  e ^{\mathcal{U}''}
_{*} (|c ^{\rho, \mathcal{U}''}|).  
\end{equation*}
So \begin{equation*}
   e^{\mathcal{U}} _{*} (|c ^{\rho,
      \mathcal{U} }|) =  e^{\mathcal{U}'} _{*} (|c ^{\rho,
      \mathcal{U}' }|),
\end{equation*}
and so we are done.

\end{proof}
We
call $c ^{\rho} \in H ^{2k} (BG ^{top}, \mathbb{R}
^{} ) $ \textbf{\emph{the universal Chern-Weil characteristic 
class associated to $\rho$}}.

\subsection{Universal cohomological Chern-Weil homomorphism} \label{sec_Universal cohomological Chern-Weil homomorphism}
Let $$hcw: \mathcal{I} (G)  \to H ^{*} (BG
   ^{top}, \mathbb{R} ),
  $$
be the algebra map sending $\rho$ to $c ^{\rho}$ as above.
Then to summarize, we have the following theorem
purely about the Milnor classifying space $BG
^{top}$, reformulating Theorem \ref{thm:A} of
the introduction: 
\begin{theorem} \label{thm:universalChernWeilclass} Let $G$
be a generalized Lie group.  The homomorphism $hcw$
satisfies the following property.
Let $G \hookrightarrow Z \to Y$ be a smooth principal
$G$-bundle.  Let $c ^{\rho} (Z) \in H ^{2k} (Y) $
denote the standard Chern-Weil class associated
to $\rho$. Then $$f _{Z}
^{*} hcw (\rho)  = c^{\rho}  (Z),
$$  where $f _{Z}: Y \to BG ^{top}$  is the
classifying map of the underlying topological
$G$-bundle.
\end{theorem} 
\begin{proof} 
Let $\mathcal{U} _{0} \ni Z$  be a $G$-admissible 
Grothendieck universe. By Proposition
\ref{prop:abstractChernWeil} $$ c ^{\rho} (Z ^{\Delta^{} }) = f _{Z ^{\Delta^{} }} ^{*} (c ^{\rho,
\mathcal{U} _{0} } ). $$  And by 
   Proposition \ref{prop:expected},  $|c ^{\rho} (Z ^{\Delta^{} })| _{sm} = c ^{\rho} 
(Z)$. So we have 
\begin{align*}
   c ^{\rho} (Z) & = |c ^{\rho} (Z ^{\Delta^{} })| _{sm}  \\
   & = |f _{Z ^{\Delta^{} }} ^{*} (c ^{\rho,
   \mathcal{U} _{0}})| _{sm}  \\
	 & = N ^{*}  (|f _{Z ^{\Delta^{} }} ^{*} c ^{\rho,
   \mathcal{U} _{0}}|), \text{ Part
	 \ref{notation:geomRealizationClasspart2} of Notation
	 \ref{notation:geomRealizationClass}}  \\ 
	 & = N ^{*} \circ |f _{Z ^{\Delta^{} }}| ^{*} (|c ^{\rho,
   \mathcal{U} _{0}}|), \text{ by  Lemma
   \ref{lemma:naturalrealization}} \\
   & =  N ^{*} \circ |f _{Z ^{\Delta^{} }}| ^{*} \circ (e
   ^{\mathcal{U} _{0}}) ^{*} c ^{\rho}, \text{ by
   definition of $c
^{\rho}$}.  
\end{align*}
Now, we have a diagram of topological $G$-bundle maps:
\begin{equation*}
\begin{tikzcd}
\vert Z ^{\Delta^{} } \vert \ar[r, ""] \ar [d, ""] &  Z \ar
[d,""] \ar [r] & EG \ar [d, ""] \\
 \vert Y _{\bullet } \vert \ar [r, "h"]   &  Y \ar [r,
 "f _{Z}"] & BG ^{top},
\end{tikzcd}
\end{equation*}
for $h$ as in \eqref{equation_h}.
And so $e ^{\mathcal{U} } \circ f _{|Z ^{\Delta^{} }|}$
being the classifying map for $|Z ^{\Delta^{} }| \to |Y
_{\bullet}|$, is homotopic to $f _{Z} \circ h$.
\begin{equation*}
\end{equation*}

Thus, $e ^{\mathcal{U} _{0}} \circ |f _{Z ^{\Delta^{} }}| \circ N $ is homotopic to $f _{Z}$.
So that  
\begin{equation*}
  c ^{\rho} 
  (Z) = f _{Z} ^{*} c ^{\rho} = f _{Z} ^{*} hcw ({\rho}),  
\end{equation*}
and we are done. 
\end{proof}
\subsection{Universal $dg$ Chern-Weil homomorphism}
\label{sec_Universal dg Chern-Weil homomorphism proof}
\begin{proof} [Proof of Theorem \ref{thm_ChernWeildgIntro}]
Let $cw ^{D}: \mathcal{I} (G) \to \Omega ^{\bullet } (BG
^{\mathcal{U}})$ be as in the preamble of
Section \ref{sec:universalChernWeil}. The first part of
the theorem readily follows by Lemma
\ref{lemma:equality:chernweilclass}. 

Now, let $P \to Y$ be a smooth,
$\mathcal{U} $-small $G$-bundle
over a smooth manifold, and set $X = Y _{\bullet}$.
Let $P ^{\Delta^{} } \to X$ denote the induced $\mathcal{U}
$-small simplicial
$G$-bundle, and $f _{P ^{\Delta^{} }}$ its classifying map.
We need to show that $cw$ satisfies the
naturality condition:
$$\Theta  \circ cw ^{P} \simeq f _{P ^{\Delta^{} }} ^{*} \circ cw. $$

By Proposition \ref{prop:abstractChernWeil} we have:
\begin{equation*}
f _{P ^{\Delta^{} }} ^{*} \circ cw \simeq cw ^{P ^{\Delta^{}
}}.
\end{equation*}
And by Part 1 of Proposition \ref{prop:expected}  $cw ^{P
^{\Delta^{}}} \simeq \Theta  \circ cw ^{P} $. And so we are done.
\end{proof}

\subsection{Relation with Whitney-Sullivan de Rham algebra}
\label{sec_Whitney-Sullivan de Rham algebra of a topological space} Let $Y$ be a topological
space and $X$ the simplicial set of
singular topological simplices in $Y$.
Set $A (Y) = \Omega ^{\bullet } (X, \mathbb{R}
^{} )$, then this is a commutative $dga$.  

Note that if $X$ is a Kan complex then we have a homotopy
equivalence of simplicial sets $|X| _{\bullet} \simeq X$
(the natural map $h: X \to |X| _{\bullet}$ is a weak
equivalence, but as both sides are Kan complexes it is then
an equivalence.)
It follows that for a Kan complex $X$ we have
a \textbf{\emph{geometric homotopy
equivalence}}  $h_*: A (|X|) \to \Omega ^{\bullet } (X) $. That
is if $h ^{-1}$ denotes a homotopy inverse of $h$, then $h
_{*} \circ h ^{-1} _{*} $ is homotopic to the identity
and $h ^{-1}_{*} \circ h _{*} $ is homotopic to the
identity. Furthermore, such a homotopy inverse for $h _*$ is
unique up to homotopy.
\begin{proof} [Proof of Theorem \ref{thm_APLMain}]
Fix any $G$ admissible $\mathcal{U} $.
As $BG ^{\mathcal{U} }$ is a Kan complex, we have a diagram
with each map a homotopy equivalence of simplicial sets:
$$BG ^{\mathcal{U} } \xrightarrow{h} |BG ^{\mathcal{U} }| _{\bullet}
\xrightarrow{e ^{\mathcal{U} _{\bullet} }} |BG ^{Top}|
_{\bullet}.$$

Let $f: A (BG ^{Top}) \to \Omega ^{\bullet} (BG
^{\mathcal{U} })$ be the $dg$ mapping induced by the
composition above.
By the discussion of the prior paragraph, and since $e ^{\mathcal{U}} $ is also a homotopy equivalence, there is a homotopy inverse $f
^{-1}$, unique up to homotopy. Then set $cw = f ^{-1} \circ
cw ^{\mathcal{U} }$.
\end{proof}
\section{Universal Chern-Weil theory for the group
of Hamiltonian symplectomorphisms} 
\label{sec:Chern-WeiltheoryonBHam}  
Let $(M, \omega) $ be a possibly non-compact 
symplectic manifold of dimension $2n$, so
that $\omega$ is a closed non-degenerate $2$-form
on $M$. Let $ \mathcal{H} = \operatorname {Ham}  (M,\omega) $ denote the group of
its
compactly generated Hamiltonian symplectomorphisms (as in
Section \ref{sec_Generalized Lie groups}), and $\mathfrak h $
its Lie algebra.  

For example, take
$M=\mathbb{CP} ^{n-1} $ with its Fubini-Study
symplectic 2-form $\omega _{st}$. Then the natural action
of $PU (n) $ on $\mathbb{CP} ^{n-1} $ is by
Hamiltonian symplectomorphisms.

In
 ~\cite{cite_ReznikovCharacteristicclassesinsymplectictopology}
 Reznikov constructs multilinear functionals 
 $$\{r _{k}\} _{k \geq 1} \subset \mathcal{I} (\mathcal{H}):$$  
 $$(H _{1}, \ldots, H _{k}) \mapsto \int _{M} H _{1} 
   \cdot \ldots \cdot H _{k} \,\omega ^{n},$$ 
  upon identifying:
  $$\mathfrak h = \begin{cases}
     C ^{\infty} _{0} (M), &\text{ if $M$ is 
     compact} \\
    C ^{\infty}_{c} (M),  &\text{ if $M$ is
     non-compact}.
 \end{cases}
$$
Here $C ^{\infty}_{0} (M)$ denotes the set of 
smooth functions $H$ satisfying $\int _{M} H \,\omega ^{n} =0$,
and $C ^{\infty}_{c} (M)$ denotes the set of 
smooth, compactly supported functions. 
In the case $k=1$, the associated functional 
vanishes whenever $M$ is compact. 

When $M$ is compact the group $\mathcal{H} $ is a Fr\'echet Lie group having the
homotopy type of a countable CW complex. Otherwise, as  mentioned in Section
\ref{sec_Generalized Lie groups}, it is 
a generalized Lie group having the homotopy type of a CW
complex.

Thus,
Theorem \ref{thm:universalChernWeilclass} implies 
the Corollary \ref{thm:RezExtends} of the 
introduction, and in 
particular we get
induced Reznikov cohomology classes  
\begin{equation}
   \label{eq:universalReznikov}
   c ^{r_{k}} \in
H ^{2k} (B \mathcal{H}, \mathbb{R} ^{} ).
\end{equation}
As mentioned, the group $PU (n) $ naturally acts on $\mathbb{CP} ^{n-1} $
by Hamiltonian symplectomorphisms. So we have an
induced map $$i: BPU (n)  \to B \operatorname {Ham}  (\mathbb{CP}
^{n-1}, \omega _{0} ).$$
Then as one application we prove Theorem \ref{thm:B} of the
introduction, reformulated as follows:
\begin{theorem} \label{thm:BPU}  [Originally 
   Kedra-McDuff
	 ~\cite{cite_KedraMcDuffHomotopypropertiesofHamiltoniangroupactions}]  
   \label{thm:ReznikovClassBPU} $$i ^{*}: H ^{k}(B
	 \operatorname {Ham}  (\mathbb{CP}
   ^{n-1}, \omega _{0}), \mathbb{R}   ) \to H
   ^{k}(B PU (n),
   \mathbb{R} ^{} )   $$  is
   surjective for all $n \geq 2$, $k \geq 0$  and so 
   \begin{equation*}
      i _{*}: H _{k}(B PU (n),
   \mathbb{R} ^{} ) \to H _{k}(B \operatorname {Ham} (\mathbb{CP}
      ^{n-1}, \omega _{0}), \mathbb{R}),     
   \end{equation*}
  is injective for all $n \geq 2$, $k \geq 0$. 
\end{theorem}
\begin{proof} 
   Let $\mathfrak g $ denote the Lie
   algebra of $PU (n) $, and $\mathfrak h$ the Lie 
   algebra of $\operatorname {Ham} (\mathbb{CP} ^{n-1}, \omega _{0}) $. 
   Let $j: \mathfrak g  \to \mathfrak h $ denote
   the natural Lie algebra map induced by the
   homomorphism $PU (n)  \to \operatorname {Ham} (\mathbb{CP}
   ^{n-1}, \omega _{0}) $. Reznikov
~\cite{cite_ReznikovCharacteristicclassesinsymplectictopology}
shows 
that $\{j ^{*} r _{k} \} _{k>1}$  are the Chern
polynomials. In other words, the classes $$  
c^{j^{*} r _{k}} \in H ^{2k} (BPU (n),
\mathbb{R} ^{} ),$$
are the Chern classes $\{c _{k} \} _{k>1}$, which
generate real
cohomology of $BPU (n) $, as is well known.
But $c^{j^{*} r _{k}} = i ^{*}
c ^{r _{k}}$, for $c ^{r _{k}}$   as in
\eqref{eq:universalReznikov},
and so the result immediately follows.   
   \end{proof}
In 
Kedra-McDuff~\cite{cite_KedraMcDuffHomotopypropertiesofHamiltoniangroupactions} 
a proof of the above is given via homotopical 
techniques. Theirs is a difficult argument, but 
as they show their technique is also partially 
applicable to study certain 
generalized, homotopical analogues of the group 
$\mathcal{H} $. Our argument is  
elementary (at least given the general theory),   but does not obviously have 
homotopical ramifications as in 
~\cite{cite_KedraMcDuffHomotopypropertiesofHamiltoniangroupactions}. 

In   Savelyev-Shelukhin \cite{cite_SavelyevTwistedKtheory}  there are
a number of results about induced maps in (twisted)
$K$-theory. These further suggest that the map $i$ 
above should be a monomorphism in the homotopy 
category. For a start we may ask:  
\begin{question}
   \label{question:} Is the map $i$ above an
   injection on integral homology?
\end{question}
For this one may need more advanced techniques 
like
~\cite{cite_SavelyevGlobalFukayaCategoryI}.


\section {Universal coupling class for Hamiltonian 
fibrations} \label{sec:couplingClass}
Although we use here some language of symplectic 
geometry no special expertise should be necessary.
As the construction here
is a partial reformulation of our general 
constructions, for the special case of $G= \mathcal{H} =\operatorname {Ham} (M,  \omega) 
$, we will not give exhaustive details.

Let $(M, \omega) $  and $\mathcal{H} $ be as in 
the previous section, (keeping in mind our $M$ is not
assumed to be compact) and let $2n$ be the dimension 
of $M$.
\begin{definition}\label{def:hamfibration}
A \textbf{\emph{Hamiltonian $M$-fibration}} is a 
smooth fiber bundle
 $M \hookrightarrow P \to X$, with structure group $\mathcal{H} $.   
\end{definition}
Each $\mathcal{H} $-connection $\mathcal{A} $ on 
such $P$ uniquely induces a \emph{coupling 2-form} on $P$, as originally appearing in 
\cite{cite_GuilleminLermanEtAlSymplecticfibrationsandmultiplicitydiagrams}.
Specifically, this is a closed 2-form $
C _{\mathcal{A} } $ on $P$ whose restriction to fibers coincides with $\omega $  and
which has the following property. Let $\omega 
_{\mathcal{A}} \in \Omega ^{2} (X) $ denote the 
2-form defined by:  $$ \omega 
_{\mathcal{A} } (v,w) = n\int _{P _{x}} R 
^{\mathcal{A} } (v,w) \,\omega _{x}^{n}, $$ for $v,w \in T _{x} X$. Here $R ^{\mathcal{A} 
}$ as before is the curvature 2-form of 
$\mathcal{A} $, so that 
$$R ^{\mathcal{A}} (v,w) \in \lie \operatorname {Ham} (M _{x},  \omega 
_{x}) = \begin{cases}
   C ^{\infty} _{0} (P _{x}) , &\text{ if $M$ is 
   compact } \\
  C ^{\infty}_{c} (P _{x})  &\text{ if $M$ is 
   non-compact}.
\end{cases} 
$$ Note of course that $\omega _{\mathcal{A}} = 0$ when $M$ is compact.
The characterizing property of $C _{\mathcal{A} }$ is then:
\begin{equation*}  \int _{M}  {C 
^{n+1}_{\mathcal{A} }}   = \omega _{\mathcal{A}}, 
\end{equation*}
where the left-hand side is integration along the fiber. 
\footnote {$C _{\mathcal{A} }$ is not 
generally compactly 
supported but $C ^{n+1} _{\mathcal{A} } $ is,  
which is a consequence of taking $\mathcal{H} $ to 
be compactly supported Hamiltonian 
symplectomorphisms.} 

It can then be  
shown that the cohomology class 
$\mathfrak c (P) $  of $C_{\mathcal{A} }$ is 
uniquely determined by $P$ up to $\mathcal{H} 
$-bundle isomorphism. This is called the 
\textbf{\emph{coupling class of $P$}}, and it has 
important applications in symplectic geometry. See 
for instance 
~\cite{cite_McDuffSalamonIntroductiontosymplectictopology} 
for more details and some applications.

By replacing the category $\mathcal{G} $ with 
other fiber bundle categories we may define other 
kinds of simplicial fibrations over a smooth 
simplicial set. For example,  we may replace $\mathcal{G} 
$ by the category of smooth Hamiltonian 
$M$-fibrations, keeping the other axioms in the 
Definition \ref{def:simplicialGbun} intact. This then gives us the notion of a 
Hamiltonian simplicial $M$-bundle over a smooth 
simplicial set. 

Let $\mathcal{U} $ be a $\mathcal{H} $-admissible Grothendieck 
universe. 
Let $M ^{\mathcal{U}, \mathcal{H}  } $ denote the 
Hamiltonian simplicial $M$-fibration, naturally associated to 
${E} \mathcal{H}  ^{\mathcal{U}} 
\to B \mathcal{H} ^{\mathcal{U} }  $. So that for 
each $k$-simplex $\Sigma 
\in B \mathcal{H} ^{\mathcal{U} }$ we have a 
Hamiltonian $M$-fibration 
$M ^{\mathcal{U}, \mathcal{H}} _{\Sigma} \to 
\Delta^{k} $, which is the associated 
$M$-bundle to the principal $\mathcal{H} $-bundle 
${E} \mathcal{H} ^{\mathcal{U}} _{\Sigma}$. 

Fix a (simplicial)  $\mathcal{H}$-connection  
$\mathcal{A} $ on the universal $\mathcal{H} $-bundle $E \mathcal{H}  ^{\mathcal{U} } \to B \mathcal{H} ^{\mathcal{U} }$.
This induces a (simplicial) connection with the same
name $\mathcal{A} $ on $M ^{\mathcal{U},
\mathcal{H}}$.

By the discussion above, for each $k$-simplex $\Sigma 
\in B \mathcal{H} ^{\mathcal{U} }$ we have the 
associated coupling   
2-form $C _{\mathcal{A}, \Sigma}$ on the Hamiltonian
$M$-bundle $M ^{\mathcal{U}, \mathcal{H}} _{\Sigma} \to
\Delta^{k} $. The 
collection of these 2-forms 
then readily induces a cohomology class $\mathfrak c 
^{\mathcal{U} }$ on the geometric realization:
\begin{equation*}
   |{M} ^{\mathcal{U}, \mathcal{H}} | = \colim _
   {\Sigma \in \Delta (B \mathcal{H} ^{\mathcal{U}}) } M ^{\mathcal{U}, \mathcal{H}} _{\Sigma}.
\end{equation*}
This is analogous to the construction of the class $|\alpha
|$ in Section
\ref{sec:Relation with ordinary homology cohomology}.

Now by the proof of Theorem 
\ref{thm:homotopytypeBGU} we have an $\mathcal{H} 
$-structure preserving, $M$-bundle map  
over the homotopy equivalence $e ^{\mathcal{U} }$: 
$$g ^{\mathcal{U}}: |{M} ^{\mathcal{U}, 
\mathcal{H}} | \to M ^{\mathcal{H} },$$ where ${M} 
^{\mathcal{H}} $   
denotes the universal Hamiltonian $M$-fibration 
over $B \mathcal{H} $. 
And these $g ^{\mathcal{U} }$  are natural, 
so that 
if $\mathcal{U} \ni  \mathcal{U} '$ then 
\begin{equation} \label{eq:naturalityg}
   [g ^{\mathcal{U}'} \circ |\widetilde{i}  ^{\mathcal{U},
   \mathcal{U} '}|] = [g ^{\mathcal{U}}], 
\end{equation}
where $|\widetilde{i}  ^{\mathcal{U},
\mathcal{U}'}|: |{M} ^{\mathcal{U}, \mathcal{H}}| 
\to  |{M} ^{\mathcal{U}, \mathcal{H}'}|
$ is the natural $M$-bundle map over ${i}  ^{\mathcal{U},
\mathcal{U}'}$ (as in Theorem \ref{thm:homotopytypeBGU} ),  and where $[\cdot] $ denotes
the homotopy class.

Each $g ^{\mathcal{U} }$ is a homotopy 
equivalence, so we may set $$\mathfrak c := g 
^{\mathcal{U}} _{*} (\mathfrak c  ^{\mathcal{U} 
}) \in H ^{2} ({M} ^{\mathcal{H}}). $$ 
\begin{lemma}
   \label{lemma:welldefinedc}  The class 
   $\mathfrak c$ is well-defined, (independent of 
   the choice $\mathcal{U}$). 
\end{lemma}
The proof is analogous to the proof of Lemma 
\ref{lemma:independent}.
Given this definition of the universal coupling 
class $\mathfrak c$, the proof of Theorem 
\ref{thm:couplingClass} is analogous to 
the proof of Theorem 
\ref{thm:universalChernWeilclass}.

\begin{appendices}
\section{$A _{\infty}$ homotopies}
\label{appendix_Ainfhomotopy}

We prove here the following:
\begin{proposition} \label{prp_ainftyhomotopy} A geometric
homotopy of $dg$ maps induces an $A _{\infty}$ homotopy.
\end{proposition}
This is relegated to the appendix for the following reasons. It is somewhat
esoteric, and not central to the paper. Moreover, the best reference I have is too
general, working with non-unital $A _{\infty}$
categories and functors rather than unital algebras. In the
restricted setting here, the
result above should be known. Secondly,
the geometric homotopy is a more refined 
notion and should work better for intended future applications like
Cheeger-Simons differential characters. However, $A
_{\infty}$ homotopy notion does have some advantages, as it is purely algebraic and so should be easier to manipulate. 

\begin{proof} [Proof] (Sketch)
We omit specifying the coefficient ring $\mathbb{R} ^{} $ in
what follows. 
Suppose we are given a geometric homotopy:
\begin{equation*}
\widetilde{f}: A \to \Omega
^{\bullet} (X \times I),
\end{equation*}
between $dg$ maps $f _{0}, f _{1}: A \to \Omega ^{\bullet}
(X)$.

Define $\widetilde{e}
_{i}: \Omega ^{\bullet} (X) \otimes \Omega ^{\bullet} (I)
\to  \Omega ^{\bullet} (X)$, $i=0,1$, on generators by:
\begin{equation*}
\widetilde{e} _{i} (\omega _{0} \otimes \omega _{1}
) = \begin{cases}
	0, &\text{ if degree $\omega _{1}>0$} \\
	\omega _{0} \cdot \omega _{1} (i), &\text{ if degree
	$\omega _{1} = 0$}, 
\end{cases}
\end{equation*}
where $\omega _{1} (i)$ is evaluation at $i$.

For $e _{i}$ as in Definition \ref{def_dghomotopy}, we will
factorize: 
\begin{equation*}
\begin{tikzcd}
\Omega ^{\bullet} (X \times I) \ar [r, "e
_{i}"] \ar [d, "u"] & \Omega
^{\bullet} (X) \\
\Omega ^{\bullet} (X) \otimes \Omega ^{\bullet} (I) \ar [ur,
"\widetilde{e} _{i}"], 
\end{tikzcd}
\end{equation*}
where $u$ is a certain $A _{\infty}$ homotopy inverse to the
Kunneth, injective $dg$ quasi-isomorphism $h: \Omega ^{\bullet} (X) \otimes \Omega ^{\bullet}
(I) \to \Omega ^{\bullet} (X \times I)$, $h (a \otimes b) = \pi _{{X}} ^{*}a \wedge \pi _{I} ^{*}b   $. This will readily
imply our claim.

Note first that an
$A _{\infty}$ homotopy inverse exists, see Seidel
~\cite[Corollary 1.14]{cite_SeidelFukayacategoriesandPicard-Lefschetztheory}. But we need $u$ with the specific property
above, and  we'll construct it following Seidel's argument based on
homological perturbation lemma. 

To get this $u$, note first that we have the opposite
factorization:
\begin{equation*}
\begin{tikzcd}
\Omega ^{\bullet} (X \times I) \ar [r, "e
_{i}"] & \Omega
^{\bullet} (X), \\
\Omega ^{\bullet} (X) \otimes \Omega ^{\bullet} (I) \ar [u, "h"] \ar [ur,
"\widetilde{e} _{i}"], 
\end{tikzcd}
\end{equation*}
and the image of $h$ will be denoted by $A$.

We have a splitting of chain complexes:  
\begin{equation} \label{eq_splitChainAB}
\Omega ^{\bullet} (X \times I) \simeq A \oplus B,
\end{equation}
where $B = \Omega ^{\bullet} (X \times I)/A
$ is acyclic. 

Then $u$ is the composition of maps:
\begin{equation*}
\begin{tikzcd}
 \Omega ^{\bullet} (X \times I) \ar[r, ""] & A  \ar [r,""]
  &  \Omega ^{\bullet} (X)
\otimes  \Omega ^{\bullet} (I),  
\end{tikzcd}
\end{equation*}
where the last map is the natural map (the inverse to $h$),
and the first map as a set map is the projection map
with respect to the decomposition \eqref{eq_splitChainAB}.
The latter map can be upgraded to an $A _{\infty}$ map using 
the homological perturbation lemma as described in
~\cite[Proof of Corollary
1.14]{cite_SeidelFukayacategoriesandPicard-Lefschetztheory}. 
\end{proof}

\end{appendices}

\bibliographystyle{siam}  
\bibliography{link.bib} 

\def\cprime{$'$}
\begin{thebibliography}{10}

\bibitem{cite_BottChernWeil}
{\sc R.~Bott}, {\em On the {Chern}-{Weil} homomorphism and the continuous
  cohomology of {Lie}- groups}, Adv. Math., 11 (1973), pp.~289--303.

\bibitem{cite_CheegersSimonsCharacters}
{\sc J.~Cheeger and J.~Simons}, {\em Differential characters and geometric
  invariants},  (1985).

\bibitem{cite_ChenIteratedPathIntegrals}
{\sc K.-T. {Chen}}, {\em {Iterated path integrals}}, {Bull. Am. Math. Soc.}, 83
  (1977), pp.~831--879.

\bibitem{cite_homotopytheorydiffeological}
{\sc J.~D. Christensen and E.~Wu}, {\em The homotopy theory of diffeological
  spaces}, New York J. Math., 20 (2014), pp.~1269--1303.

\bibitem{cite_ChristensenWusmoothclassifying}
\leavevmode\vrule height 2pt depth -1.6pt width 23pt, {\em Smooth classifying
  spaces}, Isr. J. Math., 241 (2021), pp.~911--954.

\bibitem{cite_completeconnectionsonfiberbundles}
{\sc M.~del Hoyo}, {\em Complete connections on fiber bundles}, Indag. Math.,
  New Ser., 27 (2016), pp.~985--990.

\bibitem{cite_DupontCharClasses}
{\sc J.~L. {Dupont}}, {\em {Curvature and characteristic classes}}, vol.~640,
  Springer, Cham, 1978.

\bibitem{cite_GoodCover}
{\sc D.~{Fiorenza}, U.~{Schreiber}, and J.~{Stasheff}}, {\em {\v{C}ech cocycles
  for differential characteristic classes: an \(\infty \)-Lie theoretic
  construction}}, {Adv. Theor. Math. Phys.}, 16 (2012), pp.~149--250.

\bibitem{cite_FreedHopkins}
{\sc D.~S. Freed and M.~J. Hopkins}, {\em Chern-{Weil} forms and abstract
  homotopy theory}, Bull. Am. Math. Soc., New Ser., 50 (2013), pp.~431--468.

\bibitem{cite_GuilleminLermanEtAlSymplecticfibrationsandmultiplicitydiagrams}
{\sc V.~Guillemin, E.~Lerman, and S.~Sternberg}, {\em Symplectic fibrations and
  multiplicity diagrams}, Cambridge University Press, Cambridge, 1996.

\bibitem{cite_IglesiasDiffForms}
{\sc S.~G{\"u}rer and P.~Iglesias-Zemmour}, {\em Differential forms on
  manifolds with boundary and corners}, Indag. Math., New Ser., 30 (2019),
  pp.~920--929.

\bibitem{cite_RichardHamiltonNashMoser}
{\sc R.~S. Hamilton}, {\em The inverse function theorem of {Nash} and {Moser}},
  Bull. Am. Math. Soc., New Ser., 7 (1982), pp.~65--222.

\bibitem{cite_HatcherAlgebraic}
{\sc A.~T. {Hatcher}}, {\em {Algebraic topology.}}, Cambridge: Cambridge
  University Press, 2002.

\bibitem{cite_overflowCWstructureDiff}
{\sc T.~(https://mathoverflow.net/users/54788/tyrone)}, {\em Does the group of
  compactly supported diffeomorphisms have the homotopy type of a cw complex?}
\newblock MathOverflow.
\newblock URL:https://mathoverflow.net/q/463951 (version: 2024-02-12).

\bibitem{cite_NobuyukiMayer-Vietoris}
{\sc N.~Iwase and N.~Izumida}, {\em Mayer-{Vietoris} sequence for
  differentiable/diffeological spaces}, in Algebraic topology and related
  topics. Selected papers based on the presentations at the 7th East Asian
  conference on algebraic topology, Mohali, Punjab, India, December 1--6, 2017,
  Singapore: Birkh{\"a}user, 2019, pp.~123--151.

\bibitem{cite_joyal1999introduction}
{\sc A.~Joyal and M.~Tierney}, {\em An Introduction to Simplicial Homotopy
  Theory}, 1999.

\bibitem{cite_KarshonDomainMaps}
{\sc Y.~Karshon and J.~Watts}, {\em Smooth maps on convex sets}, 794 (2024),
  pp.~97--111.
\newblock AMS-EMS-SMF special session, Université de Grenoble-Alpes, Grenoble,
  France, July 18--20, 2022.

\bibitem{cite_KedraMcDuffHomotopypropertiesofHamiltoniangroupactions}
{\sc J.~K{e}dra and D.~McDuff}, {\em Homotopy properties of {H}amiltonian group
  actions}, Geom. Topol., 9 (2005), pp.~121--162.

\bibitem{cite_bookKobayashiNomizuVol1}
{\sc S.~Kobayashi and K.~Nomizu}, {\em Foundations of differential geometry.
  {I}}, vol.~15 of Intersci. Tracts Pure Appl. Math., Interscience Publishers,
  New York, NY, 1963.

\bibitem{cite_bookKobayashiNomizuVol2}
{\sc S.~Kobayashi and K.~Nomizu}, {\em Foundations of differential geometry.
  {Vol}. {II}. ({Osnovy} differentsial'noj geometrii. {Tom} {II}). {Transl}.
  from the {English} by {L}. {V}. {Sabinin}}, 1981.

\bibitem{cite_KatsuhikoSimplicialCochainAlgebras}
{\sc K.~Kuribayashi}, {\em Simplicial cochain algebras for diffeological
  spaces}, Indag. Math., New Ser., 31 (2020), pp.~934--967.

\bibitem{cite_KatsuhikoOther}
\leavevmode\vrule height 2pt depth -1.6pt width 23pt, {\em A comparison between
  two de {Rham} complexes in diffeology}, Proc. Am. Math. Soc., 149 (2021),
  pp.~4963--4972.

\bibitem{cite_OhTanakaSmoothType}
{\sc H.~Lee~Tanaka and Y.-G. Oh}, {\em {Smooth constructions of
  homotopy-coherent actions}}, To appear in Algebraic and Geometric Topology,
  arXiv:2003.06033.

\bibitem{cite_LurieHighertopostheory}
{\sc J.~Lurie}, {\em {Higher topos theory}}, {Annals of Mathematics Studies
  170. Princeton, NJ: Princeton University Press. }, 2009.

\bibitem{cite_RosenbergChern2}
{\sc Y.~Maeda and S.~Rosenberg}, {\em Traces and characteristic classes in
  infinite dimensions}, in Geometry and analysis on manifolds. In memory of
  Professor Shoshichi Kobayashi, Cham: Birkh{\"a}user/Springer, 2015,
  pp.~413--435.

\bibitem{cite_MagnotChernForms}
{\sc J.-P. Magnot}, {\em Chern forms on mapping spaces}, Acta Appl. Math., 91
  (2006), pp.~67--95.

\bibitem{cite_MagnotNonLinearGrassmanians}
\leavevmode\vrule height 2pt depth -1.6pt width 23pt, {\em On
  {{\(\mathrm{Diff}(M)\)}}-pseudo-differential operators and the geometry of
  non linear {Grassmannians}}, Mathematics, 4 (2016), p.~27.
\newblock Id/No 1.

\bibitem{cite_MagnotNotRegular}
\leavevmode\vrule height 2pt depth -1.6pt width 23pt, {\em The group of
  diffeomorphisms of a non-compact manifold is not regular}, Demonstr. Math.,
  51 (2018), pp.~8--16.

\bibitem{cite_MagnoWattsDiffeology}
{\sc J.-P. {Magnot} and J.~{Watts}}, {\em {The diffeology of Milnor's
  classifying space}}, {Topology Appl.}, 232 (2017), pp.~189--213.

\bibitem{cite_McDuffSalamonIntroductiontosymplectictopology}
{\sc D.~McDuff and D.~Salamon}, {\em Introduction to symplectic topology},
  Oxford Math. Monographs, The Clarendon Oxford University Press, New York,
  second~ed., 1998.

\bibitem{cite_MickelssonChernWeil}
{\sc J.~Mickelsson and S.~Paycha}, {\em Renormalised {Chern}-{Weil} forms
  associated with families of {Dirac} operators}, J. Geom. Phys., 57 (2007),
  pp.~1789--1814.

\bibitem{cite_Milnoruniversalbundles}
{\sc J.~W. {Milnor}}, {\em {Construction of universal bundles. II.}}, {Ann.
  Math. (2)}, 63 (1956), pp.~430--436.

\bibitem{cite_MullerSmoothVsContinousBundles}
{\sc C.~{M\"uller} and C.~{Wockel}}, {\em {Equivalences of smooth and
  continuous principal bundles with infinite-dimensional structure group}},
  {Adv. Geom.}, 9 (2009), pp.~605--626.

\bibitem{cite_Neeb}
{\sc K.-H. Neeb}, {\em Towards a lie theory for locally convex groups},
  Japanese journal of mathematics, 1 (2006).

\bibitem{cite_RosenbergLoopSpacesChern}
{\sc S.~Paycha and S.~Rosenberg}, {\em Traces and characteristic classes on
  loop spaces}, in Infinite dimensional groups and manifolds. Based on the 70th
  meeting of theoretical physicists and mathematicians at IRMA, Strasbourg,
  France, May 2004., Berlin: de Gruyter, 2004, pp.~185--212.

\bibitem{cite_PressleySegalLoopgroups}
{\sc A.~Pressley and G.~Segal}, {\em Loop groups}, Oxford Mathematical
  Monographs,  (1986).

\bibitem{cite_ReznikovCharacteristicclassesinsymplectictopology}
{\sc A.~G. Reznikov}, {\em Characteristic classes in symplectic topology},
  Selecta Math. $($N.S.$)$, 3 (1997), pp.~601--642.
\newblock \ \ Appendix D by L\,Katzarkov.

\bibitem{cite_RosenbergChernWeil}
{\sc S.~Rosenberg}, {\em Chern-{Weil} theory for certain infinite-dimensional
  {Lie} groups}, in Lie groups: structure, actions, and representations. In
  honor of Joseph A. Wolf on the occasion of his 75th birthday, New York, NY:
  Birkh{\"a}user/Springer, 2013, pp.~355--380.

\bibitem{cite_RudinMetricParacompact}
{\sc M.~E. Rudin}, {\em A new proof that metric spaces are paracompact}, Proc.
  Am. Math. Soc., 20 (1969), p.~603.

\bibitem{cite_SavelyevGlobalFukayaCategoryI}
{\sc Y.~Savelyev}, {\em Global {Fukaya} category {I}}, Int. Math. Res. Not.,
  2023 (2023), pp.~18302--18386.

\bibitem{cite_SavelyevTwistedKtheory}
{\sc Y.~{Savelyev} and E.~{Shelukhin}}, {\em {\(K\)-theoretic invariants of
  Hamiltonian fibrations}}, {J. Symplectic Geom.}, 18 (2020), pp.~251--289.

\bibitem{cite_SegalCategoriesAndCohomology}
{\sc G.~Segal}, {\em Categories and cohomology theories}, Topology, 13 (1974),
  pp.~293--312.

\bibitem{cite_SeidelFukayacategoriesandPicard-Lefschetztheory}
{\sc P.~Seidel}, {\em {Fukaya categories and Picard-Lefschetz theory}}, {Zurich
  Lectures in Advanced Mathematics. Z\"urich: European Mathematical Society
  (EMS). vi, 326~p. EUR~46.00 }, 2008.

\bibitem{cite_SouriauDiffeology}
{\sc J.~M. Souriau}, {\em Groupes differentiels}, Springer Berlin Heidelberg,
  Berlin, Heidelberg, 1980.

\bibitem{cite_SouriaGroupesDiff}
\leavevmode\vrule height 2pt depth -1.6pt width 23pt, {\em Groupes
  diff{\'e}rentiels}.
\newblock Differential geometrical methods in mathematical physics, {Proc}.
  {Conf}. {Aix}-en-{Provence} and {Salamanca} 1979, {Lect}. {Notes} {Math}.
  836, 91-128 (1980)., 1980.

\bibitem{cite_SullivanInfinitesimalcomputationsintopology}
{\sc D.~Sullivan}, {\em {Infinitesimal computations in topology.}}, Publ.
  Math., Inst. Hautes \'Etud. Sci,  (1977).

\end{thebibliography}
\end{document}